\documentclass[a4paper]{amsart}
\usepackage{amssymb}
\usepackage[all]{xypic}
\usepackage{graphicx}

\newtheorem{theorem}{Theorem}[section]
\newtheorem{lemma}[theorem]{Lemma}
\newtheorem{prop}[theorem]{Proposition}
\newtheorem{cor}[theorem]{Corollary}

\theoremstyle{definition}
\newtheorem{definition}[theorem]{Definition}

\newtheorem{constr}[theorem]{Construction}

\theoremstyle{remark}
\newtheorem{remark}[theorem]{Remark}

\numberwithin{equation}{section}

\newcommand{\C}{\mathbb{C}}
\newcommand{\N}{\mathbb{N}}
\newcommand{\R}{\mathbb{R}}

\DeclareMathOperator{\image}{im}

\DeclareMathOperator{\diag}{diag}
\DeclareMathOperator{\codim}{codim}

\DeclareMathOperator{\id}{Id}
\DeclareMathOperator{\ord}{ord}
\DeclareMathOperator{\sign}{sign}
\DeclareMathOperator{\scal}{scal}
\DeclareMathOperator{\Riem}{Riem}

\title[$S^1$-equivariant bordism]{$S^1$-equivariant bordism,\\ invariant metrics of positive scalar curvature,\\ and rigidity of elliptic genera}

\author{Michael Wiemeler}
\address{Mathematisches Institut\\WWU M\"unster\\Einsteinstrasse 62\\D-48149 M\"unster\\Germany}
\email{wiemelerm@uni-muenster.de}
\thanks{The research for this paper was supported by DFG grant HA 3160/6-1.}



\begin{document}
\begin{abstract}
  We construct geometric generators of the effective $S^1$-equivariant Spin- (and oriented) bordism groups with two inverted.
 We apply this construction to the question of which \(S^1\)-manifolds
 admit invariant metrics of positive scalar curvature.

It turns out that, up to taking connected sums with several copies of the same manifold, the only obstruction to the existence of such a metric is an \(\hat{A}\)-genus of orbit spaces.
This \(\hat{A}\)-genus generalizes a previous definition of Lott for orbit spaces of semi-free \(S^1\)-actions.

As a further application of our results, we give a new proof of the
vanishing of the \(\hat{A}\)-genus of a Spin manifold with
non-trivial $S^1$-action originally proven by Atiyah and Hirzebruch.
 Moreover, based on our computations we can give a bordism-theoretic proof for the rigidity of elliptic genera originally proven by Taubes and Bott--Taubes.
\end{abstract}

\maketitle


\section{Introduction}
\label{sec:intro}

The problem of determining generators of \(S^1\)-equivariant bordism
rings dates back to the 1970s.
The first results were obtained by Uchida \cite{MR0278338}, Ossa
\cite{MR0263093}, Kosniowski and Yahia \cite{MR695647} and Hattori and
Taniguchi \cite{MR0309134}.
Most of these papers deal with oriented or unitary bordism. Moreover
they construct additive generators.

Using these generators bordism-theoretic proofs of the Kosniowski formula and the Atiyah--Singer formula have been given \cite{MR0309134}, \cite{MR0286131}.
The Kosniowski formula expresses the \(T_y\)-genus of a unitary \(S^1\)-manifold in terms of fixed point data \cite{MR0261636}.
The Atiyah--Singer formula expresses the signature of an oriented \(S^1\)-manifold in terms of the signatures of the fixed point components \cite{MR0236952}.
They were originally proved using the Atiyah--Singer \(G\)-signature theorem. 

More recently the problem of finding multiplicative generators for
unitary bordism was studied by Sinha \cite{MR2115453}.

Semi-free \(S^1\)-equivariant Spin-bordism has previously been considered by Borsari \cite{MR882700}.
Her motivation was a question of Witten, who asked if the equivariant indices of certain twisted Dirac operators are constants and suggested to approach this question via equivariant bordism theory \cite[pp. 258-259]{MR830167}.
The indices of these twisted Dirac operators are coefficients in the Laurent expansion of the universal elliptic genus in the \(\hat{A}\)-cusp.
Therefore a positive answer to Witten's question is implied by the rigidity of elliptic genera which was proven by Taubes \cite{MR998662} and Bott--Taubes \cite{MR954493}.
However, contrary to what Witten suggested their proof was not based on equivariant bordism theory but instead used equivariant K-theory, the Lefschetz fixed point formula and some complex analysis.
Later an alternative proof was given by Liu \cite{MR1331972}  using modularity properties of the universal elliptic genus.

It seems that after their proofs appeared nobody carried out the bordism theoretic approach to the rigidity problem.
We pick up this problem and prove the rigidity of elliptic genera via equivariant bordism theory, thus realizing Witten's original plan.
Moreover, we give a bordism-theoretic proof for the vanishing of the \(\hat{A}\)-genus of a Spin-manifold which admits a non-trivial smooth \(S^1\)-action originally proved by Atiyah and Hirzebruch \cite{MR0278334}.

This classical result follows from general existence results for \(S^1\)-invariant metrics of positive scalar curvature (see Theorem \ref{sec:introduction-1} and Theorem \ref{sec:introduction-2}).
Up to a power of \(2\) these results are conclusive, thus finishing a line of thought begun in \cite{berard83:_scalar} and continued in \cite{MR982331}, \cite{MR2376283}, \cite{wiemeler15:_circl}.
We remark that the Atiyah-Hirzebruch vanishing theorem mentioned above and the existence of \(S^1\)-invariant metrics of positive scalar curvature have not been considered as related subjects, until now.

These proofs are based on a construction of additive generators of the \(S^1\)-equivariant Spin- and oriented bordism groups.
These generators are described in the following theorem.
To state it we first have to fix some notations.

  Let \(G=SO\) or \(G=\text{Spin}\). Denote by \(\Omega_n^{G,S^1}\)
  the bordism group of \(n\)-dimensional manifolds with effective
  \(S^1\)-actions and \(G\)-structures on their tangent bundles.
  We do not assume that our \(S^1\)-actions preserve the Spin structures. But by the connectedness of \(S^1\) they always preserve the orientations.

  Moreover, denote by \(\Omega_{\geq 4, n}^{G,S^1}\) similar groups of
  those \(S^1\)-manifolds which satisfy the above conditions and do
  not have fixed point components of codimension two. We also assume in this case that the bordisms between the manifolds do not have codimension-two fixed point components.

We also need the notion of a generalized Bott manifold.
A generalized Bott manifold \(M\) is a manifold of the following
type:
There exists a sequence of fiber bundles
\begin{equation*}
  M=N_l\rightarrow N_{l-1} \rightarrow \dots \rightarrow N_1\rightarrow N_0=\{pt\},
\end{equation*}
such that \(N_0\) is a point and each \(N_i\) is the projectivization
of a Whitney sum of \(n_i+1\) complex line bundles over \(N_{i-1}\).
Then \(M\) has dimension \(2n=2\sum_{i=1}^l n_i\) and admits an effective action of
an \(n\)-dimensional torus \(T\) which has a fixed point.
Hence it is a so-called torus manifold (for details on the
construction of this torus
action see Section~\ref{sec:unitary}).

\begin{theorem}
\label{sec:introduction-3}
    \(\Omega_*^{G,S^1}[\frac{1}{2}]\) and \(\Omega_{\geq 4, *}^{G,S^1}[\frac{1}{2}]\) are generated as modules over \(\Omega^G_*[\frac{1}{2}]\) by manifolds of the following two types:
    \begin{enumerate}
    \item Semi-free \(S^1\)-manifolds, i.e. \(S^1\)-manifolds \(M\), such that all orbits in \(M\) are free orbits or fixed points.
    \item Generalized Bott manifolds \(M\), equipped with a restricted \(S^1\)-action.
    \end{enumerate}
\end{theorem}

The proof of this result is based on techniques first used by
Kosniowski and Yahia \cite{MR695647} in combination with a result of Saihi
\cite{MR1846617}.

Using the above result, we can give a bordism-theoretic proof of the rigidity of elliptic genera.
This gives the following theorem:

\begin{theorem}[{\cite{MR954493}}]
\label{sec:introduction-4}
  Let \(\varphi:\Omega_n^{\text{Spin},S^1}\rightarrow H^{**}(BS^1,\C)=\C[[z]]\) be an equivariant elliptic genus.
  Then \(\varphi(M)\) is constant, as a power series in \(z\), for every effective \(S^1\)-manifold \([M]\in \Omega_n^{\text{Spin},S^1}\).
\end{theorem}

The idea of our proof of this theorem is as follows.
By Theorem \ref{sec:introduction-3} one only has to prove the rigidity of elliptic genera for semi-free \(S^1\)-manifolds and generalized Bott manifolds.
For semi-free \(S^1\)-manifolds this was done by Ochanine \cite{MR970284}.
His proof can be modified in such a way that it gives the rigidity of \(T\)-equivariant elliptic genera of effective \(T\)-manifolds \(M\) such that all fixed point components have minimal codimension \(2\dim T\).
Here \(T\) denotes a torus.
This minimality condition is satisfied for torus manifolds and therefore also for generalized Bott manifolds.
So the theorem follows.

We also apply Theorem~\ref{sec:introduction-3} to the question of which \(S^1\)-manifolds admit
invariant metrics of positive scalar curvature.
It is necessary for this application to consider the bordism groups
\(\Omega_{\geq 4, n}^{G,S^1}\) because the bordism principle which we
will prove to attack this question only works for bordisms which do
not have fixed point components of codimension two (see Theorem~\ref{sec:more-results} and Remark \ref{sec:resolv-sing-9}).

The existence question in the non-equivariant setting was
 finally answered by Gromov and Lawson \cite{MR577131}  for high dimensional simply connected manifolds which do not admit
 Spin structures and by Stolz \cite{MR1189863} for high dimensional simply connected manifolds which admit such a
 structure.

 Their results are summarized by the following theorem.

 \begin{theorem}
\label{sec:introduction}
  Let \(M\) be a simply connected closed manifold of dimension at
  least five. Then the following holds:
  \begin{enumerate}
  \item If \(M\) does not admit a Spin-structure, then \(M\) admits a
    metric of positive scalar curvature.
  \item If \(M\) admits a Spin-structure, then \(M\) admits a metric
    of positive scalar curvature if and only if \(\alpha(M)=0\).
  \end{enumerate}
\end{theorem}

In the above theorem \(\alpha(M)\) denotes the \(\alpha\)-invariant of
\(M\).
It is a KO-theoretic refinement of the \(\hat{A}\)-genus of \(M\) and
an invariant of the Spin-bordism type of \(M\).

The proof of this theorem consists of two steps; one is geometric,
the other is topological.
The geometric step is to show that a manifold \(M\) which is
constructed from another manifold \(N\) by surgery in codimension at
least 3 admits a metric of positive scalar curvature if \(N\) admits
such a metric.
This is the so-called surgery principle.
It has been shown independently by Gromov--Lawson \cite{MR577131} and Schoen--Yau \cite{MR535700}.

From this principle it follows that a manifold of dimension at least
five admits a metric of positive scalar curvature if and only if its
class in a certain bordism group can be represented by a manifold with
such a metric.
This is called the bordism principle.

The final step in the proof of the above theorem is then to find all
bordism classes which can be represented by manifolds which admit a
metric of positive scalar curvature.

The answer to the existence question in the equivariant setting is less clear.

First of all,
the question if there exists an invariant metric of positive scalar curvature on a \(G\)-manifold, \(G\) a compact Lie group acting effectively, has been answered positively
by Lawson and Yau \cite{MR0358841} for the case that the identity component of \(G\) is
non-abelian.
The proof of this result does not use bordism theory or surgery.
It is based on the fact that a homogeneous \(G\)-space admits an invariant metric of positive scalar curvature, which is induced from an bi-invariant metric on \(G\).

If the identity component of \(G\) is abelian, then the answer to
this question is more complicated.
In this case the existence question was first studied by B\'erard Bergery \cite{berard83:_scalar}.

He gave examples of simply
connected manifolds with a non-trivial \(S^1\)-action which admit
metrics of positive scalar curvature, but no \(S^1\)-invariant such
metric.
There are also examples of manifolds which admit \(S^1\)-actions,
but no metric of positive scalar curvature.
Such examples are given by certain homotopy spheres not bounding Spin manifolds \cite{MR0221518}, \cite{MR0380853}, \cite{MR632190}.

B\'erard Bergery also showed that the proofs of the surgery principle carry over to the equivariant setting for actions of any compact Lie group \(G\).
Based on this several authors have tried to adopt the proof of
Theorem~\ref{sec:introduction} to the equivariant situation.

At first a bordism principle was proposed by Rosenberg and Weinberger \cite{MR982331} for
finite cyclic groups \(G\) and actions without fixed point components
of codimension two.
The proof of this theorem is potentially problematic, because its proof needs more
assumptions than those which are stated in the theorem (see the discussion following Corollary 16 in \cite{MR2376283}).
Based on Rosenberg's and Weinberger's theorem Farsi \cite{MR1152317} studied Spin-manifolds of dimension less
than eight with actions of cyclic groups of odd order.

Later Hanke \cite{MR2376283} proved a bordism principle for actions of any compact
Lie group \(G\) which also takes codimension-two singular strata into
account.
He used this result to prove the existence of invariant metrics of
positive scalar curvature on certain non-Spin \(S^1\)-manifolds which
do not have fixed points and satisfy Condition C (see Definition~\ref{sec:resolv-sing-8}).

In \cite{wiemeler15:_circl} the author showed that every
\(S^1\)-manifold with a fixed point component of codimension two
admits an invariant metric of positive scalar curvature.
In that paper existence results for invariant metrics of positive scalar
curvature on semi-free \(S^1\)-manifolds without fixed point
components of codimension two were also discussed.

Here we extend the results from that paper to certain non-semi-free \(S^1\)-manifolds.
We prove the following existence results for metrics of positive scalar curvature on non-semi-free \(S^1\)-manifolds.

\begin{theorem}
\label{cha:discussion}
  Let \(M\) be a connected effective \(S^1\)-manifold of dimension at least six.
Denote by \(M_{\text{max}}\) the maximal stratum of \(M\), i.e. the union of the principal orbits.

If \(\pi_1(M_\text{max})=0\) and \(M_{\text{max}}\) is not Spin, 
  then for some \(k\geq 0\), the equivariant connected sum of \(2^k\) copies of \(M\) admits an invariant metric of positive scalar curvature.
\end{theorem}

Here an equivariant connected sum of two \(S^1\)-manifolds \(M_{1}\), \(M_{2}\) can be a fiber connected sum at principal orbits, a connected sum at fixed points or more generally the result of a zero-dimensional equivariant surgery on orbits \(O_i\subset M_i\), \(i=1,2\).

Now we turn to a similar result in the case that \(M\) is a Spin manifold.
For this we first note that on Spin manifolds there are two types of actions,
those which lift to actions on the Spin-structure and those which do not lift to the Spin-structure.
The actions of the first type are called actions of even type,
whereas the actions of the second kind are called actions of odd type.

\begin{theorem}
\label{sec:introduction-1}
  Let \(M\) be a Spin \(S^1\)-manifold with \(\dim M\geq 6\), an effective \(S^1\)-action of odd type and \(\pi_1(M_{\text{max}})=0\). Then there is a \(k\in \N\) such that the equivariant connected sum of \(2^k\) copies of \(M\) admits an invariant metric of positive scalar curvature.
\end{theorem}

\begin{theorem}
\label{sec:introduction-2}
  There is an equivariant bordism invariant \(\hat{A}_{S^1}\) with
  values in \(\mathbb{Z}[\frac{1}{2}]\), such that, for a Spin
  \(S^1\)-manifold \(M\) with \(\dim M\geq 6\), an effective \(S^1\)-action
  of even type and \(\pi_{1}(M_{\text{max}})=0\), the following
  conditions are
  equivalent:
  \begin{enumerate}
  \item \(\hat{A}_{S^1}(M)=0\).
  \item There is a \(k\in \N\) such that the equivariant connected sum of \(2^k\) copies of \(M\) admits an invariant metric of positive scalar curvature.
  \end{enumerate}
\end{theorem}

For free \(S^1\)-manifolds \(M\), \(\hat{A}_{S^1}(M)\) is equal to the \(\hat{A}\)-genus of the orbit space of \(M\).
If the action on \(M\) is semi-free, then it coincides with a generalized \(\hat{A}\)-genus of the orbit space defined by Lott \cite{MR1758446}.
In general we can identify \(\hat{A}_{S^1}\) with the index of a Dirac operator defined on a submanifold of \(M_{\text{max}}/S^1\)  with boundary.

\(\hat{A}_{S^1}(M)\) can only be non-trivial if the dimension of \(M\) is \(4k+1\). Moreover, the usual \(\hat{A}\)-genus of \(M\) is zero in these dimensions. It also vanishes if \(M\) admits a metric of positive scalar curvature.
Therefore Theorems~\ref{sec:introduction-1} and \ref{sec:introduction-2} imply the following theorem which 
 was originally proved by Atiyah and Hirzebruch \cite{MR0278334} using the Lefschetz fixed point formula and some complex analysis.

 \begin{theorem}
   Let \(M\) be a Spin-manifold with a non-trivial \(S^1\)-action. Then \(\hat{A}(M)=0\).
 \end{theorem}

The proofs of the above results are based on a generalization (Theorem
\ref{sec:more-results}) of the bordism principle proved by Hanke
\cite{MR2376283} to \(S^1\)-manifolds with fixed points.
When this has been established the theorems follow from the fact that generalized Bott manifolds admit
invariant metrics of positive scalar curvature and the existence
results from \cite{wiemeler15:_circl} for semi-free \(S^1\)-manifolds.

This paper is organized as follows.
In Sections \ref{sec:unitary} and \ref{sec:Spin_case} we prove Theorem~\ref{sec:introduction-3} for manifolds satisfying Condition C.
Then in Section \ref{sec:not_con_c} we generalize these results to manifolds not satisfying Condition C. This completes the proof of Theorem~\ref{sec:introduction-3}.

We next turn to the existence question for invariant metrics of positive scalar curvature on \(S^1\)-manifolds.
In Section~\ref{sec:resolv} we generalize the bordism principle of Hanke \cite{MR2376283} to \(S^1\)-manifolds with fixed points.
Then in Section~\ref{sec:loc_sym_are_gen} we show that, under mild assumptions on the isotropy groups of the codimension-two singular strata, normally symmetric metrics are dense in all invariant metrics on an \(S^1\)-manifold with respect to the \(C^2\)-topology.
Here normally symmetric metrics are metrics which are invariant under certain extra \(S^1\)-symmetries which are defined on small neighborhoods of the codimension-two singular strata.
In Section~\ref{sec:obstruction} we introduce our obstruction \(\hat{A}_{S^1}\) to invariant metrics of positive scalar curvature on Spin \(S^1\)-manifolds with actions of even type.
Then in Section~\ref{sec:atiyah_hirzebruch} we complete the proof of our existence results for metrics of positive scalar curvature on \(S^1\)-manifolds.
Moreover, we give a new proof of the above mentioned result of Atiyah and Hirzebruch using our results.

In the last Section~\ref{sec:rigidity} we give a proof of the rigidity of elliptic genera based on our Theorem~\ref{sec:introduction-3}.

I would like to thank Bernhard Hanke and Anand Dessai for helpful discussions on the subject of this paper.
I also want to thank Peter Landweber, Martin Kerin and the anonymous referee for detailed comments which helped to improve the presentation of the paper.

\section{Non-semi-free actions on non-Spin manifolds}
\label{sec:unitary}

In this section we prove a version of Theorem~\ref{sec:introduction-3} for the \(S^1\)-equivariant oriented bordism groups of manifolds which satisfy Condition C.
At first we recall Condition C.

\begin{definition}
\label{sec:resolv-sing-8}
  Let \(T\) be a torus and \(M\) a compact \(T\)-manifold. We say that \(M\) satisfies Condition C if for each closed subgroup \(H\subset T\), the \(T\)-equivariant normal bundle \(N(M^H,M)\) of the closed submanifold \(M^H\subset M\) is equipped with the structure of a complex \(T\)-vector bundle such that the following compatibility condition holds: 
If \(K\subset H\subset T\) are two closed subgroups, then the restriction of \(N(M^K,M)\) to \(M^H\) is a direct summand of \(N(M^H,M)\) as a complex \(T\)-vector bundle.
\end{definition}

Before we state and prove our main result of this section, we introduce some notation from \cite{MR695647}.

Let \(M\) be a \(S^1\)-manifold satisfying Condition C and \(x\in M\).
Then the isotropy group \(S^1_x\) of \(x\) acts linearly on the tangent space of \(M\) at \(x\).
There is an isomorphism of \(S^1_x\)-representations \(T_xM\cong \bar{V}_x\oplus V_x\), where \(\bar{V}_x\) is a trivial \(S^1_x\)-representation and \(V_x\) is a unitary \(S^1_x\)-representation without trivial summands.
The slice type of \(x\in M\) is defined to be the pair \([S^1_x;V_x]\).

A \(S^1\)-slice-type is a pair \([H;W]\), where \(H\) is a closed subgroup of \(S^1\) and \(W\) is a unitary \(H\)-representation without trivial summands.
We call a slice-type \([H;W]\) effective (semi-free, resp.) if the \(S^1\)-action on \(S^1\times_H W\) is effective (semi-free, resp.). 
By a family \(\mathcal{F}\) of slice-types we mean a collection of \(S^1\)-slice types such that if \([H;W]\in\mathcal{F}\) then for every \(x\in S^1\times_H W\), \([S^1_x,V_x]\in\mathcal{F}\).

A \(S^1\)-manifold \(M\) satisfying Condition C is of type \(\mathcal{F}\), if for every \(x\in M\), \([S^1_x,V_x]\in \mathcal{F}\).
We denote by \(\Omega_*^{C,S^1}[\mathcal{F}]\) the bordism groups of all oriented \(S^1\)-manifolds satisfying Condition C of type \(\mathcal{F}\).
We set  \(\Omega_*^{C,S^1}=\Omega_*^{C,S^1}[\text{All}]\), where \(\text{All}\) denotes the family of all slice-types.
We denote by \(\mathcal{AE}\) the family of all effective slice-types.

Let \(\rho=[H;W]\) be a \(S^1\)-slice type.
A complex \(S^1\)-vector bundle \(E\) over an \(S^1\)-manifold \(N\) is called of type \(\rho\) if the set of points in \(E\) having slice type \(\rho\) is precisely the zero section of \(E\).
Bordism of bundles of type \(\rho\) leads to a bundle bordism group \(\Omega_*^{C,S^1}[\rho]\).

If \(\mathcal{F}=\mathcal{F}'\cup \{\rho\}\) is a family of slice-types such that \(\mathcal{F}'\) is a family of slice types,
then \(\Omega_*^{C,S^1}[\rho]\) is isomorphic to the relative equivariant bordism group \(\Omega_*^{C,S^1}[\mathcal{F},\mathcal{F}']\), which consists of those \(S^1\)-manifolds with boundary of type \(\mathcal{F}\) whose boundary is of type \(\mathcal{F}'\).
Moreover, there is a long exact sequence of \(\Omega_*^{SO}\)-modules
\begin{equation*}
\xymatrix{
  \dots\ar[r]&\Omega^{C,S^1}_n[\mathcal{F}']\ar[r]&\Omega_n^{C,S^1}[\mathcal{F}]\ar^{\nu}[r]&\Omega_n^{C,S^1}[\rho]\ar^{\partial}[r]& \Omega^{C,S^1}_{n-1}[\mathcal{F}']\ar[r]& \dots.}
\end{equation*}
Here \(\nu\) is the map which sends a \(S^1\)-manifold satisfying Condition C of type \(\mathcal{F}\) to the normal bundle of the submanifold of points of type \(\rho\).
Moreover, \(\partial\) assigns to a bundle of type \(\rho\) its sphere bundle.

This sequence provides an inductive method of calculating the bordism groups of \(S^1\)-manifolds satisfying Condition C of type \(\mathcal{F}\).

\begin{theorem}
\label{sec:non-semi-free}
  Let \(\mathcal{F}=\mathcal{AE}\) or  \(\mathcal{F}=\mathcal{AE}-\{[S^1;W];\dim_\C W=1\}\) be the family of all effective \(S^1\)-slice types with or without, respectively, the slice types of the form \([S^1;W]\) with \(W\) an unitary \(S^1\)-representation of dimension one.
  Then \(\Omega_*^{C,S^1}[\mathcal{F}]\) is generated by semi-free
  \(S^1\)-manifolds and generalized Bott manifolds \(M\) with a
  restricted \(S^1\)-action.
\end{theorem}

As we will see shortly a generalized Bott manifold of dimension \(2n\) has an effective action of an \(n\)-dimensional torus \(T\) and the \(S^1\)-actions in the above theorem are restrictions of this \(T\)-action to suitable circle subgroups of \(T\).
In both cases these \(S^1\)-actions on generalized Bott manifolds can be chosen in such a way that they do not have fixed point components of codimension two.

Next we describe the definition of generalized Bott manifolds and the torus actions on them in more detail.
Note that the name ``Bott manifold'' was coined in \cite{MR1301185}.
The more general definition of a ``generalized Bott manifold'' was first given in \cite{MR2551516}.

A generalized Bott manifold \(M\) is a manifold of the following
type:
There exists a sequence of fiber bundles
\begin{equation*}
  M=N_l\rightarrow N_{l-1} \rightarrow \dots \rightarrow N_1\rightarrow N_0=\{pt\},
\end{equation*}
such that \(N_0\) is a point and each \(N_i\) is the projectivization
of a Whitney sum of \(n_i+1\) complex line bundles over \(N_{i-1}\).

The torus action on these manifolds can be constructed inductively as
follows.
At first note that each \(N_{i-1}\) is simply connected.
Therefore, if \(N_{i-1}\) has an effective action of a \((\sum_{j=1}^{i-1}
n_j)\)-dimensional torus \(T\), then the natural map \(H^2_T(N_{i-1};\mathbb{Z})\rightarrow H^2(N_{i-1};\mathbb{Z})\) from equivariant to ordinary cohomology is surjective.
Hence, by \cite{MR0461538}, the \(T\)-action
lifts to an action on each of the \(n_i+1\) line bundles from which
\(N_i\) is constructed.
Together with the action of an \((n_i+1)\)-dimensional torus given by
componentwise multiplication on each of these line bundles, this
action induces an effective action of a \((\sum_{j=1}^{i}n_j)\)-dimensional torus
\(T'\) on \(N_i\).
Note that, by \cite{MR0461538}, the \(T'\)-action constructed in this way is unique up to automorphisms of \(T'\), i.e. does not depend on the actual choice of the lifts of the \(T\)-actions.
So each \(N_i\) becomes a so-called torus manifold with this torus
action, i.e. the dimension of the acting torus is half of the
dimension of the manifold and there are fixed points.

We also remark here that generalized Bott manifolds admit torus invariant metrics of positive scalar curvature (and even of non-negative sectional curvature \cite{MR3355120}).
Indeed, since the (\(U(n+1)\)-invariant) standard metric on \(\mathbb{C} P^n\) has positive curvature, metrics of positive scalar curvature on generalized Bott manifolds are given by certain connection metrics.

\begin{proof}[Proof of Theorem \ref{sec:non-semi-free}]
  At first 
  we define a sequence of families \(\mathcal{F}_i\) of \(S^1\)-slice types such that
  \begin{enumerate}
  \item \(\mathcal{F}_{i+1}=\mathcal{F}_i\cup \{\sigma_{i+1}\}\) for a slice type \(\sigma_{i+1}\),
  \item \(\mathcal{F}=\bigcup_{i=0}^{\infty}\mathcal{F}_i\),
  \item \(\mathcal{F}_0\) consists of semi-free \(S^1\)-slice types.
  \end{enumerate}
When this is done it is sufficient to prove that each
\(\Omega_*^{C,S^1}[\mathcal{F}_i]\) is generated by semi-free
\(S^1\)-manifolds, generalized Bott manifolds and manifolds which bound in \(\Omega_*^{C,S^1}[\mathcal{F}]\).

Before we do that we introduce some notation for the \(S^1\)-slice types.
The irreducible non-trivial \(S^1\)-representations are denoted by
\begin{equation*}
  \dots,V_{-1},V_1,V_2,\dots,
\end{equation*}
where \(V_i\) is \(\C\) equipped with the \(S^1\)-action given by multiplication with \(s^i\) for \(s\in S^1\).
For \(\mathbb{Z}_m\subset S^1\) we denote the irreducible non-trivial \(\mathbb{Z}_m\)-representations by \(V_{-1},\dots,V_{-m+1}\).
Here \(V_i\) is \(\C\) with \(s\in \mathbb{Z}_m\) acting by multiplication with \(s^i\).

The effective \(S^1\)-slice types are then of the form
\begin{equation*}
  [S^1;V_{k(1)}\oplus V_{k(2)}\oplus\dots\oplus V_{k(n)}]
\end{equation*}
with \(k(1)\geq k(2)\geq\dots\geq k(n)\), \(n\geq 1\) and \(\gcd\{k(1),\dots, k(n)\}=1\) or of the form
\begin{equation*}
  [\mathbb{Z}_m;V_{k(1)}\oplus V_{k(2)}\oplus\dots\oplus V_{k(n)}]
\end{equation*}
with \(0>k(1)\geq k(2)\geq\dots\geq k(n)>-m\), \(n\geq 1\) if \(m>1\), and \(\gcd\{k(1),\dots, k(n),m\}=1\).
Such a slice type is semi-free if it is of the form \([S^1;V_{k(1)}\oplus V_{k(2)}\dots\oplus V_{k(n)}]\) with \(|k(i)|=1\) for all \(i\) or \([\mathbb{Z}_1;0]\).

The non-semi-free effective slice types fall into three different classes:
\begin{align*}
  \mathfrak{F}&=\{[H;W];\;H \text{ is finite}\}\\
  \mathfrak{SF}&=\{[S^1;W];\; W= V_{k(1)}\oplus\dots\oplus V_{k(n)}\oplus V_{-m} \text{ with }\\ &\;\;\; 0>k(1)\geq k(2)\geq\dots\geq k(n)>-m\}\\
  \mathfrak{T}&=\{\text{all other non-semi-free effective slice types}\}
\end{align*}
For \(\rho=[\mathbb{Z}_m; V_{k(1)}\oplus\dots\oplus V_{k(n)}]\in\mathfrak{F}\) we define \(e(\rho)= [S^1;V_{k(1)}\oplus\dots\oplus V_{k(n)}\oplus V_{-m}]\in \mathfrak{SF}\).

We define an ordering of the non-semi-free slice types as follows, similarly to the ordering in \cite[Section 6]{MR695647}:
Let
\begin{align*}
  \delta[S^1;V_{k(1)}\oplus\dots\oplus V_{k(n)}]&=\max \{|k(1)|,\dots,|k(n)|\}\\
  d[S^1;V_{k(1)}\oplus\dots\oplus V_{k(n)}]&=n,
\end{align*}
we order \(\mathfrak{T}\cup \mathfrak{SF}\) at first by \(\delta +d\), then by \(d\) and then lexicographically.
This together with the following conditions gives an ordering on  \(\mathfrak{T}\cup \mathfrak{SF}\cup \mathfrak{F}\).
If \(\rho\in \mathfrak{F}\) and \(\rho'\not\in \mathfrak{F}\) then define \(\rho<\rho'\) if \(e(\rho)\leq \rho'\).
If \(\rho\in \mathfrak{F}\) and \(\rho'\in \mathfrak{F}\)  then let \(\rho\leq\rho'\) if \(e(\rho)\leq e(\rho')\).
Note that with this ordering \(e(\rho)\) directly follows \(\rho\).

Denote the elements of \(\mathfrak{T}\cup \mathfrak{SF}\cup \mathfrak{F}\) by \(\sigma_1,\sigma_2,\dots\) so that \(\sigma_i<\sigma_j\) if \(i<j\).
Then one can check that
\begin{align*}
  \mathcal{F}_0&=\{\text{semi-free slice-types}\}-\{[S^1;V_i];\;|i|=1\},\\
  \mathcal{F}_i&=\mathcal{F}_0\cup\{\sigma_1,\dots,\sigma_i\}\; \text{ for } i\geq 1
\end{align*}
are families of slice types.
The difference between these families of slice types and the families of slice types \(\mathcal{ST}_i\) appearing in \cite{MR695647} is that in our ordering all the semi-free slice types are already contained in \(\mathcal{F}_0\) whereas in \cite{MR695647} they only appear in later families.
We choose this different ordering because we already have generators of bordism groups of semi-free \(S^1\)-actions which suit our needs \cite{wiemeler15:_circl} and we do not want the semi-free slice types to disturb our inductive argument. 

Now we prove by induction on \(i\) that
\(\Omega_*^{C,S^1}[\mathcal{F}_i]\) is generated by semi-free
\(S^1\)-manifolds, generalized Bott manifolds and manifolds which bound in \(\Omega_*^{C,S^1}[\mathcal{F}_{i+1}]\).

We may assume that \(i\geq 1\).
Then we have an exact sequence
\begin{equation*}
\xymatrix{
  \dots\ar[r]&\Omega^{C,S^1}_n[\mathcal{F}_{i-1}]\ar[r]&\Omega_n^{C,S^1}[\mathcal{F}_i]\ar^{\nu_i}[r]&\Omega_n^{C,S^1}[\sigma_i]\ar^{\partial_i}[r]& \Omega^{C,S^1}_{n-1}[\mathcal{F}_{i-1}] \ar[r]&.}
\end{equation*}

For a proof that this sequence is exact (in the case of unoriented bordism) see \cite[Theorem 1.3.2]{MR518871}.

At first assume that \(\sigma_i\in\mathfrak{F}\).
Then one can define a section \(q_i\) to \(\nu_i\) as in \cite[Section 6]{MR695647}.
For \(E\in \Omega_*^{C,S^1}[\sigma_i]\), \(q_i(E)\) is represented by a sphere bundle associated to a unitary \(S^1\)-vector bundle of rank at least two over a manifold with trivial \(S^1\)-action.

Hence \(\Omega_n^{C,S^1}[\mathcal{F}_i]\) is generated by manifolds of type \(\mathcal{F}_{i-1}\) and the \(q_i(E)\), \(E\in \Omega_n^{C,S^1}[\sigma_i]\).
Moreover, the disc bundle associated to \(E\) is bounded by \(q_i(E)\) and is of type \(\mathcal{F}_{i+1}\).
Therefore the claim follows in this case from the induction hypothesis.

Next assume that \(\sigma_i\in\mathfrak{SF}\).
Then \(\sigma_{i-1}\in \mathfrak{F}\) and \(\sigma_i=e(\sigma_{i-1})\).
Hence, as in \cite[Section 5]{MR695647}, one sees that \(\partial_i\) induces an isomorphism \(\Omega_n^{C,S^1}[\sigma_i]\rightarrow \image q_{i-1}\).
Therefore we have an isomorphism \(\Omega_*^{C,S^1}[\mathcal{F}_{i-2}]\rightarrow\Omega_*^{C,S^1}[\mathcal{F}_i]\).
Hence, the claim follows in this case from the induction hypothesis.

Finally, assume that \(\sigma_i\in\mathfrak{T}\).
Then one defines a section \(q_i\) to \(\nu_i\) as in \cite[Section 7]{MR695647}.
For \(E\in \Omega_*^{C,S^1}[\sigma_i]\), \(q_i(E)\) is represented by the projectivization \(P(\tilde{E})\) of a unitary \(S^1\)-vector bundle \(\tilde{E}\) of rank at least two over a manifold \(B\) with trivial \(S^1\)-action or by a manifold \(M_2\) of the following form.
First let \(M_1\) be the projectivization of a unitary \(S^1\)-vector bundle \(\tilde{E}_1\) over a manifold \(B\) with trivial \(S^1\)-action.
\(M_2\) is then the projectivization of a unitary \(S^1\)-vector bundle \(E_2\) of rank at least two over \(M_1\).

The (more complicated) manifolds of the second kind are needed because for some slice types \(\sigma_i\) the first construction does not lead to manifolds of type \(\mathcal{F}_i\). They have points with slice types \(\rho>\sigma_i\).

Moreover, one sees from the definition of \(q_i\) in \cite[Section 7]{MR695647} and the fact that \(\sigma_i\) is not semi-free that \(q_i(E)\) does not have fixed point components of codimension two.

Indeed, for manifolds of the first type, the \(S^1\)-action on \(P(\tilde{E})\) restricts to a linear action on each fiber.
The above claim therefore follows in this case because a linear action on \(\C P^n\) can only have codimension two fixed point components if it is semi-free and by the definition of \(\tilde{E}\) in \cite{MR695647} this only happens if \(\sigma_i\) is a semi-free slice type.

In the second case one can argue similarly. We can restrict the \(S^1\)-action to any fiber over a point in \(B\).
This is a \(\C P^{k_1}\)-bundle over \(\C P^{k_2}\) equipped with a linear \(S^1\)-action.
The fixed point components of such actions are also bundles with fiber  \(\C P^{k_1'}\) over \(\C P^{k_2'}\) with \(k_1'\leq k_1\) and \(k_2'\leq k_2\).
If such a component has codimension two, then we must have (\(k_1'=k_1-1\)  and \(k_2'=k_2\)) or (\(k_1'=k_1\)  and \(k_2'=k_2-1\)).
In the first case the action on the base is trivial and we can argue as in the case of manifolds of the first kind.
In the second case the action on the base must be semi-free and trivial on the fibers over a fixed point component (in the base).
But with the definition of \(M_2\) in \cite{MR695647}, this can only happen if \(\sigma_i\) is semi-free.

Since \(\Omega_*^{SO}(\prod_i BU(j_i))\) is generated as a module over
\(\Omega_*^{SO}\) by sums of line bundles over complex projective spaces, we may assume that the \(q_i(E)\) are generalized Bott manifolds.
Hence the claim follows in this case from the induction hypothesis.
\end{proof}

\begin{remark}
  The definition of the section \(q_i\) in \cite{MR695647} contains some confusing typos: In line 8 of page 103 it must be \(te_0\) instead of \(t^{-1}e_0\). Moreover, in the first remark on that page the base and fiber of the described bundles have to be interchanged.
\end{remark}

\section{The Spin case}
\label{sec:Spin_case}

In this section we prove the following theorem about the \(S^1\)-equivariant Spin-bordism groups of those manifolds which satisfy Condition C.

\begin{theorem}
\label{sec:Spin-case}
  Let \(\mathcal{F}=\mathcal{AE}\) or \(\mathcal{F}=\mathcal{AE}-\{[S^1;W];\dim_\C W=1\}\) be the family of all effective \(S^1\)-slice types with or without, respectively, the slice types of the form \([S^1;W]\), with \(W\) an unitary \(S^1\)-representation of dimension one.
  Then \(\Omega_*^{C,Spin,S^1}[\frac{1}{2}][\mathcal{F}]\) is
  generated by semi-free \(S^1\)-manifolds and generalized
  Bott manifolds with restricted \(S^1\)-action.
\end{theorem}

As in the non-Spin case in Theorem~\ref{sec:non-semi-free} one can choose the circle actions on the generalized Bott manifolds in such a way that they do not have fixed point components of codimension two.

A Spin structure on an oriented manifold \(M\) is here to be understood as a lift \(\hat{f}:M\rightarrow B\text{Spin}\) of the classifying map \(f:M\rightarrow BSO\) of the stable normal bundle of \(M\).
Spin bordisms \(W\)  between Spin manifolds \(M_1\) and \(M_2\) have Spin structures which extend the Spin structures on \(M_1\) and \(M_2\).
This notion of bordism then leads to Spin-bordism groups \(\Omega_*^{Spin}\) and their equivariant analogues \(\Omega_*^{Spin,S^1}\) and \(\Omega_*^{C,Spin,S^1}\).
Here we do not assume that the \(S^1\)-action preserves the Spin structure.

For the proof of Theorem~\ref{sec:Spin-case} we will use the same families of slice types and the following exact sequences analogous to the ones used in the proof of Theorem~\ref{sec:non-semi-free}.

\begin{equation*}
  \xymatrix{
    \dots \ar[r]&\Omega_n^{C,Spin,S^1}[\mathcal{F}_{i-1}]\ar[r]&\Omega_n^{C,Spin,S^1}[\mathcal{F}_i]\ar[d]\\ &&\Omega_n^{C,Spin,S^1}[\sigma_i]\ar[r]
&\Omega_{n-1}^{C,Spin,S^1}[\mathcal{F}_{i-1}]\ar[r]&\dots
}
\end{equation*}

Before we prove Theorem~\ref{sec:Spin-case}, we describe the groups \(\Omega_n^{C,Spin,S^1}[\sigma_i]\).
These groups have been computed by Saihi \cite{MR1846617}. She showed that
\begin{equation*}
  \Omega_n^{C,Spin,S^1}[H,V] \cong \Omega_{n-1-2\sum_{i=1}^k n_i}(B\Gamma_u\amalg B\Gamma_v,f),
\end{equation*}
where \(H\) is a finite group, \(\Gamma_u,\Gamma_v\) are connected two-fold covering groups of \(\Gamma=SO\times S^{1}\times_H\prod_i^{k}U(n_i)\) and \(f=f_u\amalg f_v\) where \(f_u: B\Gamma_u\rightarrow BSO,f_v: B\Gamma_v\rightarrow BSO\) are the fibrations induced by the projection \(\Gamma\rightarrow SO\).
Moreover, for a fibration \(p:B\rightarrow BSO\) we denote by \(\Omega_{*}(B,p)\) the bordism groups of manifolds with \(B\)-structure; i.e. we consider bordisms of oriented manifolds equipped with lifts of the classifying map of their stable normal bundles to \(B\).

Here the \(U(n_i)\) factors in \(\Gamma\) are due to the fact that if \(E\) is a bundle of type \([H,V]\) then \(E\) splits into a direct sum of weight bundles which correspond to the weight spaces of \(V\).
Moreover, the \(SO\)-factor corresponds to the fact that the base space carries a natural orientation induced from the orientations on the fibers and the total space of \(E\).
As the following lemma shows, the fact that we have to deal with two components is due to the fact that the Spin structure on \(S^1\times_HV\) is not unique.

\begin{lemma}
\label{sec:Spin-case-4}
The two components of \(B\Gamma_u\amalg B\Gamma_v\) are in one-to-one correspondence with the two Spin-structures \(\mathcal{U},\mathcal{V}\)  on \(S^1\times_H V\), in such a way that for \(E\in \Omega_*(B\Gamma_u\amalg B\Gamma_v,f_u\amalg f_v)\) we have \(E\in \Omega_*(B\Gamma_u,f_u)\) (\(E\in \Omega_*(B\Gamma_v,f_v)\), respectively) if and only if the restriction of the Spin structure on \(E\) to an invariant tubular neighborhood of an orbit of type \([H,V]\) in \(E\) coincides with the Spin structure \(\mathcal{U}\) (\(\mathcal{V}\), respectively).  
\end{lemma}
\begin{proof}
  The Spin structures on \(S^1\times_H V\times \R^k\) are in one to
  one-to-one correspondence to the Spin structures on \(S^1\times_H V\). 
  Every Spin structure on \(S^1\times_H V\) induces a
  homotopy class of lifts of the map \(pt\rightarrow B\Gamma\) to \(B\Gamma_u\amalg B\Gamma_v\).
  Moreover, every such lift induces a Spin structure on \(S^1\times_H V\).
  Since the fiber of \(B\Gamma_u\amalg B\Gamma_v\rightarrow B\Gamma\) has two components, there are exactly two homotopy classes of lifts.
  These lifts induce different Spin structures on \(S^1\times_H V\) because they represent different elements in \(\Omega_*(B\Gamma_u\amalg B\Gamma_v,f_u\amalg f_v)\).
Since there are exactly two Spin structures on \(S^1\times_H V\) the claim follows.
\end{proof}

After inverting two, we get the following isomorphism:
\begin{align*}
  \Omega_n^{C,Spin,S^1}[H,V][\frac{1}{2}] &\cong \Omega_{n-1-2\sum_{i=1}^kn_i}[\frac{1}{2}](B\Gamma \amalg B\Gamma,f)\\ &\cong \Omega_{n-1-2\sum_{i=1}^kn_i}^{Spin}[\frac{1}{2}]\Big(B\big(S^1\times_H\prod_i^kU(n_i)\big)\amalg B\big(S^1\times_H\prod_i^kU(n_i)\big)\Big)\\ &\cong \Omega_{n-1-2\sum_{i=1}^kn_i}^{Spin}[\frac{1}{2}]\Big(B\big((S^1/H)\times\prod_i^kU(n_i)\big)\amalg B\big((S^1/H)\times\prod_i^kU(n_i)\big)\Big).
\end{align*}

Here the first isomorphism is induced by the two-fold coverings
\(\Gamma_u,\Gamma_v\rightarrow \Gamma\). Moreover, the third
isomorphism is induced by the homomorphism
\begin{align*}
  S^1\times_H\prod_{i=1}^k U(n_i) &\rightarrow
  S^1/H\times\prod_{i=1}^k U(n_i),\\
[z,u_1,\dots,u_k]&\mapsto ([z],z^{\alpha_1}u_1,\dots,z^{\alpha_k}u_k).
\end{align*}

We also have
\begin{align*}
  \Omega_n^{C,Spin,S^1}[S^1,V][\frac{1}{2}] &\cong \tilde{\Omega}_n^{Spin}[\frac{1}{2}](MU(n_1)\wedge\dots\wedge MU(n_k))\\
&\cong \tilde{\Omega}_n^{SO}[\frac{1}{2}](MU(n_1)\wedge\dots\wedge MU(n_k))\\
&\cong \Omega_{n-2\sum_{i=1}^k n_i}^{SO}[\frac{1}{2}](BU(n_1)\times\dots\times BU(n_k))\\
&\cong \Omega_{n-2\sum_{i=1}^k n_i}^{Spin}[\frac{1}{2}](BU(n_1)\times\dots\times BU(n_k)).
\end{align*}

Here the first isomorphism has been shown by Saihi \cite{MR1846617}, while
the second and fourth isomorphism are deduced from the fact that, after
inverting two, \(\Omega_*^{SO}\) and \(\Omega_*^{\text{Spin}}\) become isomorphic.

Moreover, the third isomorphism is constructed as follows: Make a map
\(\phi:M\rightarrow MU(n_1)\wedge\dots\wedge MU(n_k)\) transversal to the
zero section of the classifying bundle over \(BU(n_1)\times\dots\times
BU(n_k)\).
Then restrict \(\phi\) to the preimage of the zero section.
This gives an element of \(\Omega_{*-2\sum_{i=1}^k n_i}^{SO}(BU(n_1)\times\dots\times BU(n_k))\).

We will need a basis of the \(\Omega_*^{Spin}[\frac{1}{2}]\)-module \(\Omega_*^{Spin}[\frac{1}{2}](BU(n_1)\times\dots\times BU(n_k))\).
For odd \(i>0\) denote by \(X_i\) the tensor product of two copies of the tautological line bundle over \(\C P^i\).
For even \(i>0\) denote by \(X_i\) the tensor product of two copies of the tautological line bundle over \(\C P(\gamma\otimes\gamma\oplus \C)\). Here \(\gamma\) denotes the tautological bundle over \(\C P^{i-1}\).
Moreover for \(i=0\) denote by \(X_0\) the trivial line bundle over a point.
Then the bundles
\begin{equation}
\label{eq:1}
  \prod_{i=1}^k\prod_{h=1}^{n_i}X_{j_{hi}}\quad \text{ with } 0\leq j_{1i}\leq\dots\leq j_{n_ki}
\end{equation}
form a basis of \(\Omega_*^{Spin}[\frac{1}{2}](BU(n_1)\times\dots\times BU(n_k))\) as a \(\Omega^{Spin}_*[\frac{1}{2}]\)-module.

This can be seen as follows: First of all the \(X_i\), \(i\geq 0\),
form a basis of \(\Omega_*^{Spin}[\frac{1}{2}](BS^1)\) as a module over
\(\Omega_*^{Spin}[\frac{1}{2}]\) because the characteristic numbers \(\langle
c_1^i(X_i),[B(X_i)] \rangle\) are powers of two \cite[Theorem 18.1]{MR0176478}.
Here \(B(X_i)\) denotes the base space of \(X_i\).
Moreover, for a torus \(T=(S^1)^k\), \(\Omega_*^{Spin}[\frac{1}{2}](BT)\) is isomorphic to \(\bigotimes_{i=1}^k \Omega_*^{Spin}[\frac{1}{2}](BS^1)\).
Now consider the Atiyah--Hirzebruch spectral sequences for
\[\Omega_*^{Spin}[\frac{1}{2}](BT)\]
 and
\[\Omega_*^{Spin}[\frac{1}{2}](BU(n_1)\times\dots\times BU(n_k)),\]
 where \(T\) is a maximal torus of \(U(n_1)\times\dots\times U(n_k)\).
They degenerate at the \(E^2\)-level.
Hence, it follows
 that the products in (\ref{eq:1}) form a basis of
 \(\Omega_*^{Spin}[\frac{1}{2}](BU(n_1)\times\dots\times BU(n_k))\) from a
 comparison of the singular homologies of the two spaces (for details
 see \cite[Section 4.3]{MR1407034}).

The proof of Theorem~\ref{sec:Spin-case} follows a similar inductive strategy as the proof of Theorem~\ref{sec:non-semi-free}.
To realize this strategy we have to construct Spin \(S^1\)-manifolds \(M\) of type \(\mathcal{F}_i\) such that we have some control on the normal bundle of the \(\sigma_i\)-stratum in \(M\).
Some of the manifolds \(M\) will be provided by total spaces of twisted projective space bundles \(\widetilde{\C TP}(E_1;E_2)\) which we construct first.

\begin{constr}
\label{sec:Spin-case-3}
Let \(E_i\rightarrow B_i\), \(i=1,2\), be unitary \(S^1\)-vector bundles such that the actions on the base manifolds \(B_i\) are trivial.
Then \(E_i\) splits as a sum of unitary \(S^1\)-vector bundles
\begin{equation*}
  E_i=\bigoplus_{j=1}^{k_i}E_i^{\beta_{ij}}
\end{equation*}
with \(\beta_{i1},\dots,\beta_{ik_i}\in \mathbb{Z}\) such that the action of \(z\in S^1\) on each \(E_i^{\beta_{ij}}\) is given by multiplication with \(z^{\beta_{ij}}\).
We call \(E_i^{\beta_{ij}}\) the weight bundle of the weight \(\beta_{ij}\).

Moreover, we write
\begin{equation*}
  E_i^+=\bigoplus_{\beta_{ij}>0} E_i^{\beta_{ij}}
\end{equation*}
and 
\begin{equation*}
  E_i^-=\bigoplus_{\beta_{ij}<0} E_i^{\beta_{ij}}.
\end{equation*}
Then we have \(E_i=E_i^+\oplus E_i^-\oplus E_i^0\).

Also let \(\beta_{2\max}\in\{\beta_{21},\dots,\beta_{2k_2}\}\) such that \(|\beta_{2\max}|=\max_{1\leq j\leq k_2} |\beta_{2j}|\).

To ensure that the total spaces of our twisted projective space bundles are Spin manifolds we have to make some assumptions on \(E_1\) and \(E_2\). We also do the construction separately in different cases depending on the dimensions of the vector bundles and the appearing weights.

In the following we assume that \(\dim_\C E_2\) is even and \(\dim_\C E_i^0=1\) for \(i=1,2\).
We also assume that one of the following three cases holds:
\begin{enumerate}
\item\label{item:4} \(\dim_\C E_1\) is odd.
\item\label{item:5} \(\dim_\C E_1\) is even, \(\beta_{2\max}=\pm 2\), \(\dim_\C E_2^{\beta_{2\max}}=1\), \(|\beta_{2j}|=|\beta_{1j'}|=1\) for \(\beta_{2j}\neq \beta_{2\max}\) and all \(j'\in \{1,\dots,k_1\}\).
\item\label{item:6} \(\dim_\C E_1\) is even, case \ref{item:5} does not hold and \(E_2^{\beta_{2\max}}\) splits off an \(S^1\)-invariant line bundle \(F\).
\end{enumerate}

In each of these cases we define a free \(T^2\)-action on  the product of sphere bundles \(S(E_1)\times S(E_2)\).

In case \ref{item:4} we define this action as follows
\begin{equation*}
  (s,t)\cdot ((e_1^+,e_1^-,e^0_1),(e_2^+,e_2^-,e^0_2))=((te_1^+,t^{-1}e_1^-,te^0_1),(tse_2^+,t^{-1}s^{-1}e_2^-,se^0_2)),
\end{equation*}
where \((s,t)\in S^1\times S^1=T^2\), \(e_i^{\pm}\in E_i^{\pm}\) and \(e_i^0\in E_i^0\).

If, in case \ref{item:5}, \(\beta_{2\max}<0\), we define this action by
\begin{equation*}
  (s,t)\cdot ((e_1^+,e_1^-,e^0_1),(e_2^+,e_2^-,e^0_2,f))=((te_1^+,t^{-1}e_1^-,te^0_1),(tse_2^+,t^{-1}s^{-1}e_2^-,se^0_2,s^{-1} f)),
\end{equation*}
where \((s,t)\in S^1\times S^1=T^2\), \(e_1^{\pm}\in E_1^{\pm}\), \(e_2^+\in E_2^+\), \(e_2^-\in E_2^-\ominus E_2^{\beta_{2\max}}\), \(e_i^0\in E_i^0\) and \(f\in E_2^{\beta_{2\max}}\).

If, in case \ref{item:5}, \(\beta_{2\max}>0\), we define this action by
\begin{equation*}
  (s,t)\cdot ((e_1^+,e_1^-,e^0_1),(e_2^+,e_2^-,e^0_2,f))=((te_1^+,t^{-1}e_1^-,te^0_1),(tse_2^+,t^{-1}s^{-1}e_2^-,se^0_2,sf)),
\end{equation*}
where \((s,t)\in S^1\times S^1=T^2\), \(e_1^{\pm}\in E_1^{\pm}\), \(e_2^+\in E_2^+\ominus E_2^{\beta_{2\max}}\), \(e_2^-\in E_2^-\), \(e_i^0\in E_i^0\) and \(f\in E_2^{\beta_{2\max}}\).

If, in case \ref{item:6}, \(\beta_{2\max}<0\), we define this action by
\begin{equation*}
  (s,t)\cdot ((e_1^+,e_1^-,e^0_1),(e_2^+,e_2^-,e^0_2,f))=((te_1^+,t^{-1}e_1^-,te^0_1),(tse_2^+,t^{-1}s^{-1}e_2^-,se^0_2,s^{-1}t^{-2} f)),
\end{equation*}
where \((s,t)\in S^1\times S^1=T^2\), \(e_1^{\pm}\in E_1^{\pm}\), \(e_2^+\in E_2^+\), \(e_2^-\in E_2^-\ominus F\), \(e_i^0\in E_i^0\) and \(f\in F\).

If, in case \ref{item:6}, \(\beta_{2\max}>0\), we define this action by
\begin{equation*}
  (s,t)\cdot ((e_1^+,e_1^-,e^0_1),(e_2^+,e_2^-,e^0_2,f))=((te_1^+,t^{-1}e_1^-,te^0_1),(tse_2^+,t^{-1}s^{-1}e_2^-,se^0_2,st^2 f)),
\end{equation*}
where \((s,t)\in S^1\times S^1=T^2\), \(e_1^{\pm}\in E_1^{\pm}\), \(e_2^+\in E_2^+\ominus F\), \(e_2^-\in E_2^-\), \(e_i^0\in E_i^0\) and \(f\in F\).

We denote the orbit space of these actions by \(\widetilde{\C TP}(E_1;E_2)\).
It is diffeomorphic to a \(\C P^{n_2}\)-bundle over a manifold \(M\), where \(M\) is a \(\C P^{n_1}\)-bundle over \(B_1\times B_2\).
Here we have \(n_i=\dim_\C E_i -1\).

Our definition of \(\widetilde{\C TP}(E_1;E_2)\) is a bit different from the one given in \cite{MR695647}, which is essentially our definition in the first case with the extra assumption that \(\dim_{\mathbb{C}}E_1^+=0\).
We have to alter their definition to guarantee that our twisted projective space bundles are Spin manifolds if \(E_1\) and \(E_2\) are Spin vector bundles over Spin manifolds.
Let us now check this fact.

One can compute the cohomology of the total space of the \(\C P^{n_1}\)-bundle \(\xi:M\rightarrow B_1\times B_2\) by using the Leray--Hirsch Theorem as
  \begin{equation*}
    H^*(M;\mathbb{Z}_2)\cong H^*(B_1\times B_2;\mathbb{Z}_2)[u]/(f(u)),
  \end{equation*}
where \(u\) is the \(\mathrm{mod}\,\, 2\)-reduction of the first Chern class of the tautological line bundle over \(M\) and \(f(u)\) is a polynomial of degree \(2(n_1+1)\).

Moreover, its tangent bundle splits as a direct sum
\begin{equation*}
  \xi^*T(B_1\times B_2)\oplus \eta,
\end{equation*}
where \(\eta\) is the tangent bundle along the fibers of \(\xi\).

Note that \(\eta\) is isomorphic to \(\gamma\otimes (E_1^+\oplus
E_1^0)\oplus \bar{\gamma}\otimes E_1^-\), where \(\gamma\) denotes the
tautological line bundle over \(M\).

From this fact it follows that the second Stiefel--Whitney class of \(M\) is given by
\begin{equation*}
  w_2(M)=w_2(B_1\times B_2)+ (n_1+1) u + w_2(E_1).
\end{equation*}

The same reasoning with \(M\) replaced by \(\widetilde{\C TP}(E_1;E_2)\) and \(B_1\times B_2\) replaced by \(M\) shows that the second Stiefel--Whitney class of \(\widetilde{\C TP}(E_1;E_2)\) is given by
\begin{equation*}
  w_2(\widetilde{\C TP}(E_1;E_2))=w_2(B_1\times B_2)+w_2(E_1)+w_2(E_2).
\end{equation*}
Therefore it follows that \(\widetilde{\C TP}(E_1;E_2)\) is a Spin manifold if \(B_1\), \(B_2\) are Spin-manifolds and \(E_1\), \(E_2\) are Spin-vector bundles.

The \(S^1\)-actions on \(E_1\) and \(E_2\) induce an \(S^1\)-action on \(\widetilde{\C TP}(E_1;E_2)\).
The submanifolds 
\begin{equation*}
N_{\beta_{1j},\beta_{2j'}}=(S(E_{1}^{\beta_{1j}})\times S(E_2^{\beta_{2j'}}))/T^2\subset \widetilde{\C TP}(E_1; E_2)^{S^1}  
\end{equation*}
are fixed by this \(S^1\)-action on \(\widetilde{\C TP}(E_1;E_2)\).

In the proof of Theorem \ref{sec:Spin-case} we want the manifolds \(\widetilde{\C TP}(E_1;E_2)\) to be of type \(\mathcal{F}_i\) for certain families \(\mathcal{F}_i\) of slice types.
For this it is important to know the weights of the \(S^1\)-representations on the normal bundles of \(N_{\beta_{1j},\beta_{2j'}}\).
In particular it is important how they differ from the weights at \(N_{0,0}\).
The weights of the \(S^1\)-action on the normal bundle of \(N_{\beta_{1j},\beta_{2j'}}\) can be computed as listed in Tables \ref{tab:1}, \ref{tab:2}, \ref{tab:3}.

\begin{table}
  \centering
  \begin{tabular}{|l|l|}
    \((\beta_{1j},\beta_{2j'})\)& weights at \(N_{\beta_{1j},\beta_{2j'}}\)\\\hline\hline
    \((0,0)\)& \(\beta_{11},\dots,\beta_{1k},\beta_{21},\dots,\beta_{2k'}\)\\\hline
    \((0,\gamma)\) with \(\gamma\neq 0\)& \(-|\gamma|\), \\ & \(\beta_{11},\dots,\beta_{1k}\), \\ & \(\sign\beta_{2i}(|\beta_{2i}|-|\gamma|)\) for \(i=1,\dots k'\) and \(\beta_{2i}\neq 0\)\\\hline
    \((\gamma,0)\) with \(\gamma\neq 0\)& \(-|\gamma|\), \\&\(\sign\beta_{1i}(|\beta_{1i}|-|\gamma|)\) for \(i=1,\dots k\) and \(\beta_{1i}\neq 0\),\\ & \(\sign\beta_{2i}(|\beta_{2i}|-|\gamma|)\) for \(i=1,\dots k'\) and \(\beta_{2i}\neq 0\)\\\hline
    \((\gamma_1,\gamma_2)\) with \(\gamma_i\neq 0\)& \(-|\gamma_1|\),\\ & \(\sign\beta_{1i}(|\beta_{1i}|-|\gamma_1|)\) for \(i=1,\dots k\) and \(\beta_{1i}\neq 0\),\\ & \(\sign\beta_{2i}(|\beta_{2i}|-|\gamma_2|)\) for \(i=1,\dots k'\) and \(\beta_{2i}\neq 0\),\\ & \(|\gamma_1|-|\gamma_2|\)
  \end{tabular}
  \caption{The weights of the \(S^1\)-action on \(\widetilde{\C TP}\) in the first case.}
  \label{tab:1}
\end{table}

\begin{table}
  \centering
  \begin{tabular}{|l|l|}
    \((\beta_{1j},\beta_{2j'})\)& weights at \(N_{\beta_{1j},\beta_{2j'}}\)\\\hline\hline
    \((0,0)\)&\(\beta_{11},\dots,\beta_{1k},\beta_{21},\dots,\beta_{2k'}\)\\\hline
    \((0,\pm 1)\)& \(\beta_{11},\dots,\beta_{1k}, \mp 1\)\\\hline
    \((0,\pm 2)\)& \(\beta_{11},\dots,\beta_{1k}, -2, \mp 1\)\\\hline
    \((\pm 1,0)\), \((\pm 1,\pm 1)\)& \(-1,\pm 2\)\\\hline
    \((\pm 1, \pm 2)\) & \(-1,-2\)
  \end{tabular}
  \caption{The weights of the \(S^1\)-action on \(\widetilde{\C TP}\) in the second case.}
  \label{tab:2}
\end{table}

\begin{table}
  \centering
  \begin{tabular}{|l|l|}
    \((\beta_{1j},\beta_{2j'})\)& weights at \(N_{\beta_{1j},\beta_{2j'}}\)\\\hline\hline
    \((0,0)\)& \(\beta_{11},\dots,\beta_{1k},\beta_{21},\dots,\beta_{2k'}\)\\\hline
    \((0,\gamma)\) with \(\gamma\neq 0\)& \(-|\gamma|\), \(\beta_{11},\dots,\beta_{1k}\),\\ & \(\sign\beta_{2i}(|\beta_{2i}|-|\gamma|)\) for \(i=1,\dots k'\) and \(\beta_{2i}\neq 0\)\\\hline
    \((\gamma,0)\) with \(\gamma\neq 0\)& \(-|\gamma|\),\\&\(\sign\beta_{1i}(|\beta_{1i}|-|\gamma|)\) for \(i=1,\dots k\) and \(\beta_{1i}\neq 0\),\\ 
                                    & \(\sign\beta_{2i}(|\beta_{2i}|-|\gamma|)\) for \(i=1,\dots k'\) and \(\beta_{2i}\neq 0,\alpha\),\\ & \(\sign \alpha (|\alpha|-2|\gamma|)\)\\\hline
    \((\gamma_1,\gamma_2)\) with \(\gamma_i\neq 0\),& \(-|\gamma_1|\),\\ & \(\sign\beta_{1i}(|\beta_{1i}|-|\gamma_1|)\) for \(i=1,\dots k\) and \(\beta_{1i}\neq 0\),\\
 \(\gamma_2\neq \alpha\) & \(\sign\beta_{2i}(|\beta_{2i}|-|\gamma_2|)\) for \(i=1,\dots k'\) and \(\beta_{2i}\neq 0,\alpha\),\\ & \(|\gamma_1|-|\gamma_2|\),\\ &\(\sign\alpha(|\alpha|-|\gamma_1|-|\gamma_2|)\)\\\hline
\((\gamma,\alpha)\) with \(\gamma\neq 0\)& \(-|\gamma|\),\\ & \(\sign\beta_{1i}(|\beta_{1i}|-|\gamma|)\) for \(i=1,\dots k\) and \(\beta_{1i}\neq 0\), \\ & \(\sign\beta_{2i}(|\beta_{2i}|-|\alpha|+|\gamma|)\) for \(i=1,\dots k'\) and \(\beta_{2i}\neq 0,\alpha\),\\ & \(2|\gamma|-|\alpha|\)
  \end{tabular}
  \caption{The weights of the \(S^1\)-action on \(\widetilde{\C TP}\) in the third case.}
  \label{tab:3}
\end{table}

In the proof of Theorem \ref{sec:Spin-case} we will always have
\begin{equation*}
  \max_{1\leq j\leq k_1}|\beta_{1j}|<\max_{1\leq j\leq k_2} |\beta_{2j}|= |\beta_{2\max}|.
\end{equation*}

Assume that this holds and that \((\beta_{1j},\beta_{2j'})\neq (0,0), (0,\pm \beta_{2\max})\), then we have in cases \ref{item:4} and \ref{item:6}:

\begin{equation*}
  \max\{|\text{weights at } N_{\beta_{1j},\beta_{2j'}}|\} < \max \{|\text{weights at } N_{0,0}|\}.
\end{equation*}

In case \ref{item:5} let \(Y_{\beta_{1j},\beta_{2j'}}\) be the component of \(\widetilde{\C TP}(E_1;E_2)^{S^1}\) containing \(N_{\beta_{1j},\beta_{2j'}}\).
Then we have
\begin{equation*}
  \codim Y_{\beta_{1j},\beta_{2j'}} < \dim E_1+\dim E_2-2=\codim Y_{0,0}
\end{equation*}
if \((\beta_{1j},\beta_{2j'})\neq (0,0),(0,\pm 2), (\pm 1, \pm 2)\) or \((\beta_{1j},\beta_{2j'})\neq (0,0),(0,\pm 2)\) and \(\dim_\C E_1\neq 2\).

If we have \((\beta_{1j},\beta_{2j'})=(\pm 1, \pm 2)\), \(\dim_\C E_1=2\) and \(\dim_\C E_2\geq 4\), then the weight \(-2\) appears in the weights at \(N_{\beta_{1j},\beta_{2j'}}\) with multiplicity greater than one.
Therefore the \(S^1\)-representation at \(N_{\beta_{1j},\beta_{2j'}}\) does not coincide with the \(S^1\)-representation at \(N_{0,0}\).

Note that in all three cases we have that the weights of the \(S^1\)-representation at \(N_{0,\beta_{2\max}}\) are given by
\begin{equation*}
  \{-|\beta_{2\max}|, \beta_{11},\dots,\beta_{1k}, \sign\beta_{2i}(|\beta_{2i}|-|\beta_{2\max}|);\; i=1,\dots, k' \text{ and } \beta_{2i}\neq 0\}.
\end{equation*}
\end{constr}

\begin{proof}[Proof of Theorem~\ref{sec:Spin-case}]
Now as in the non-Spin case we construct sections to the maps
\begin{equation*}
 \nu_i: \Omega_n^{C,Spin,S^1}[\mathcal{F}_i][\frac{1}{2}]\rightarrow\image \nu_i\subset \Omega_n^{C,Spin,S^1}[\sigma_i][\frac{1}{2}].
\end{equation*}

At first assume that \(\sigma_i=[H,V]\) with \(H\) finite and \(V=\bigoplus_{i=1}^k V_{\alpha_i}\) with \(0> \alpha_1\geq \dots\geq \alpha_k > \alpha_{k+1}=-m=-\ord H\).
Then, by Lemma \ref{sec:Spin-case-4}, the two copies of \(B\Gamma\) in \(\Omega^{C,Spin,S^1}_n[\frac{1}{2}][\sigma_i]\) correspond to the two Spin structures on
\begin{equation*}
  S^1\times_HV.
\end{equation*}

There are equivariant diffeomorphisms
\begin{align*}
  S^1/H\times \bigoplus_{i=1}^k V_{\alpha_i}&\rightarrow S^1\times_H \bigoplus_{i=1}^k V_{\alpha_i}\\
([z],v_1,\dots,v_k)&\mapsto [z,z^{-\alpha_1}v_1,\dots,z^{-\alpha_k}v_k]
\end{align*}
and
\begin{align*}
  S^1/H\times (V_{\alpha_1+m}\oplus\bigoplus_{i=2}^k V_{\alpha_i})&\rightarrow S^1\times_H \bigoplus_{i=1}^k V_{\alpha_i}\\
([z],v_1,\dots,v_k)&\mapsto [z,z^{-\alpha_1-m}v_1,\dots,z^{-\alpha_k}v_k],
\end{align*}
where the action  on the left hand spaces is given by the product action; and the action  on the right hand spaces  is induced by left multiplication on the first factor.

Moreover, we can (non-equivariantly) identify
\begin{equation*}
  S^1/H\times \bigoplus_{i=1}^k V_{\alpha_i} \cong S^1\times \bigoplus_{i=1}^k V_{\alpha_i}
\end{equation*}
and
\begin{equation*}
  S^1/H\times (V_{\alpha_i+m}\oplus\bigoplus_{i=2}^k V_{\alpha_i}) \cong S^1\times (V_{\alpha_i+m}\oplus\bigoplus_{i=2}^k V_{\alpha_i}).
\end{equation*}
Composing these diffeomorphisms leads to the map
\begin{align*}
  S^1\times \C^k\cong S^1\times\bigoplus_{i=1}^k V_{\alpha_i}&\rightarrow S^1\times (V_{\alpha_i+m}\oplus\bigoplus_{i=2}^k V_{\alpha_i})\cong S^1\times \C^k\\
(z,v_1,\dots,v_k)&\mapsto (z,zv_1,\dots, v_k).
\end{align*}

This map interchanges the two Spin structures on \(S^1\times \C^k\).
Therefore one of the two Spin structures on \(S^1\times_H V\) equivariantly bounds \(D^2\times (\bigoplus_{i=1}^k V_{\alpha_i})\).
The other equivariantly bounds \(D^2\times (V_{\alpha_i+m}\oplus\bigoplus_{i=2}^k V_{\alpha_i})\).

Let \(E=\prod_{i=1}^{k+1}\prod_{h=1}^{n_i}X_{j_{hi}}\in\Omega_*^{Spin}[\frac{1}{2}](B\Gamma_{\text{first copy}})\). Then the sphere bundle associated to \(\prod_{i=1}^{k+1}\prod_{h=1}^{n_i}X_{j_{hi},\alpha_i}\) gives the desired preimage.
Here \(X_{j,\alpha}\) denotes the bundle \(X_j\) equipped with the action of \(S^1\) induced by multiplication with \(z^{\alpha}\) for \(z\in S^1\).

This can be seen as follows:
The normal bundle of \(S(E)^H\) in \(S(E)\) bounds the normal bundle of \(D(E)^H\) in \(D(E)\).
Therefore a tubular neighborhood of an orbit of type \([H,V]\) in \(S(E)\), equipped with its natural Spin structure, equivariantly bounds \(D^2\times V\times \R^k\).
Hence, this Spin structure is the one which corresponds to \(B\Gamma_{\text{first copy}}\).
In particular, \(S(E)\) is the desired preimage of \(E\).

Next assume that \(E=\prod_{i=1}^{k+1}\prod_{h=1}^{n_i}X_{j_{hi}}\in \Omega_*[\frac{1}{2}](B\Gamma_{\text{second copy}})\); then, by an argument similar to the one from above, the sphere bundle  associated to 
\begin{equation*}(X_{j_{n_11},\alpha_1}\otimes \bar{X}_{j_{1(k+1)},m}\oplus X_{j_{1k+1},-m}) \times \prod_{h=1}^{n_1-1}X_{j_{h1},\alpha_1}\times\prod_{i=2}^k\prod_{h=1}^{n_i}X_{j_{hi},\alpha_i}\end{equation*}
gives the desired preimage.
Here \(\bar{X}_{j_{1(k+1)}}\) is the dual vector bundle of \(X_{j_{1(k+1)}}\).
Moreover, \(X_{j_{n_11}}\otimes \bar{X}_{j_{1(k+1)}}\oplus X_{j_{1k+1}}\) is the vector bundle over the base space \(B_1\times B_2\)  of \(X_{j_{n_11}}\times X_{j_{1k+1}}\) given by the Whitney sum of \(pr_1^*(X_{j_{n_11}})\otimes pr_2^* (\bar{X}_{j_{1(k+1)}})\) and \(pr_2^* (X_{j_{1k+1}})\).

Next assume that \(\sigma_i=[S^1,V]\) where \(V=\bigoplus_{i=1}^k V_{\alpha_i}^{n_i}\) with \(\alpha_i> \alpha_{i+1}\).
For this case we will use Construction~\ref{sec:Spin-case-3} of twisted projective space bundles \(\widetilde{\C TP}(E_1;\;\; E_2)\).

{\bf Case 1:}
At first also assume that all weights of \(V\) are negative and the minimal weight appears with multiplicity one.
Then one can see that the map \(\Omega_*^{C,Spin,S^1}[\frac{1}{2}][\sigma_i]\rightarrow \Omega_*^{C,Spin,S^1}[\frac{1}{2}][\mathcal{F}_{i-1}]\) has the same image as the first section constructed above for slice types with finite isotropy groups.
Hence it is injective.

{\bf Case 2:}
Next assume that all weights except \(\alpha_1\) of \(V\) are negative, the minimal weight and \(\alpha_1\) appear with multiplicity one and \(0>\alpha_1-m\geq \alpha_i\) for all \(i>1\), where \(-m=\alpha_k\).
Then one can see that the map \(\Omega_*^{C,Spin,S^1}[\frac{1}{2}][\sigma_i]\rightarrow \Omega_*^{C,Spin,S^1}[\frac{1}{2}][\mathcal{F}_{i-1}]\) has the same image as the second section constructed above for slice types with finite isotropy groups.

Then
\begin{equation*}
  (X_{j_{11},\alpha_1-m}\otimes \bar{X}_{j_{1k}, \alpha_k}\oplus X_{j_{1k}, \alpha_k}) \times \prod_{i>1}\prod_{h=1}^{n_i} X_{j_{hi},\alpha_i},\quad 0\leq j_{1i}\leq j_{2i}\leq \dots \leq j_{n_ii} \text{ for } i\geq 1
\end{equation*}
is a basis of \(\Omega_*^{C,Spin}[\frac{1}{2}][\sigma_i]\) as an
\(\Omega_*^{Spin}[\frac{1}{2}]\)-module.
To see this, note that the map
\begin{align*}
  \Omega_*[\frac{1}{2}](BU(1)\times BU(1)\times \prod_{i=2}^{k-1} BU(n_i))&\rightarrow \Omega_*[\frac{1}{2}](BU(1)\times BU(1)\times
  \prod_{i=2}^{k-1} BU(n_i))\\
  X\oplus Y \oplus Z \;& \mapsto \; X\otimes \bar{Y}\oplus Y\oplus Z
\end{align*}
is an isomorphism. Moreover, the elements given above are the images of
the basis (\ref{eq:1}) under this isomorphism.

If \(j_{11}\geq j_{n_22}\) with \(\alpha_2=\alpha_1-m\) or \(\alpha_1-m>\alpha_2\), then the image of \[X_{j_{11},{\alpha_1-m}}\otimes X_{j_{1k}, \alpha_k}\oplus X_{j_{1k},\alpha_k} \times \prod_{k>i>1}\prod_{h=1}^{n_i} X_{j_{hi},\alpha_i} \] under the map
\(\Omega_*^{C,Spin,S^1}[\frac{1}{2}][\sigma_i]\rightarrow \Omega_*^{C,Spin,S^1}[\frac{1}{2}][\mathcal{F}_{i-1}]\)
is part of a basis of the image of the second section constructed for normal orbit types with finite isotropy groups.
Hence, for these basis elements we do not have to define preimages under \(\nu_i\).

Therefore assume that \(\alpha_2=\alpha_1-m\) and \(j_{11}< j_{n_22}\).
Then 
\begin{equation*}
  \begin{split}
  \widetilde{\C TP}(\C\times \prod_{h=1}^{n_2-1}X_{j_{h2},\alpha_2} \times \prod_{2<i<k} \prod_{h=1}^{n_i}&X_{j_{hi},\alpha_i};\; \\ &(X_{j_{11},{\alpha_1-m}}\otimes \bar{X}_{j_{1k},\alpha_k}\oplus X_{j_{1k},\alpha_k})\times X_{j_{2n_2},\alpha_2}\times \C)  
  \end{split} 
  \end{equation*}
is mapped by \(\nu_i\) to
\begin{equation*}
  \begin{split}
     \prod_{h=1}^{n_2-1}X_{j_{h2},\alpha_2}\times \prod_{2<i<k}\prod_{h=1}^{n_i} X_{j_{hi},\alpha_i}
\times &( X_{j_{11},{\alpha_1-m}}\otimes \bar{X}_{j_{1k},\alpha_k}\oplus X_{j_{1k},\alpha_k}\times X_{j_{2n_2},\alpha_2}\\ &- X_{j_{11},{\alpha_1-m}}\times X_{j_{2n_2},\alpha_2}\otimes \bar{X}_{j_{1k,\alpha_k}}\oplus X_{j_{1k},\alpha_k}).
  \end{split}
  \end{equation*}
Here the two components of \(\widetilde{\C TP}(E_1;E_2)^{S^1}\) which correspond to the two summands of the above sum are given by \(\widetilde{\C TP}(\C; \C)\) and \(\widetilde{\C TP}(\C; X_{j_{1k},\alpha_k})\).
The other fixed point components are of different slice types.
Therefore, we have found the generators  in this case.

{\bf Case 3:}
Next assume that \(\sigma_i=[S^1,V]\), such that all weights except \(\alpha_1\) of \(V\) are negative, the minimal weight and \(\alpha_1\) appear with multiplicity one and \(\alpha_1<-\alpha_k=m\), \(\alpha_1-m< \alpha_2\).
Then
\begin{equation*}
  \widetilde{\C TP}(\C\times\prod_{h=1}^{n_2-1} X_{j_{h2},\alpha_2}\times \prod_{2<i<k}\prod_{h=1}^{n_i} X_{j_{hi},\alpha_i};\;\; X_{j_{11},\alpha_1}\times X_{j_{n_22},\alpha_2}\times X_{j_{1k},\alpha_k}\times \C)
  \end{equation*}
is mapped by \(\nu_i\) to
\begin{equation*}
    \prod_{i=1}^k\prod_{h=1}^{n_i} X_{j_{hi},\alpha_i}.
\end{equation*}
Here the component of \(\widetilde{\C TP}(E_1;E_2)^{S^1}\) which corresponds to this product is given by \(\widetilde{\C TP}(\C;\C)^{S^1}\).
The other components of the fixed point set are of different slice types.
Therefore we have found the generators in this case.

{\bf Case 4:}
Next assume that \(\sigma_i=[S^1,V]\), such that \(\alpha_1 > 0\) appears with multiplicity at least two and the minimal weight appears with multiplicity one and is negative and \(\alpha_1<-\alpha_k=m\).
Then
\begin{equation*}
  \widetilde{\C TP}(\C\times \prod_{h=3}^{n_1} X_{j_{h1},\alpha_1} \times \prod_{1<i<k}\prod_{h=1}^{n_i} X_{j_{hi},\alpha_i};\;\; X_{j_{11},\alpha_1}\times X_{j_{21},\alpha_1}\times X_{j_{1k},\alpha_k}\times \C)
\end{equation*}
is mapped by \(\nu_i\) to
\begin{equation*}
      \prod_{i=1}^k\prod_{h=1}^{n_i} X_{j_{hi},\alpha_i}
\end{equation*}
Therefore, we have found the generators in this case.

{\bf Case 5:}
Next assume that \(\sigma_i=[S^1,V]\), such that \(\alpha_1>\alpha_2 > 0\) and the minimal weight appears with multiplicity one and is negative and \(\alpha_1<-\alpha_k=m\).
Then
\begin{equation*}
  \widetilde{\C TP}(\C\times \prod_{h=2}^{n_1}X_{j_{h1},\alpha_1} \times \prod_{h=2}^{n_2}X_{j_{h2},\alpha_2} \times \prod_{1<i<k}\prod_{h=1}^{n_i} X_{j_{hi},\alpha_i};\;\; X_{j_{11},\alpha_1}\times X_{j_{12},\alpha_2} \times X_{j_{1k},\alpha_k} \times \C)  
  \end{equation*}
is mapped by \(\nu_i\) to
\begin{equation*}
    \prod_{i=1}^k\prod_{h=1}^{n_i} X_{j_{hi},\alpha_i}
\end{equation*}
Therefore, we have found the generators in this case.

{\bf Case 6:}
Next assume that \(\alpha_1=2\) and \(\alpha_2=\pm 1\) are the only weights.
Then
\begin{equation*}
  M=\C P(X_{j_{11}\alpha_1}\oplus \C) \times \C P(X_{j_{12},\alpha_2}\oplus \C)
\end{equation*}
is a manifold which satisfies Condition \(C\) such that the weights at the four fixed point components are \((-2,-1),(-2,1),(2,-1),(2,1)\).
If \(\alpha_2=-1\), we conjugate the complex structure on the summand \(N(M^{S^1},M^{\mathbb{Z}_2})\) of the normal bundle of the last fixed point component.
In this way we get the desired generators in this case.

{\bf Case 7:}
In all other cases let \(\beta\in\{\alpha_1,\alpha_k\}\) be a weight of maximal norm. Then there is a second weight \(\delta\neq \beta, -\beta\).
In these cases one of
\begin{equation*}
  \widetilde{\C TP}(\C\times \prod_{h=2}^{n_\delta}X_{j_{h\delta},\delta}\times \prod_{\alpha_i\neq \delta,-\beta,\beta} \prod_{h=1}^{n_i}X_{j_{hi},\alpha_i};\;\; X_{j_{1\delta},\delta}\times \prod_{h=1}^{n_\beta} X_{j_{h,\beta},\beta}\times\prod_{h=1}^{n_{-\beta}} X_{j_{h,-\beta},-\beta}\times \C)  
  \end{equation*}
or
\begin{equation*}
  \begin{split}
  \widetilde{\C TP}(\C\times \prod_{\alpha_i\neq -\beta,\beta}\prod_{h=1}^{n_i} X_{j_{hi},\alpha_i};\;\;\prod_{h=1}^{n_\beta} X_{j_{h,\beta},\beta}\times\prod_{h=1}^{n_{-\beta}} X_{j_{h,-\beta},-\beta}\times \C)  
  \end{split}
  \end{equation*}
is Spin and mapped by \(\nu_i\) to
\begin{equation*}
  \prod_{i=1}^k\prod_{h=1}^{n_i} X_{j_{hi},\alpha_i}.
\end{equation*}
Therefore, we have found the generators in this case.

Hence, we have found generators of the image of \(\nu_i\) in all cases. This completes the proof of Theorem~\ref{sec:Spin-case}.
\end{proof}

\section{$S^1$-manifolds not satisfying Condition C}
\label{sec:not_con_c}

In this section we prove Theorem~\ref{sec:introduction-3}.
In view of Theorems~\ref{sec:non-semi-free} and \ref{sec:Spin-case} it suffices to prove the following theorem.

\begin{theorem}
\label{sec:s1-manifolds-not-1}
  For \(G=SO\) or \(G=Spin\) and  \(\mathcal{F}=\mathcal{AE}\) or \(\mathcal{F}=\mathcal{AE}-\{[S^1,W]\,;\,\dim_\C W=1\}\), the natural map
  \begin{equation*}
    \Omega_{n}^{C,G,S^1}[\frac{1}{2}][\mathcal{F}]\rightarrow \Omega_{n}^{G,S^1}[\frac{1}{2}][\mathcal{F}]
  \end{equation*}
is surjective.
Here \(\Omega_{n}^{G,S^1}[\mathcal{F}]\) denotes the bordism group of
\(n\)-dimensional \(G\)-manifolds with effective \(S^1\)-action of
type \(\mathcal{F}\).  
\end{theorem}

To prove this theorem we first show that it holds whenever \(\Omega_n^{G,S^1}[\frac{1}{2}][\mathcal{F}]\) is replaced by \(\hat{\Omega}_n^{G,S^1}[\frac{1}{2}][\mathcal{F}]\), where \(\hat{\Omega}_n^{G,S^1}[\mathcal{F}]\) is the bordism group of \(n\)-dimensional \(G\)-manifolds \(M\) with effective \(S^1\)-action of type \(\mathcal{F}\) such that all closed singular strata are orientable. We also assume here that the singular strata of the bordisms are orientable.
This is Proposition~\ref{sec:s1-manifolds-not} below.

Then, in Lemma~\ref{sec:s1-manifolds-not-3} we show that the natural map \[\hat{\Omega}_n^{G,S^1}[\frac{1}{2}][\mathcal{F}]\rightarrow \Omega_n^{G,S^1}[\frac{1}{2}][\mathcal{F}]\] which forgets that the singular strata are orientable is actually an isomorphism.

\begin{remark}
\label{sec:resolv-sing-5}
  It has been shown by Edmonds \cite{MR603614} that the singular strata in a Spin-\(S^1\)-manifold are always orientable (see also Section 10 of \cite{MR954493}).
Therefore there is no difference between \(\hat{\Omega}_n^{\text{Spin},S^1}[\frac{1}{2}][\mathcal{F}]\) and \(\Omega_n^{\text{Spin},S^1}[\frac{1}{2}][\mathcal{F}]\).

Moreover, note that a singular stratum \(M^{H}\) in an effective orientable \(S^1\)-manifold can only be non-orientable if \(H\) is finite and of even order, since otherwise the normal bundle of \(M^H\) admits an invariant complex structure.
\end{remark}

\begin{prop}
\label{sec:s1-manifolds-not}
  For \(G=SO\) or \(G=Spin\) and  \(\mathcal{F}=\mathcal{AE}\) or \(\mathcal{F}=\mathcal{AE}-\{[S^1,W]\,;\,\dim_\C W=1\}\), the natural map
  \begin{equation*}
    \Omega_{n}^{C,G,S^1}[\frac{1}{2}][\mathcal{F}]\rightarrow \hat{\Omega}_{n}^{G,S^1}[\frac{1}{2}][\mathcal{F}]
  \end{equation*}
is surjective.
Here \(\hat{\Omega}_{n}^{G,S^1}[\mathcal{F}]\) denotes the bordism group of
\(n\)-dimensional \(G\)-manifolds \(M\) with effective \(S^1\)-action of
type \(\mathcal{F}\) such that, for all subgroups \(H\subset S^1\),
\(M^H\) is orientable.  
\end{prop}

\begin{proof}
To prove this proposition, we use an induction as in Sections~\ref{sec:unitary} and \ref{sec:Spin_case}.

To do so we first recall the definition of \emph{real} slice types.
A real \(S^1\)-slice type is a pair \([H,W]_{\R}\) where \(H\) is a closed subgroup of \(S^1\) and \(W\) is a isomorphism class of real \(H\)-representations with \(W^H=\{0\}\).
In this section we refer to the slice types of the previous section, i.e. to pairs \([H,W]\) with \(W\) an isomorphism class of complex \(H\)-representations, as \emph{complex} slice types.
 
The real slice types of \(S^1\), which appear in an orientable \(S^1\)-manifold, are then given by
\begin{align}
\label{eq:4}
  [\mathbb{Z}_m;&\prod_{-m<\alpha\leq -\frac{m}{2}} V_\alpha^{j_\alpha}]_{\R},&
  [S^1;&\prod_{\alpha<0} V_\alpha^{j_\alpha}]_{\R}.
\end{align}
Here the \(V^{j_\alpha}_\alpha\) are to be understood as real representations of the groups in question, that is, after forgetting the complex structure.

In particular, for each real slice type \(\rho\) there is a preferred complex slice type \(\sigma\) such that after forgetting the complex structure we have \(\rho=\sigma\).

Note that this slice type is the first complex slice type representing \(\rho\) with respect to the ordering introduced in Section~\ref{sec:unitary}.
We call two complex slice types equivalent if they represent the same real slice type.

The ordering of preferred complex slice types induces an ordering of the real slice types.
Moreover, this leads to a sequence \(\mathcal{G}_i\) of families of slice types such that
\begin{align*}
  \mathcal{G}_0&\text{ consists of semi-free real slice types}\\
  \mathcal{G}_{i+1}&=\mathcal{G}_i\amalg\{\rho_{i+1}\} \text{ for a real slice type } \rho_{i+1}\\
  \bigcup_{i=0}^{\infty}\mathcal{G}_i&=\mathcal{F}
\end{align*}

We will show by induction on \(i\) that every \([M] \in \Omega_*^{G,S^1}[\frac{1}{2}][\mathcal{G}_i]\) can be represented by an effective \(S^1\)-manifold which satisfies Condition C.

We have the following commutative diagram with exact rows:
\begin{equation*}
  \xymatrix{
    \Omega^{C,G,S^1}_*[\frac{1}{2}][\mathcal{F}_{h(i)-1}]\ar[r]\ar[d]&
    \Omega^{C,G,S^1}_*[\frac{1}{2}][\mathcal{F}_{h(i)}]\ar^{\nu^C_{h(i)}}[r]\ar[d]&
    \Omega^{C,G,S^1}_*[\frac{1}{2}][\sigma_{h(i)}]\ar^{\partial_{h(i)}^C}[r]\ar[d]&\Omega^{C,G,S^1}_{*-1}[\frac{1}{2}][\mathcal{F}_{h(i)-1}]\ar[d]\\
    \hat{\Omega}^{G,S^1}_*[\frac{1}{2}][\mathcal{G}_{i-1}]\ar[r]&
    \hat{\Omega}^{G,S^1}_*[\frac{1}{2}][\mathcal{G}_i]\ar^{\nu_i}[r]&
\hat{\Omega}^{G,S^1}_*[\frac{1}{2}][\rho_i]\ar^{\partial_i}[r]&\hat{\Omega}^{G,S^1}_{*-1}[\frac{1}{2}][\mathcal{G}_{i-1}]
}
\end{equation*}
Here \(\sigma_{h(i)}\) denotes the preferred complex slice type
representing the real slice type \(\rho_i\). Moreover the first row is the exact sequence considered in Sections \ref{sec:unitary} or \ref{sec:Spin_case} if we are in the oriented case or the Spin case, respectively.
 The second row is the analogous sequence in the real setting.
The group \(\hat{\Omega}^{G,S^1}_*[\rho_i]\) denotes the bordism group of \(S^1\)-vector bundles of type \(\rho_i\) with a \(G\)-structure on the total space and orientable base space.
The vertical maps are given by forgetting the complex structures on the normal bundles to the singular strata.

A diagram chase shows that we are done with the induction step if we
can show that the composition of maps
\begin{equation}
\label{eq:3}
  \Omega_*^{C,G,S^1}[\frac{1}{2}][\mathcal{F}_{h(i)}]\rightarrow
  \hat{\Omega}_*^{G,S^1}[\frac{1}{2}][\mathcal{G}_i]\rightarrow
  \ker \partial_i
\end{equation}
is surjective.

If \(E\) is an \(S^1\)-vector bundle of type \(\rho=[H,V]_\R\) as in (\ref{eq:4}), then \(E\) splits as a direct sum of weight bundles \(E_\alpha\) corresponding to the decomposition of \(V\) into weight spaces \(V_\alpha^{j_\alpha}\).
In almost all cases these weight bundles have natural \(S^1\)-invariant complex structures.
The only exception is the case that \(H=\mathbb{Z}_m\) is of even order and \(\alpha=-\frac{m}{2}\). 
If the total space of \(E\) is oriented and the base space of \(E\) is orientable, then this weight bundle is also a orientable bundle.
Moreover, an orientation of the base determines an orientation for the weight bundle and vice versa.
Hence, for \(m\) even we get the following isomorphisms similar to the isomorphisms in the case of complex slice types.

\begin{equation*}
  \begin{split}
  \hat{\Omega}^{SO,S^1}_n[\mathbb{Z}_m,\prod_{\alpha=-m+1}^{-\frac{m}{2}}V_\alpha^{j_\alpha}]_\R[\frac{1}{2}]&\cong\Omega^{SO}_{n-1-2\sum_{\alpha=-m+1}^{-m/2}j_\alpha}[\frac{1}{2}]\Big(B(S^1\times_{\mathbb{Z}_m}
  SO(2j_{-m/2}))\\ &\;\;\;\quad\quad\times B(\prod_{\alpha=-m+1}^{-m/2-1} U(j_\alpha))\Big)\\
\end{split}
\end{equation*}

\begin{equation}
  \label{sec:s1-manifolds-not-5}
\begin{split}
  \hat{\Omega}^{Spin,S^1}_n[\mathbb{Z}_m,\prod_{\alpha=-m+1}^{-\frac{m}{2}}V_\alpha^{j_\alpha}]_\R[\frac{1}{2}]&\cong\Omega^{Spin}_{n-1-2\sum_{\alpha=-m+1}^{-m/2}j_\alpha}[\frac{1}{2}]\Big(B(S^1\times_{\mathbb{Z}_m}
  SO(2j_{-m/2}))\\ &\;\;\;\quad\quad\times B(\prod_{\alpha=-m+1}^{-m/2-1} U(j_\alpha))\Big) \text{ if } j_{-m/2}>0\\
\end{split} 
\end{equation}

\begin{align*}
 \hat{\Omega}^{Spin,S^1}_n[\mathbb{Z}_m,\prod_{\alpha=-m+1}^{-\frac{m}{2}-1}V_\alpha^{j_\alpha}]_\R[\frac{1}{2}]&\cong\Omega^{Spin}_{n-1-2\sum_{\alpha=-m+1}^{-m/2-1}j_\alpha}[\frac{1}{2}]\Big(B(S^1/\mathbb{Z}_m)\times\prod_{\alpha=-m+1}^{-m/2-1} BU(j_\alpha)\\ &\;\;\;\quad\quad\amalg B(S^1/\mathbb{Z}_m)\times\prod_{\alpha=-m+1}^{-m/2-1} BU(j_\alpha)\Big)\\
&\cong\Omega^{Spin}_{n-1-2\sum_{\alpha=-m+1}^{-m/2-1}j_\alpha}[\frac{1}{2}](B\Gamma_{\text{first copy}}\amalg B\Gamma_{\text{second copy}})\\
\end{align*}

Here the two components \(B\Gamma_{\text{first copy}}\) and \(B\Gamma_{\text{second copy}}\) in the last formula correspond to the two copies of \(B\Gamma\) in the description of 
the Spin bundle bordism groups considered in Section~\ref{sec:Spin_case}.
We have a connected space on the right hand side of the second formula because in that case we do not have an orientation on the base space of \(S^1\)-vector bundles of the type considered there induced from the orientation of the total space.

For \(H=\mathbb{Z}_m\) with odd \(m\) we get similar isomorphisms. The only difference is that the \(SO\) factors do not appear.
For \(H=S^1\) we have:

\begin{equation}
   \label{eq:5}
  \Omega^{G,S^1}_n[S^1,\prod_{\alpha<0}V_\alpha^{j_\alpha}]_\R[\frac{1}{2}]\cong\Omega^{G}_{n-2\sum_{\alpha}j_\alpha}[\frac{1}{2}](\prod_{\alpha<0} BU(j_\alpha)),
\end{equation}
For more details on the construction of these isomorphisms in the Spin case see \cite{MR1846617}.

In particular, it follows that the maps \(\Omega^{C,G,S^1}_n[\sigma_{h(i)}][\frac{1}{2}]\rightarrow \hat{\Omega}^{G,S^1}_n[\rho_i][\frac{1}{2}]\) are surjective.

Indeed, if \(\rho_i\neq [\mathbb{Z}_m,
\prod_{\alpha=-m}^{-\frac{m}{2}}V_{\alpha}^{j_\alpha}]_\R\) with
\(j_{-\frac{m}{2}}> 0\) and \(m\) even, then a bundle of type \(\rho_i\) has a natural invariant complex structure. Hence, in this case these maps are isomorphisms.

In the case that \(\rho_i=[\mathbb{Z}_m,
\prod_{\alpha=-m}^{-\frac{m}{2}}V_\alpha^{j_{\alpha}}]_\R\) with
\(j_{-\frac{m}{2}}>0\) and \(m\) even, we consider the map
\begin{equation*}
  f:BT\rightarrow B\tilde{G},
\end{equation*}
where \(T\) is a maximal torus of
\(\tilde{G}=(S^1/\mathbb{Z}_{\frac{m}{2}})\times
SO(2j_{\frac{m}{2}})\times \prod_{\alpha\neq \frac{m}{2}}
U(j_{\alpha})\).
This map induces a surjective map in singular homology with
coefficients in \(\mathbb{Z}[\frac{1}{2}]\) because
\(\tilde{G}\) does not have odd torsion.

Since \(H_*(BT;\mathbb{Z}[\frac{1}{2}])\) and
\(H_*(B\tilde{G};\mathbb{Z}[\frac{1}{2}])\) are concentrated in even
degrees, the Atiyah--Hirzebruch spectral sequence for the bordism
groups (with two inverted) for these spaces degenerates at the
\(E^2\)-level.
Hence, it follows that \(f\) induces a surjective map on bordism
groups.

Because this map factors through \(B\tilde{G}_C\), where \(\tilde{G}_C=(S^1/\mathbb{Z}_{\frac{m}{2}})\times
U(j_{\frac{m}{2}})\times \prod_{\alpha\neq \frac{m}{2}}
U(j_{\alpha})\) and \(\tilde{G}\) and \(\tilde{G}_C\) are two-fold covering groups of

\begin{align*}
  (S^1\times_{\mathbb{Z}_m}
SO(2j_{\frac{m}{2}}))\times \prod_{\alpha\neq \frac{m}{2}}
U(j_{\alpha})&&\text{and}&&(S^1\times_{\mathbb{Z}_m}
U(j_{\frac{m}{2}}))\times \prod_{\alpha\neq \frac{m}{2}}
U(j_{\alpha}),
\end{align*}
respectively, it follows from the above isomorphism that \(\Omega^{C,G,S^1}_n[\sigma_{h(i)}][\frac{1}{2}]\rightarrow \hat{\Omega}^{G,S^1}_n[\rho_i][\frac{1}{2}]\) is surjective.

As shown in the proofs of Theorems~\ref{sec:non-semi-free} and~\ref{sec:Spin-case} we have a section \(q_{h(i)}:\image \nu^C_{h(i)}\rightarrow
\Omega^{C,G,S^1}_*[\mathcal{F}_{h(i)}][\frac{1}{2}]\).
Since  \(\Omega^{C,G,S^1}_n[\sigma_{h(i)}][\frac{1}{2}]\rightarrow \hat{\Omega}^{G,S^1}_n[\rho_i][\frac{1}{2}]\) is surjective, in order to show that (\ref{eq:3}) is surjective, it suffices to consider those real slice types for which \(\nu^C_{h(i)}\) is not surjective.
These are of the form
\begin{equation*}
  [S^1;V_{\alpha_1}^{j_1}\times \dots \times V_{\alpha_{k-1}}^{j_{k-1}}\times V_{-m}]_{\mathbb{R}}
\end{equation*}
with \(0>\alpha_1>\alpha_2 > \dots >\alpha_k=-m\).
For these slice types \(\sigma_{h(i)-1}\) is of the form \[[\mathbb{Z}_m,V_{\alpha_1}^{j_1}\times \dots \times V_{\alpha_{k-1}}^{j_{k-1}}].\]

{\bf Case 1:}
At first assume that \(\alpha_1<-\frac{m}{2}\).
Then \(\Omega^{C,G,S^1}[\sigma_{h(i)-1}][\frac{1}{2}]\) and  \(\hat{\Omega}^{G,S^1}[\rho_{i-1}][\frac{1}{2}]\) are isomorphic.
Moreover, \(\partial^C_{h(i)}\) is a injection into the image of \(q_{h(i)-1}\).
In the non-Spin case we have \(\image q_{h(i)-1}=\image \partial_{h(i)}^C\).
In the Spin case we have \(q_{h(i)-1}(B\Gamma_{\text{first
    copy}})=\image \partial_{h(i)}^C\).
Since \(\Omega^{C,G,S^1}[\sigma_{h(i)}][\frac{1}{2}]\) and
\(\hat{\Omega}^{G,S^1}[\rho_{i}][\frac{1}{2}]\) are isomorphic, it follows
that \(\partial_i\) is also injective.
Therefore in this case the maps  \(\nu_{h(i)}^{C}\) and \(\nu_i\) are the zero maps.
Hence we have shown that the composition (\ref{eq:3}) is surjective.

{\bf Case 2:}
Next assume that \(\alpha_1>-\frac{m}{2}\) and \(\alpha_{i_0}>- \frac{m}{2}>\alpha_{i_0+1}\) for some \(i_0\).
Then \(\sigma_{h(i)}\) is equivalent to  the complex slice type \(\tilde{\sigma}=[S^1;V_{-\alpha_1}^{j_1}\times\dots\times V_{-\alpha_{i_0}}^{j_{i_0}}\times V_{\alpha_{i_0+1}}^{j_{i_0+1}}\times\dots\times V_{-m}]\).

In the non-Spin case the map \(\nu^C_{\tilde{\sigma}}\) is surjective for this slice type.
Moreover, the section \(q_{\tilde{\sigma}}\) of \(\nu^C_{\tilde{\sigma}}\) has the following property:
If \(X\in\Omega^{C,SO,S^1}_*[\frac{1}{2}][\tilde{\sigma}]\), then there is no point in \(q_{\tilde{\sigma}}(X)\) with slice type \(\sigma'\) such that \(\sigma_{h(i)}\leq\sigma'<\tilde{\sigma}\).
Therefore in the family of slice types \(\mathcal{F}_{h(i)}\) we can replace \(\sigma_{h(i)}\) by \(\tilde{\sigma}\).
The map \(q_{\tilde{\sigma}}\) is then still a section to the map \(\nu^C_{\tilde{\sigma}}\) with image in \(\Omega_*^{C,SO,S^1}[\mathcal{F}_{h(i)}]\).

In the Spin case we have to look at three cases:
\begin{itemize}
\item \(\alpha_2>-\frac{m}{2}\);
\item \(\alpha_1\) is the only weight with \(\alpha_1>-\frac{m}{2}\) and appears with multiplicity greater than one;
\item \(\alpha_1\) is the only weight with \(\alpha_1>-\frac{m}{2}\) and appears with multiplicity one.
\end{itemize}

As seen in Cases 4 and 5 of the proof of Theorem~\ref{sec:Spin-case},
in the first two cases the map \(\nu^C_{\tilde{\sigma}}\) is
surjective. Moreover, one can check that the section
\(q_{\tilde{\sigma}}\) has the same property as in the non-Spin case.
Therefore one can argue as in the non-Spin case to get the conclusion
in these two cases.

In the last case we can replace \(\sigma_{h(i)}\) by
\(\tilde{\sigma}\) for the same reason as in the first two cases.
Therefore we have the commutative diagram with non-exact rows
\begin{equation*}
  \xymatrix{
    \Omega_{*+1}^{C,G,S^1}[\frac{1}{2}][\tilde{\sigma}] \ar^{\partial^C_{h(i)}}[r]\ar[d]&
    \Omega_*^{C,G,S^1}[\frac{1}{2}][\mathcal{F}_{h(i)-1}] \ar^{\nu^C_{h(i)-1}}[r]\ar[d]&
    \Omega_*^{C,G,S^1}[\frac{1}{2}][\sigma_{h(i)-1}] \ar[d]\\
    \hat{\Omega}_{*+1}^{G,S^1}[\frac{1}{2}][\rho_{i}] \ar^{\partial_i}[r]&
    \hat{\Omega}_*^{G,S^1}[\frac{1}{2}][\mathcal{G}_{i-1}] \ar^{\nu_{i-1}}[r]&
    \hat{\Omega}_*^{G,S^1}[\frac{1}{2}][\rho_{i-1}]
}
\end{equation*}
where the vertical maps on the left and right are isomorphisms because there is no \(\alpha_i\) equal to \(-\frac{m}{2}\) or \(\frac{m}{2}\).
Moreover, in Case 2 of the proof of Theorem \ref{sec:Spin-case} it was
shown that the restriction of the upper right map to the image of the
upper left map is injective.
Therefore it follows that the kernel of the lower left map is
contained in the image of the kernel of the upper left map under the
isomorphism on the left. 

From this it follows by diagram chasing that the map
\(\Omega_*^{C,G,S^1}[\mathcal{F}_{h(i)}]\rightarrow
\Omega_*^{G,S^1}[\mathcal{G}_i]\) is surjective.
Hence also the composition (\ref{eq:3}) is surjective.

{\bf Case 3:}
Next assume \(\alpha_1>-\frac{m}{2}\) and \(\alpha_{i_0}=-\frac{m}{2}\). 
Then \(\sigma_i\) is equivalent to the complex slice type \({\tilde{\sigma}}=[S^1;V_{-\alpha_1}^{j_1}\times V_{-\alpha_{i_0}}^{j_{i_0}}\times \prod_{i\neq 1,i_0} V_{\alpha_i}^{j_i}]\).
As shown in Case 5 of the proof of Theorem~\ref{sec:Spin-case}, for this slice type \(\nu^C_{\tilde{\sigma}}\) is surjective in the Spin case.
In the non-Spin case this map is surjective because \(-\alpha_1>0\).
Moreover, the section \(q_{\tilde{\sigma}}\) to \(\nu^C_{\tilde{\sigma}}\) has the following property:
If \(X\in\Omega^{C,G,S^1}_*[\frac{1}{2}][\tilde{\sigma}]\), then there is no point in \(q_{\tilde{\sigma}}(X)\) with slice type \(\sigma'\) such that \(\sigma_{h(i)}\leq\sigma'<\tilde{\sigma}\).
Therefore we can argue as in the previous cases to get the conclusion.

{\bf Case 4:}
Finally assume that \(\alpha_1=-\frac{m}{2}\).
Then we have a commutative diagram:
\begin{equation*}
  \xymatrix{
    \hat{\Omega}^{G,S^1}[\frac{1}{2}][\rho_i]\ar[r]^-{\cong}\ar[ddd]^{\nu_{i-1}\circ\partial_i}
    & \Omega_*[\frac{1}{2}]\left(B\left((S^1/\mathbb{Z}_{m})\times U(j_1)\right)\right)\ar[d] & \Omega_*[\frac{1}{2}]\left(B\left((S^1/\mathbb{Z}_{\frac{m}{2}})\times U(j_1)\right)\right)\ar[l]_-\cong \ar[d]^{f_1}\\
    &\Omega_*[\frac{1}{2}]\left(B(S^1\times_{\mathbb{Z}_{m}} U(j_1))\right)\ar[d] & \Omega_*[\frac{1}{2}]\left(B(S^1\times_{\mathbb{Z}_{\frac{m}{2}}} U(j_1))\right)\ar[l]_-\cong \ar[d]^{f_2}\\
    &\Omega_*[\frac{1}{2}]\left(B(S^1\times_{\mathbb{Z}_{m}} SO(2j_1))\right) \ar[d] & \Omega_*[\frac{1}{2}]\left(B(S^1\times_{\mathbb{Z}_{\frac{m}{2}}} SO(2j_1))\right)\ar[l]_-\cong \ar[d]^{f_3}\\
    \hat{\Omega}^{G,S^1}[\frac{1}{2}][\rho_{i-1}]\ar[r]^-{\cong}
    &\Omega_*[\frac{1}{2}]\left(B\left((S^1/\mathbb{Z}_{\frac{m}{2}})\times_{\mathbb{Z}_2} SO(2j_1)\right)\right) & \Omega_*[\frac{1}{2}]\left(B\left((S^1/\mathbb{Z}_{\frac{m}{2}})\times SO(2j_1)\right)\right)\ar[l]_-\cong}
\end{equation*}

Here we have omitted the \(BU(j_i)\) factors with \(i>1\) in the middle and right hand column. Note here that \(\rho_{i-1}\) has finite isotropy group of even order.

Moreover, the horizontal maps on the left are the isomorphisms (\ref{eq:5}) and (\ref{sec:s1-manifolds-not-5}).
The horizontal maps on the right are induced by two-fold coverings of the respective groups.

The vertical maps on the right are as follows:

The map \(f_1\) maps a pair \((X,Y)\)  to \((X,X\otimes Y)\), where \(X\) is a line bundle over some manifold \(M\) and \(Y\) is an \(j_1\)-dimensional complex vector bundle over \(M\).

The  map \(f_2\) is the map which forgets the complex structure on \(Y\).
The kernel of this map can be described as follows.

There are natural actions of the Weyl groups \(W(U(j_1))\) and
\(W(SO(2j_1))\) on \(\Omega_*(BT^{j_1})\), where
\(S^1\times\dots\times S^1=T^{j_1}\subset U(j_1)\subset SO(2j_1)\) is
a maximal torus.
Note that \(W(SO(2j_1))\) can be identified with a semi-direct product
\(\mathbb{Z}_2^{j_1-1}\rtimes S_{j_1}\).
Moreover, \(W(U(j_1))\) can be identified with the permutation subgroup \(S_{j_1}\).
 An element of
\(\mathbb{Z}_2^{j_1-1}\) acts on \(H^2(BT)=H^2(BS^1)\oplus\dots\oplus
H^2(BS^1)\) via multiplication by \(-1\) on an even number of
summands.
An element of \(S_{j_1}\) acts by permuting the summands.
It follows from an inspection of the relevant Atiyah--Hirzebruch spectral sequences that the kernels of the natural surjective maps
\[\Omega_*[\frac{1}{2}]\left(B(S^1\times_{\mathbb{Z}_{\frac{m}{2}}} T^{j_1})\right)\rightarrow\Omega_*[\frac{1}{2}]\left(B(S^1\times_{\mathbb{Z}_{\frac{m}{2}}} U(j_1))\right)\]
and
\[\Omega_*[\frac{1}{2}]\left(B(S^1\times_{\mathbb{Z}_{\frac{m}{2}}} T^{j_1})\right)\rightarrow\Omega_*[\frac{1}{2}]\left(B(S^1\times_{\mathbb{Z}_{\frac{m}{2}}} SO(2j_1))\right)\]
are generated by elements of the form \((X,Z-\gamma Z)\) with \(\gamma\in W(U(j_1))\) or \(\gamma\in W(SO(2j_1))\), respectively.
Therefore the kernel of the map
\begin{equation*}
  \Omega_*[\frac{1}{2}]\left(B(S^1\times_{\mathbb{Z}_{\frac{m}{2}}} U(j_1))\right)\rightarrow\Omega_*[\frac{1}{2}]\left(B(S^1\times_{\mathbb{Z}_{\frac{m}{2}}} SO(2j_1))\right)
\end{equation*}
is generated by elements of the form \((X,Z-\gamma Z)\) with \(\gamma\in W(SO(2j_1))/W(U(j_1))=\mathbb{Z}_2^{j_1-1}\).

The third map \(f_3\) is induced by the isomorphism \(S^1\times_{\mathbb{Z}_{\frac{m}{2}}} SO(2j_1)\rightarrow S^1/\mathbb{Z}_{\frac{m}{2}} \times SO(2j_1)\).
This isomorphism exists since \(\mathbb{Z}_{\frac{m}{2}}\) acts trivially on \(SO(2j_1)\).

Since the horizontal maps on the right are induced by two-fold coverings, it follows that, if \(j_1>1\), the kernel of the composition of the vertical maps in the middle is generated by bundles of the form 
\begin{equation*}
  (X,Z-(Z_1\oplus \bar{Z}_2\otimes \bar{X})),
\end{equation*}
where \(X\) and \(Z\) are complex vector bundles over the same base manifold with dimension one and \(j_1\) respectively.  Moreover, we have \(Z=Z_1\oplus Z_2\) with \(\gamma Z_1=Z_1\) and \(\gamma Z_2=\bar{Z}_2\), where \(\gamma\) is an element of \(W(SO(2j_1))/W(U(j_1))\)  and \(\bar{X}\), \(\bar{Z}_2\) denote the conjugated bundles of \(X\) and \(Z_2\), respectively. 

If \(j_1=1\) then this map is injective. Therefore in the following we will assume that \(j_1>1\).
In this case a bundle as above is the image under \(\nu_i\) of 
\begin{equation*}
  \widetilde{\mathbb{C}TP}(Y\oplus Z_1\oplus \C;\;X\oplus Z_2 \oplus \C),
\end{equation*}
where \(Z=Z_1\oplus Z_2\), X is a complex line bundle and \(Y\) is a complex vector bundle induced by the projection \(BS^1\times BU(j_1)\times \prod_{i>1}BU(j_i)\rightarrow \prod_{i>1} BU(j_i)\).
Note that, since \(W(SO(2j_1))\) acts on \(H^2(BT)\) in the
way described above, \(Z_2\) is always even-dimensional.
Therefore, the above manifold is Spin if the involved bundles \(X,Y,Z_1,Z_2\) are Spin bundles and the base space is a Spin manifold.
This is the case for our generators of \(\Omega^{\text{Spin},S^1}_*[\sigma_i][\frac{1}{2}]\) considered in Section~\ref{sec:Spin_case}. 
Hence, Proposition~\ref{sec:s1-manifolds-not} is proved.
\end{proof}

\begin{lemma}
\label{sec:s1-manifolds-not-3}
  For \(\mathcal{F}=\mathcal{AE}\) or \(\mathcal{F}=\mathcal{AE}-\{[S^1,W];\dim_\C W=1\}\), the natural map
\[\hat{\Omega}_n^{SO,S^1}[\frac{1}{2}][\mathcal{F}]\rightarrow \Omega_n^{SO,S^1}[\frac{1}{2}][\mathcal{F}]\]
is an isomorphism.
\end{lemma}
\begin{proof}
Let \(\mathcal{G}_i\) be the same families of real slice types as in the proof of Proposition~\ref{sec:s1-manifolds-not}.
We will show by induction on \(i\) that the lemma is true if one replaces \(\mathcal{F}\) by \(\mathcal{G}_i\).
The lemma then follows because there are only finitely many slice types in a compact \(S^1\)-manifold.

We have a commutative diagram with exact rows as follows:
\begin{equation*}
  \xymatrix{\dots\ar[r]&
    \hat{\Omega}^{SO,S^1}_*[\frac{1}{2}][\mathcal{G}_{i-1}]\ar[r]\ar[d]&
    \hat{\Omega}^{SO,S^1}_*[\frac{1}{2}][\mathcal{G}_{i}]\ar[r]\ar[d]&
    \hat{\Omega}^{SO,S^1}_*[\frac{1}{2}][\rho_{i}]\ar[r]\ar^f[d]&\dots\\
\dots\ar[r]&
    \Omega^{SO,S^1}_*[\frac{1}{2}][\mathcal{G}_{i-1}]\ar[r]&
    \Omega^{SO,S^1}_*[\frac{1}{2}][\mathcal{G}_i]\ar[r]&
\Omega^{SO,S^1}_*[\frac{1}{2}][\rho_i]\ar[r]&\dots
}
\end{equation*}
Here \(\Omega^{SO,S^1}_*[\frac{1}{2}][\rho_i]\) denotes the bundle bordism group of \(S^1\)-vector bundles of type \(\rho_i\) with a prescribed orientation on the total space. But we do not have an orientability condition on the base.
The vertical maps here are induced by forgetting that the singular strata are orientable.

By the induction hypothesis and the five-lemma, it suffices to show that the map \(f\) is an isomorphism.
Surjectivity of \(f\) is essentially Lemma 3 in Ossa's \cite{MR0263093}. Therefore, we only have to show that \(f\) is injective.
To do so we slightly modify Ossa's argument.
Let \(E\rightarrow M\) be an \(S^1\)-vector bundle of type \(\rho_i\) with oriented total space and orientable base space.
We will show that if \(E\) bounds an \(S^1\)-vector bundle \(F\rightarrow W\) of type \(\rho_i\) with oriented total space and not necessarily orientable base space \(W\), then there is an \(S^1\)-vector bundle \(F'\rightarrow W'\) of type \(\rho_i\) with oriented total space and orientable base space such that \(2E\) bounds \(W'\).

To do so we can assume that \(W\) is non-orientable. Then the isotropy group of \(\rho_i\) is necessarily finite of even order. Therefore \(W/S^1\) is a non-orientable manifold with orientable boundary.

Hence, we can find a closed codimension-one manifold \(V\subset W\) such that
\begin{enumerate}
\item The \(\mathbb{Z}_2\)-Poincare dual of \(V\) is given by \(w_1(TW)\).
\item \(V\) is contained in the interior of \(W\).
\item \(V\) is \(S^1\)-invariant, because \(TW\) is stably isomorphic to the pullback of \(T(W/S^1)\) along the orbit map.
\end{enumerate}

Then \(W-V\) is orientable.
Denote by \(N(V,W)\) the normal bundle of \(V\) in \(W\) and by \(U\) a tubular neighborhood of \(V\) in \(W\).
Then \(F|_U\) is equivariantly diffeomorphic to \(F|_V\oplus N(V,W)\).
Therefore, the antipodal map on \(N(V,W)\) induces a orientation reversing involution \(T\) on \(F|_U\)  which is actually a bundle map.

Now let \(W'=(W-U)\cup_{T|_{\partial U}}(W-U)\)
and \(F'=(F|_{W-U})\cup_{T|_{F|_{\partial U}}} (F|_{W-U})\).
Then \(2E\) bounds \(F'\) and \(W'\) is orientable.
\end{proof}

\begin{remark}
\label{sec:s1-manifolds-not-2}
  If, in the situation of the proof of the above lemma, we define
\[X=F\times[0,1]\cup_{(T\times \id_{[0,1]})|_{F|_{U}\times \{0\}}}F\times [0,1],\]
then \(X\) is a \(S^1\)-vector bundle of type \(\rho_i\) over an \(S^1\)-manifold with boundary and
\[\partial X = F' \cup ((E\amalg E)\times [0,1]) \cup (F\amalg F).\]
\end{remark}

For later use we also state the following lemma here.

\begin{lemma}
\label{sec:s1-manifolds-not-4}
  Let \(M,N\) be oriented \(S^1\)-manifolds, such that all singular strata in \(N\) are orientable.
  Let also \(W\) be an oriented \(S^1\)-bordism between \(M\) and \(N\).
  Then for some \(0\leq l \leq (\dim M-1)/2\) there is an oriented \(S^1\)-bordism \(W'\) between the disjoint union of \(2^l\) copies of \(M\) and the disjoint union of \(2^l\) copies of \(N\), such that all singular strata in \(W'\), which do not meet the copies of \(M\), are orientable.
\end{lemma}
\begin{proof}
  We prove the following claim by induction on \(k\).

{\bf Claim:} For every \(k\in \mathbb{N}\) there is an oriented \(S^1\)-bordism \(W_k\) between the disjoint union of \(2^k\) copies of \(M\) and the disjoint union of \(2^k\) copies of \(N\) such that all singular strata of \(W_k\) of dimension less than or equal to \(2k+1\) which do not meet the copies of \(M\) are orientable.

The lemma follows from this claim in the case \(k=[(\dim M-1)/2]\).

The claim is true for \(k=0\) because all manifolds of dimension less than two are orientable.
Therefore, assume that \(W_k\) as in the claim is given. We will show the existence of \(W_{k+1}\).

Let \(V\subset W_k\) be a non-orientable closed singular stratum of dimension \(2k+2\) or \(2k+3\). Note that, by orientability of \(W_k\), the singular strata of \(W_k\) have even codimension. Hence, only one of these two cases can appear.
Moreover, assume that \(V\) does not meet \(M\).

We have to distinguish between two cases:
\begin{enumerate}
\item \(V\) is a minimal stratum in \(W_k\).
\item \(V\) is not a minimal stratum in \(W_k\). In this case all the strata contained properly in \(V\) are orientable by the induction hypothesis.
\end{enumerate}

In the first case, let \(N(V,W)\) be the normal bundle of \(V\) in \(W\).
Because \(V\cap \partial W\) is orientable, we can use the argument from the proof of Lemma~\ref{sec:s1-manifolds-not-3} and Remark~\ref{sec:s1-manifolds-not-2} to find a bundle bordism \(X\) between \(N(V,W)\amalg N(V,W)\) and a \(S^1\)-vector bundle \(F\) with orientable base which induces the trivial bordism on boundaries.
Let \(\tilde{X}\) be the induced bordism between the associated sphere bundles.
Denote by \(D_F\) the disc bundle associated to \(F\) and by \(U\) an open tubular neighborhood of \(V\) in \(W\).

Then 
\begin{equation*}
  W_{k+1}=((W_k-U)\amalg (W_k-U))\cup \tilde{X}\cup D_F
\end{equation*}
is the required bordism in the case that \(V\) is a minimal stratum.

If \(V\) is not a minimal stratum, then we can cut out all the singular strata which are contained in \(V\) from \(W_k\) to get a bordism \(\tilde{W}_k\) between \(2^k\) disjoint copies of \(M\) and a manifold \(N'\) all of whose singular strata are orientable, such that \(V\) is a minimal stratum.
Then we can apply the construction of the previous case to \(\tilde{W}_k\) to get a bordism between \(2^{k+1}\) disjoint copies of \(M\) and \(2\) copies of \(N'\).
By gluing back the singular strata which we removed in the first step, we get a bordism between \(2^{k+1}\) disjoint copies of \(M\) and \(2^{k+1}\) disjoint copies of \(N\). 
\end{proof}

\section{Resolving singularities}
\label{sec:resolv}

Now we turn to the construction of invariant metrics of positive scalar curvature on \(S^1\)-manifolds.
In this and the next two sections we prepare the proof of Theorems~\ref{cha:discussion},~\ref{sec:introduction-1} and~\ref{sec:introduction-2} which will be carried out in Section~\ref{sec:atiyah_hirzebruch}.

In this section we discuss a general construction for invariant metrics of positive scalar curvature on certain \(S^1\)-manifolds.
These \(S^1\)-manifolds are not semi-free and have \(S^1\)-fixed points.
The construction is a generalization of a construction from \cite{MR2376283}, where it was done for fixed point free \(S^1\)-manifolds.

For this construction we need the following technical 

\begin{definition}[{cf. \cite[Definition 21]{MR2376283}}]
  Let \(M\) be a manifold (possibly with boundary) with an action of a torus \(T\) and a \(T\)-invariant Riemannian metric \(g\).
  \begin{enumerate}
  \item Assume \(T=S^1\) and \(M^{S^1}=\emptyset\). We say that \(g\) is \emph{scaled} if the vector field generated by the \(S^1\)-action is of constant length.
  \item  Let \(H\subset T\) be a closed subgroup and \(F\) a component of \(M^H\)  of codimension two in \(M\).
 We call \(g\) \emph{normally symmetric in codimension two at \(F\)} if the following holds:
    There is an \(T\)-invariant tubular neighborhood \(N_F\subset M\) of \(F\) together with an isometric \(S^1\)-action \(\sigma_F\) on \(N_F\) which commutes with the \(T\)-action and has fixed point  set \(F\).

We call a metric \(g\) \emph{normally symmetric in codimension two} if it is normally symmetric in codimension-two at all closed codimension two singular strata.
  \end{enumerate} 
\end{definition}

The next lemma is a generalization of Lemma 23 of \cite{MR2376283} to \(S^1\)-manifolds with fixed points.

\begin{lemma}
\label{sec:resolv-sing}
  Let \(M\) be a compact \(S^1\)-manifold, \(\dim M \geq 3\), such that \(\codim M^{S^1}\geq 4\), and \(V\) a small tubular neighborhood of \(M^{S^1}\).
  If \(M\) admits an invariant metric \(g\) of positive scalar curvature, then there is another metric \(\tilde{g}\) of positive scalar curvature on \(M\) such that  \(\tilde{g}\) is scaled on \(M-V\).
  If \(g\) is normally symmetric in codimension two at some codimension two singular stratum of \(M-V\), then the same can be assumed for \(\tilde{g}\).
\end{lemma}
\begin{proof}
  Let \(V'\subset V\) be a slightly smaller tubular neighborhood of \(M^{S^1}\) in \(M\).
  By an equivariant version of Theorem 2' of \cite{MR962295} , there is a metric \(g_1\) of positive scalar curvature on \(M-V'\), such that in a collar of \(SM^{S^1}=\partial V'\), \(g_1\) is of the form \(h_1 + ds^2\), where \(h_1\) is the restriction of a metric of the form \(g|_{N(M^{S^1},M)} +_H h\) to the normal sphere bundle of \(M^{S^1}\) of sufficiently small radius \(\delta\), where \(h\) is the metric induced by \(g\) on \(M^{S^1}\) and \(H\) is the normal connection.
  Here \(ds^2\) denotes the standard metric on \([0,1]\).
  Moreover, \(g_1\) coincides with \(g\) on the complement of \(V\).

  We have \(N(M^{S^1},M)=\bigoplus_{i> 0} E_i\), where \(E_i\) is a complex \(S^1\)-bundle such that each \(z\in S^1\) acts by multiplication with \(z^i\) on \(E_i\).
  Since \(g\) is \(S^1\)-invariant, the \(E_i\) are orthogonal to each other with respect to \(g|_{N(M^{S^1},M)}\).
  Using O'Neill's formula for the fibration \(S_\delta \hookrightarrow SM^{S^1} \rightarrow M^{S^1}\), one sees that there is a \(\delta'>0\) such that \(h_1\) is isotopic via \(S^1\)-invariant metrics of positive scalar curvature to the restriction \(h_2\) of \(\delta'(\sum_{i>0} \frac{1}{i^2} g|_{E_i}) +_H h\) to \(SM^{S^1}\). (The isotopy is given by convex combination of the metrics \(\sum_{i>0} \frac{1}{i^2} g|_{E_i} +_H h\) and \(g|_{N(M^{S^1},M)} +_H h\) and rescaling the fibers.)
  Note that \(h_2\) is scaled.
 
  Therefore by Lemma 3 of \cite{MR577131}, there is a invariant metric \(g_2\) of positive scalar curvature on \(M-V'\) such that \(g_2\) restricted to a collar of \(\partial V'\) is of the form \(h_2 + ds^2\) and \(g_2\) restricted to \(M-V\) is equal to \(g\).
  
  Now let \(X:M-V'\rightarrow T(M-V')\) be the vector field generated by the \(S^1\)-action.
  Because there are no fixed points in \(M-V'\), \(X\) is nowhere zero.
  Denote by \(\mathcal{V}\) the one-dimensional subbundle of \(T(M-V')\) generated by \(X\) and \(\mathcal{H}\) its orthogonal complement with respect to \(g_2\).
  For \(p\in M-V'\), define \(f(p)=\|X(p)\|_{g_2}\).
  Note that \(f\) is constant in a small neighborhood of \(\partial V'\).

  Next we describe some local Riemannian submersions which are useful to show that our scaled metrics have positive scalar curvature.
  Let \(S^1\times_H D(W)\subset M-V'\) be a tube around an orbit in \(M-V'\).
  We pull back the metric \(g_2\) via the covering \(S^1\times D(W)\rightarrow S^1\times_H D(W)\).
  This yields a metric which is invariant under the free circle action on the first factor.
  Let \(g_2'\) be the induced quotient metric on \(D(W)\).
  Then the argument from the proof of Theorem C of \cite{berard83:_scalar}, shows that the metric
  \begin{equation*}
    f^{\frac{2}{\dim M -2}} \cdot g_2'
  \end{equation*}
  on \(D(W)\) has positive scalar curvature.

  Now let \(dt^2\) be the metric on \(\mathcal{V}\) for which \(X\) has constant length one.
  By O'Neill's formula applied in the above local fibration, the scalar curvature of the metric
  \begin{equation*}
    g_{\epsilon,3}=(\epsilon^2 \cdot dt^2) + (f^{\frac{2}{\dim M -2}} \cdot g_2|_{\mathcal{H}})
  \end{equation*}
  on \(M-V'\) is given by
  \begin{equation}
\label{eq:2}
    \scal_{g_{\epsilon,3}}= \scal_{f^{\frac{2}{\dim M -2}}g_2'} - \epsilon^2 \|A\|_{g_{1,3}},
  \end{equation}
  where \(A\) is the \(A\)-tensor for the connection induced by \(g_{1,3}\) in the fibration \(S^1\times D(W)\rightarrow D(W)\).
Since \(M\) is compact it follows that there is an \(\epsilon_0>0\) such that for all \(\epsilon_0>\epsilon>0\) the metric \(g_{\epsilon,3}\) has positive scalar curvature.

  Moreover, since the restriction of \(g_2\) to a collar of \(\partial V'\) was a product metric and \(f\) is constant in this neighborhood \(g_{\epsilon,3}\) is also a product metric on this collar.

Next we show that \(g_{\epsilon,3}|_{\partial V}\) and \(h_2\) are isotopic via invariant metrics of positive scalar curvature.
We have
\begin{align*}
    g_{\epsilon,3}|_{\partial V'}&=(\epsilon^2 \cdot dt^2) + (f^{\frac{2}{\dim M -2}} \cdot g_2|_{\mathcal{H}\cap T\partial V'})\\
    h_2&=f^{-\frac{2}{\dim M -2}}((f^2 f^{\frac{2}{\dim M -2}} \cdot dt^2) + (f^{\frac{2}{\dim M -2}} \cdot g_2|_{\mathcal{H}\cap T\partial V'})),
\end{align*}
where \(f>0\) is constant.

In other words, \(h_2\) is equal to the metric \(g_{f\cdot f^{\frac{1}{\dim M -2}},3}\)up to scaling.
Hence, it follows from formula (\ref{eq:2}) that the metrics \(h_2\) and \(g_{\epsilon,3}|_{\partial V}\) are isotopic via invariant metrics of positive scalar curvature because both metrics have positive scalar curvature.
One only has to increase or decrease the parameter \(\epsilon\) and then rescale the metric.
Since \(h_2\) is isotopic to \(h_1\) it follows from Lemma 3 of \cite{MR577131}, that there is an invariant metric of positive scalar curvature on \(M\) whose restriction to \(M-V'\) is \(g_{\epsilon,3}\).

The remark about the normally symmetric metrics of positive scalar curvature can be seen as follows.
Because the local \(S^1\)-actions \(\sigma_F\)  commute with the global \(S^1\)-action, they respect the decomposition \(T(M-V)=\mathcal{H}\oplus \mathcal{V}\).
Therefore the new metric is invariant under these actions.
This completes the proof.
\end{proof}

\begin{remark}
\label{sec:resolv-sing-1}
Note that every \(g_{\epsilon,3}\), \(0<\epsilon<\epsilon_0\), can be extended to an invariant metric of positive scalar curvature on \(M\) and that the restrictions of all these metrics to \(\mathcal{H}\) are the same.
This shows that the metric \(g_{\epsilon,3}\) can be scaled down on \(\mathcal{V}\) without affecting the restriction of the metric to \(\mathcal{H}\) and the fact that the \(g_{\epsilon,3}\) can be extended to a metric of positive scalar curvature on \(M\).
\end{remark}

Next we describe a resolution of singularities for singular strata of codimension two from \cite{MR2376283}.
Let \(M\) be a \(S^1\)-manifold of dimension \(n\geq 3\) and
\begin{equation*}
  \phi: S^1\times_H(S^{n-3}\times D(W))\hookrightarrow M
\end{equation*}
be an \(S^1\)-equivariant embedding where \(H\) is a finite subgroup of \(S^1\) and \(W\) is an one-dimensional unitary effective \(H\)-representation.
Here \(H\) acts trivially on \(S^{n-3}\).

Since \(S(W)/H\) can be identified with \(S^1\), the \(S^1\)-principal bundle
\begin{equation*}
  S^1\hookrightarrow S^1\times_H(S^{n-3}\times S(W))\rightarrow S^{n-3}\times S(W)/H
\end{equation*}
is trivial.
Choose a trivialization
\begin{equation*}
  \chi: S^1\times_H(S^{n-3}\times S(W))\rightarrow S^1\times S^{n-3}\times S(W)/H
\end{equation*}
and consider \(S(W)/H\) as the boundary of \(D^2\).
Then we can glue the free \(S^1\)-manifold \(S^1\times S^{n-3}\times D^2\) to \(M-\image \phi\) to get a new \(S^1\)-manifold \(M'\).
We say that \(M'\) is obtained from \(M\) by resolving the singular stratum \(\phi(S^1\times_H(S^{n-3}\times \{0\}))\).

Now we can state the following generalization of Theorem 25 of \cite{MR2376283}.

\begin{theorem}
\label{sec:resolv-sing-2}
  Let \(M\) be a closed \(S^1\)-manifold of dimension \(n\geq 3\) such that \(\codim M^{S^1}\geq 4\) and \(H\subset S^1\) a finite subgroup.
  Let \(V\) be an invariant tubular neighborhood of \(M^{S^1}\).
  Moreover, let
  \begin{equation*}
    \phi: S^1\times_H(S^{n-3}\times D(W))\hookrightarrow M - V
  \end{equation*}
  be an \(S^1\)-equivariant embedding where \(W\) is a unitary effective \(H\)-representation of dimension one.
  Let \(M'\) be obtained from \(M\) by resolving the singular stratum \(\phi(S^1\times_H(S^{n-3}\times \{0\}))\subset M\).
  If \(M\) admits an invariant metric of positive scalar curvature which is normally symmetric in codimension two at this singular stratum, then also \(M'\) admits an invariant metric of positive scalar curvature.
\end{theorem}

Since the proof of Theorem 25 of \cite{MR2376283} is mainly a local argument in a neighborhood of the singular stratum which is resolved with some down-scaling at the end, its proof is also valid in the case where \(M\) has fixed point components of codimension at least four, by Lemma \ref{sec:resolv-sing} and Remark~\ref{sec:resolv-sing-1}.
Therefore we have the above theorem.

Moreover, note that, by the local nature of the construction, if the original metric on \(M\) is normally symmetric in codimension two at some singular stratum \(F\) which is disjoint from the singular stratum which is resolved, then the new metric on \(M'\) is also normally symmetric in codimension two at \(F\).

The next step is the following generalization of Lemma 24 of \cite{MR2376283}:

\begin{lemma}
\label{sec:more-results-1}
Let \(Z\) be a compact orientable \(S^1\)-bordism between \(S^1\)-manifolds \(X\) and \(Y\).
Assume that \(X\) carries an invariant metric of positive scalar curvature which is normally symmetric in codimension \(2\).
If \(Z\) admits a decomposition into special \(S^1\)-handles of codimension at least \(3\), then \(Y\) carries an invariant metric of positive scalar curvature which is normally symmetric in codimension \(2\) at all codimension-two singular strata which are boundaries of orientable components of \(Z^H\) with \(H\subset S^1\) a closed subgroup.
\end{lemma}
\begin{proof}
  First recall that a special \(S^1\)-handle is an \(S^1\)-handle of
  the form
  \begin{equation*}
    S^1\times_H(D^{d}\times D(W)),
  \end{equation*}
where \(H\) is a subgroup of \(S^1\), \(W\) is an orthogonal
\(H\)-representation, \(D(W)\) is the unit disc in \(W\) and \(H\)
acts trivially on the \(d\)-dimensional disc \(D^d\).
Here the codimension of the handle is given by \(\dim W\).

  Therefore as in the proof of Lemma 24 in \cite{MR2376283} we may assume that \(Y\) can be constructed from \(X\) by equivariant surgery on \(S=S^1\times_H (S^{d-1}\times D(W))\).

  First assume that \(d=0\).
  Then we have \(S=\emptyset\).
  Hence the surgery on \(S\) produces a new component \(S^{1}\times_H
  (D^0\times S(W))\) of \(Y\).
  Since \(Z\) is orientable, there is a homomorphism \(\phi: H\rightarrow SO(n)\) which corresponds to the \(H\)-representation \(W\).
  As a subgroup of \(S^1\), \(H\) has a dense cyclic subgroup.
  Therefore \(\phi(H)\) is contained in a maximal torus of \(SO(n)\).
  Hence \(W\) is isomorphic to \(\R^k \oplus \bigoplus_i W_i\), where \(\R^k\) denotes the trivial \(H\)-representation and the \(W_i\) are complex one-dimensional \(H\)-representations.

  If \(K\subset S^1\) is a subgroup such that \(S^{1}\times_H (D^0\times S(W))^K\neq \emptyset\), then \(K\) is contained in \(H\).
  If \(S^K\) has codimension two in \(S\), then there is exactly one \(W_{i_0}\) such that \(K\) acts non-trivially on \(W_{i_0}\).
  Therefore the \(S^1\)-action on \(S^{1}\times_H (D^0\times S(W))\), which is induced by complex multiplication on \(W_{i_0}\) commutes with the original \(S^1\)-action and leaves all connection metrics induced from the round metric on \(S(W)\) invariant.
  Since there are such connection metrics with positive scalar curvature on \(S^{1}\times_H (D^0\times S(W))\) the theorem follows in this case.

  Now assume that \(d\geq1\). Then \(S\) is non-empty. The proof
  proceeds as in Hanke's paper.
  This is done as follows. As in Hanke's paper we may assume that there is a \(K\subset H\) and a component \(F\subset X^K\) of codimension \(2\) with
  \begin{equation*}
    S^1\times_H(S^{d-1}\times \{0\})\subset F.
  \end{equation*}
  Moreover, the extra symmetry \(\sigma\) induces an orthogonal action
  \(\sigma\) of \(S^1\) on the third factor of
  \(S^1\times_H(S^{d-1}\times D(W))\).
  This action extends to an orthogonal action on the third factor of
  the handle \(S^1\times_H(S^{d-1}\times D(W))\).
  
  Indeed, if \(d>1\), then \(S^{d-1}\) is connected and hence the
  action extends.

  If \(d=1\), then \(S^{d-1}=\{\pm 1\}\) has two components.
  And in principle the \(H\times S^1\)-representations \(\sigma_\pm\)  on \(\{\pm
  1\}\times W\) might have different isomorphism types.
  If this happens the \(S^1\)-action on \(S^1\times_H(S^{d-1}\times
  D(W)\) cannot be extended to an action on the handle \(S^1\times_H(D^d\times D(W))\).

Therefore we have to rule out this case if the component of \(Z^K\) which contains \(F\) is orientable.
This is done as follows.
  We have that
  \(F'=(F\times I)\cup S^1\times_H([-1,+1]\times D(W)^K)\subset Z^K\)
  is orientable.
  Therefore the structure group of \(N(F',Z)\) is \(SO(2)=U(1)\) and
  the \(S^1\)-action \(\sigma\) extends to an action which is defined
  on a neighborhood of \(F'\).
  
  Since the actions \(\sigma_\pm\) are restrictions of this action
  their isomorphism types coincide.
  Therefore the actions \(\sigma_\pm\) extend to an orthogonal action
  on \(S^1\times_H([-1,1]\times D(W))\) with fixed point set
  \(S^1\times_H([-1,1]\times D(W)^K)\).

  Now one can construct 
 an invariant metric of positive scalar curvature which is normally
 symmetric in codimension \(2\) (at those codimension two singular strata which are parts of boundaries of orientable strata in \(Z\)) on \(Y\) as in the proof of Lemma 24
 of \cite{MR2376283}.
\end{proof}

Let \(p:B\rightarrow BO\) be a fibration.
A \(B\)-structure on a manifold \(M\) is a lift \(\hat{\nu}: M \rightarrow B\) of the classifying map \(\nu\rightarrow BO\) of the stable normal bundle of \(M\).
We denote by \(\Omega^B_{n,S^1}\) the equivariant bordism group of \(S^1\)-manifolds with \(B\)-structures on maximal strata, i.e. an element of \(\Omega^B_{n,S^1}\) is represented by a pair \((M,\hat{\nu})\), where \(M\) is an \(n\)-dimensional \(S^1\)-manifold and \(\hat{\nu}:M_{\text{max}}\rightarrow B\) is a \(B\)-structure.
Moreover, such a pair represents zero if there is an \(n+1\)-dimensional \(S^1\)-manifold with boundary \(W\) with a \(B\)-structure \(f:W_{\text{max}}\rightarrow B\) on its maximal stratum such that \(\partial W=M\) and \(f|_{M_{\text{max}}}=\hat{\nu}\).

\begin{lemma}
\label{sec:more-results-2}
  Let \(M\) be a connected \(S^1\)-manifold of dimension \(n\geq 6\) with  a \(B\)-structure \(\hat{\nu}:M_{\text{max}}\rightarrow B\) on its maximal stratum such that \(\hat{\nu}\) is a two-equivalence, i.e. \(\hat{\nu}\) induces an isomorphism on \(\pi_1\) and a surjection on \(\pi_2\).
Let \(W\) be a connected equivariant \(B\)-bordism between \(M\) and another \(S^1\)-manifold \(N\).
Then there is a \(B\)-bordism \(W'\) between \(M\) and \(N\) such that \(M_{\mathrm{max}}\hookrightarrow W'_{\mathrm{max}}\) is a two-equivalence and there is a diffeomorphism \(V\rightarrow V'\), where \(V\) and \(V'\) are open neighborhoods of \(\partial W\cup W_{\mathrm{sing}}\) and \(\partial W'\cup W'_{\mathrm{sing}}\), respectively.
\end{lemma}
\begin{proof}
  We first show that surgery on \(W_{\text{max}}\) can be used to construct a \(B\)-bordism \(W_1\) such that the inclusion \(M_{\text{max}}\hookrightarrow W_{1,\text{max}}\) induces an isomorphism on fundamental groups.
  We have the following commutative diagram with exact rows.
  \begin{equation*}
    \xymatrix{ 
      \pi_1(S^1)\ar[d]^\id\ar@/^2pc/[ddd]^\id\ar[r]&\pi_1(M_{\text{max}})\ar[r]\ar[d]\ar@/^3pc/[ddd]^\id&\pi_1(M_{\text{max}}/S^1)\ar[d]\ar@/^3pc/[ddd]^\id\ar[r] & 0\\
      \pi_1(S^1)\ar[r]\ar[dd]^\id&\pi_1(W_{\text{max}})\ar[r]\ar[d]^{f_*}&\pi_1(W_{\text{max}}/S^1)\ar@{.>}[dd]^\phi\ar[r] & 0\\
      &\pi_1(B)\ar[d]^{\hat{\nu}_*^{-1}}\\
      \pi_1(S^1)\ar[r]&\pi_1(M_{\text{max}})\ar[r]&\pi_1(M_{\text{max}}/S^1)\ar[r] & 0
}
  \end{equation*}
Here the dashed map \(\phi\) is induced by the commutativity of the diagram and the universal property of the quotient group.
It follows from an easy diagram chase that \(\ker \phi \cong \ker \hat{\nu}^{-1}_* \circ f_*\).

Let \(c: S^1\rightarrow W_{\text{max}}/S^1\) be an embedding that represents an element of \(\ker \phi\) and does not meet the boundary.
Then there is an equivariant embedding \(c':S^1\times S^1\rightarrow W_{\text{max}}\), where \(S^1\) acts by rotation on the first factor, such that \(c=\pi\circ c'\circ \iota_2\), where \(\iota_2:S^1\rightarrow S^1\times S^1\) is the inclusion of the second factor and \(\pi:W_{\text{max}}\rightarrow W_{\text{max}}/S^1\) is the orbit map.
Since \(c\in \ker \phi\), there is a map \(c'':S^1\rightarrow S^1\) such that \(c'\circ (\iota_2 * (\iota_1\circ c''))\) is contained in the kernel of \(f_*\).
Therefore by precomposing \(c'\) with
\begin{equation*}
  (s,t)\mapsto (s c''(t),t)
\end{equation*}
we may assume that \(c'\circ \iota_2\in \ker \hat{\nu}^{-1}_* \circ f_*\).
Therefore we can do equivariant surgery on \(c'\) to construct a new \(B\)-bordism \(W'\) such that \(\pi_1(W'_{\text{max}}/S^1)=\pi_1(W_{\text{max}}/S^1)/\langle c \rangle\), where \(\langle c \rangle\) denotes the normal subgroup of \(\pi_1(W_{\text{max}}/S^1)\) which is generated by \(c\).

Since \(\pi_1(W_{\text{max}}/S^1)\) and \(\pi_1(M_{\text{max}}/S^1)\) are finitely presentable, it follows that \(\ker \phi\) is finitely generated as a normal subgroup of \(\pi_1(W_{\text{max}}/S^1)\).
Therefore after a finite number of iterations of the above surgery step we may achieve that \(\ker \hat{\nu}^{-1}_* \circ f_* =0\).
This is equivalent to the fact that \(\pi_1(M_{\text{max}})\rightarrow \pi_1(W_{\text{max}})\) is an isomorphism. 

The next step is to make the map \(\pi_2(M_{\text{max}})\rightarrow \pi_2(W_{\text{max}})\) surjective.
We have the following commutative diagram with exact rows.
\begin{equation*}
  \xymatrix{
    0\ar[r]&\pi_2(M_{\text{max}})\ar[r]\ar[d]&\pi_2(M_{\text{max}}/S^1)\ar[r]\ar[d]& \pi_1(S^1)\ar[r]\ar[d]^\cong &\pi_1(M_{\text{max}})\ar[d]^\cong\\
    0\ar[r]&\pi_2(W_{\text{max}})\ar[r]&\pi_2(W_{\text{max}}/S^1)\ar[r]& \pi_1(S^1)\ar[r] &\pi_1(W_{\text{max}})}
\end{equation*}

It follows from an application of the snake lemma that there is an isomorphism
\begin{equation*}
  \pi_2(W_{\text{max}})/\pi_2(M_{\text{max}})\rightarrow \pi_2(W_{\text{max}}/S^1)/\pi_2(M_{\text{max}}/S^1).
\end{equation*}

Let \(c: S^2\rightarrow W_{\text{max}}\) be a representative of a class in \(\pi_2(W_\text{max})/\pi_2(M_\text{max})\) such that \(\pi\circ c\) is an embedding.
Then there is an equivariant embedding \(c':S^1\times S^2\hookrightarrow W_\text{max}\) such that \(c'\circ \iota_2=c\).
Since \(\hat{\nu}_*: \pi_2(M_\text{max})\rightarrow \pi_2(B)\) is surjective we may assume that \(f_*c=0\).
Therefore we can do equivariant surgery on \(c'\) to construct a new \(B\)-bordism \(W'\) such that
\begin{equation*}
  \pi_2(W'_\text{max}/S^1)=\pi_2(W_\text{max}/S^1)/\langle \pi_* c\rangle,
\end{equation*}
where \(\langle \pi_* c \rangle \) denotes the \(\mathbb{Z}[\pi_1(W_\text{max}/S^1)]\)-submodule of \(\pi_2(W_\text{max}/S^1)\) which is generated by \(\pi_*c\).

Since \(\pi_2(W_\text{max})/S^1\) is a finitely generated \(\mathbb{Z}[\pi_1(W_\text{max}/S^1)]\)-module we get after a finite number of repetitions of this step a bordism \(W'\)  for which 
$$\pi_2(W'_{\text{max}})/\pi_2(M_{\text{max}})=0.$$
This completes the proof.
\end{proof}

\begin{remark}
\label{sec:more-results-4}
  If in the situation of the above lemma \(M_{\text{max}}\) is Spin, then \(p:B=B\pi_1(M_\text{max})\times BSpin\rightarrow BO\) and \(\hat{\nu}=f\times s: M_{\text{max}}\rightarrow B\), where \(p\) is the composition of the projection to the second factor with the natural fibration \(BSpin\rightarrow BO\), \(f\) is the classifying map of the universal covering of \(M_{\text{max}}\) and \(s\) is a Spin-structure on \(M_\text{max}\), satisfies the assumptions on \(B\).

  If \(M_{\text{max}}\) is orientable, not Spin with universal covering not Spin, then \(p:B=B\pi_1(M_\text{max})\times BSO\rightarrow BO\) and  \(\hat{\nu}=f\times s: M_{\text{max}}\rightarrow B\), where \(p\) is the composition of the projection to the second factor with the natural fibration \(BSO\rightarrow BO\),   \(f\) is the classifying map of the universal covering of \(M_{\text{max}}\) and \(s\) is an orientation on \(M_\text{max}\), satisfies the assumptions on \(B\).

  If \(M_{\text{max}}\) is orientable, not Spin with universal covering Spin, then it follows that \(w_2(M_{\text{max}})=f^*(\beta)\) for some \(\beta \in H^2(B\pi_1(M_\text{max};\mathbb{Z}/2\mathbb{Z}))\). Let \(Y(\pi_1(M_\text{max}), \beta)\) be the pullback of \(\beta: B\pi_1(M_\text{max})\rightarrow B\mathbb{Z}/2\mathbb{Z}\) and \(w_2: BSO\rightarrow B\mathbb{Z}/2\mathbb{Z}\). Then \(f\times s:M_\text{max}\rightarrow B\pi_1(M_\text{max})\times BSO\) as in the previous case lifts to a map \(\hat{\nu}:M_\text{max}\rightarrow Y(\pi_1(M_\text{max}), \beta)\). If we let \(p: B= Y(\pi_1(M_\text{max}), \beta)\rightarrow BO\) be the composition of the natural fibrations \(Y(\pi_1(M_\text{max}), \beta)\rightarrow BSO\) and \(BSO\rightarrow BO\), then \(\hat{\nu}\) and \(B\) satisfy the assumptions from the above lemma.

For more details see \cite{MR1268192}.
\end{remark}

Now we can prove the following generalization of Theorem 34 of \cite{MR2376283}.

\begin{theorem}
\label{sec:more-results}
  Let \(Z\) be a compact connected oriented \(S^1\)-bordism between closed \(S^1\)-manifolds \(X\) and \(Y\).
  Assume that for all subgroups \(H\subset S^1\) all components of codimension two of \(Z^H\) which do not meet \(Y\) are orientable and that the following holds
  \begin{enumerate}
  \item \(\dim Z/S^1\geq 6\),
  \item \(\codim Z^{S^1} \geq 4\),
  \item There is a \(B\)-structure on \(Z_\text{max}\), whose restriction to \(Y_\text{max}\) induces a two-equivalence \(Y_\text{max}\rightarrow B\).
  \end{enumerate}
  Then, if \(X\) admits an \(S^1\)-invariant metric of positive scalar curvature which is normally symmetric in codimension \(2\), then \(Y\) admits an \(S^1\)-invariant metric of positive scalar curvature.
\end{theorem}
\begin{proof}
Let \(\dim Z=n+1\).
The first step in the proof is to replace the bordism \(Z\) by a \(B\)-bordism \(Z'\) between \(X\) and \(Y'\) such that
\begin{enumerate}
\item\label{item:2} All codimension-two strata in \(Z'\) meet \(Y'\).
\item\label{item:1} \(Y\) can be constructed from \(Y'\) by resolving singular strata which are parts of boundaries of orientable strata in \(Z'\).
\end{enumerate}
By (\ref{item:1}) and Theorem~\ref{sec:resolv-sing-2} it is sufficient to construct a metric of positive scalar curvature which is normally symmetric in codimension \(2\) on \(Y'\) at these singular strata.
The construction of \(Z'\) is as follows.
Let \(F\) be a codimension \(2\) singular stratum in \(Z\) which does not meet \(Y\) and \(\Omega\subset F\) an orbit.
By the slice theorem there is a \(S^1\)-invariant tubular neighborhood \(N\) of \(\Omega\) in \(Z\) which is \(S^1\)-equivariantly diffeomorphic to
\begin{equation*}
  S^1\times_H(D^{n-2}\times D(W)).
\end{equation*}
Then we have \(\partial N= S^1\times_H(D^{n-2}\times S(W)\cup S^{n-3}\times D(W))\).
Let \(B\subset \partial N\) be a tubular neighborhood of an orbit in \(S^1\times_H(D^{n-2}\times S(W))\).
Then \(B\) is equivariantly diffeomorphic to \(S^1\times D^{n-1}\).
Since \(Z_{\text{max}}/S^1\) is connected, there is an embedding
\begin{equation*}
  \Psi:S^1\times D^{n-1}\times [0,1]\hookrightarrow Z_{\text{max}},
\end{equation*}
such that
\begin{align*}
  S^1\times D^{n-1}\times \{0\}&\subset Y_{\text{max}}\\
  S^1\times D^{n-1}\times\{1\}&=B\\
  S^1\times D^{n-1}\times ]0,1[&\subset Z-(Y\cup N).
\end{align*}
We set \(Z'=Z-(N\cup \image \Psi)\).

As in the proof of Theorem 34 of \cite{MR2376283} one sees that (\ref{item:2}) and (\ref{item:1}) hold.
To be more precise we have an equivariant diffeomorphism
\begin{equation*}
  Y\cong (Y'-\phi'(S^1\times_H (S^{n-3}\times D(W))))\cup \phi(S^1\times S^{n-3}\times D^2),
\end{equation*}
where \(\phi': S^1\times_H (S^{n-3}\times D(W))\rightarrow Y'\) and \(\phi:S^1\times S^{n-3}\times D^2\rightarrow Y\) are equivariant embeddings, such that
\(\image \phi\) is contained in an equivariant coordinate chart of \(Y\).

Since \(\pi_2(B)\) is finitely generated as a \(\mathbb{Z}[\pi_1(B)]\)-module we can assume that \(\image \Psi\) avoids a finite set of embedded two-spheres which are mapped by the \(B\)-structure \(\nu\) to the generators of \(\pi_2(B)\).
Hence, we may assume that
\begin{equation*}
  \nu_*':\pi_2(Y'_{\text{max}})\rightarrow \pi_2(B)
\end{equation*}
is still surjective.

But there might be a non-trivial linking sphere \(S^1\subset Y'_{\text{max}}/S^1\) of \(\Sigma/S^1\subset Y'/S^1\) where \(\Sigma\subset Y'\) is the singular stratum \(\phi'(S^1\times_H(S^{n-3}\times \{0\})\).
This problem can be dealt with as in Hanke's paper by attaching a \(2\)-handle to \(Z\) which can be canceled by a \(3\)-handle.
Therefore the same argument as in Hanke's paper leads to an isomorphism \(\pi_1(Y')\rightarrow \pi_1(B)\).

Now it follows from Lemma~\ref{sec:more-results-2} and Theorem 15 of \cite{MR2376283} (with the refinement of Lemma~\ref{sec:more-results-1}), that \(Y'\) admits an invariant metric of positive scalar curvature which is normally symmetric in codimension \(2\) at the singular strata described  in (\ref{item:1}).
Therefore it follows from Theorem~\ref{sec:resolv-sing-2} that \(Y\) admits an invariant metric of positive scalar curvature. 
\end{proof}

\begin{remark}
\label{sec:resolv-sing-9}
  Given the other conditions on \(Z\) from the above Theorem, the condition \(\codim Z^{S^1}\geq 4\) cannot be relaxed.
  This can be seen as follows.
  Let \(Y\) be a free simply connected \(S^1\)-manifold, whose orbit space does not admit a metric of positive scalar curvature.
  Then \(Y\) is necessarily Spin and the \(S^1\)-action is of even type.
  Let \(Z\) be the trace of an equivariant surgery on an orbit in \(M\), as in Lemma~3.1 of \cite{wiemeler15:_circl}.
  Then \(Z\) is a semi-free \(S^1\)-manifold, not Spin and has a codimension-two fixed point component which meets the boundary component \(X\) which is not equal to \(Y\).
  By Theorem~2.4 of \cite{wiemeler15:_circl}, \(X\) admits an invariant metric of positive scalar curvature.
  But \(Z_{\text{max}}\) is homotopy equivalent to \(Y\) and therefore admits a Spin-structure.
  Therefore, by Remark~\ref{sec:more-results-4}, all assumptions of Theorem~\ref{sec:more-results} except the one about the codimension of the fixed point set are satisfied. 
But \(Y\) does not admit an invariant metric of positive scalar curvature by \cite[Theorem C]{berard83:_scalar}.

The explanation for this is that the bordism \(Z\) has a codimension-two fixed point component which does not meet \(Y\).
Therefore every equivariant handle decomposition of \(Z\) contains handles of codimension less than \(3\) \cite[Proposition 17]{MR2376283}.
But to these the surgery principle cannot be applied.

In principle this phenomenon also appears if there are codimension-two singular strata with finite isotropy group in \(Z\).
To deal with this case a desingularization process was introduced in \cite{MR2376283}.
\end{remark}

We have the following corollaries to Theorem~\ref{sec:more-results}:

\begin{cor}
\label{sec:resolv-sing-7}
  Let \(M\) be a Spin \(S^1\)-manifold of dimension at least six with simply connected maximal stratum and without fixed point components of codimension two.
  Then \(M\) admits an invariant metric of positive scalar curvature which is normally symmetric in codimension \(2\) if and only if \(M\) is equivariantly Spin-bordant to a manifold \(M'\) which admits such a metric and has no fixed point components of codimension two, such that the bordism \(Z\) between \(M\) and \(M'\) does not have fixed point components of codimension two.
\end{cor}
\begin{proof}
  Let \(Z\) be an equivariant Spin cobordism between \(M\) and \(M'\) as above.

  Then, by Remark~\ref{sec:resolv-sing-5}, all strata of \(Z\) are orientable.
  Therefore the corollary follows from Theorem~\ref{sec:more-results}.
 \end{proof}

\begin{cor}
\label{sec:resolv-sing-6}
  Let \(M\) be a \(S^1\)-manifold of dimension at least six with simply connected non-Spin maximal stratum and without fixed point components of codimension two.
  Then there is an \(0\leq l\leq (\dim M-1)/2\) such that the equivariant connected sum of \(2^l\) copies of \(M\) admits an invariant metric of positive scalar curvature if there is an equivariant bordism from \(M\) to a manifold \(M'\) which admits a metric which is normally symmetric in codimension two, has no fixed point components of codimension two and all whose singular strata are orientable.
\end{cor}

\begin{proof}
  Let \(Z\) be a equivariant bordism between \(M\) and \(M'\).
  Then there might be codimension-two fixed point components in \(Z\).
  But by assumption they do not meet the boundary.
  Therefore we can cut them out of \(Z\).
  This construction leads to a new bordism \(Z'\) between \(M\) and \(M'\amalg M_1\amalg\dots\amalg M_k\), where \(S^1\) acts freely on the \(M_i\).
  After attaching handles of codimension at least \(3\) to \(Z'\), we may assume that the \(M_i/S^1\) are simply connected and not Spin.
  Since the \(M_i/S^1\) are not Spin, it follows from Theorem C of \cite{berard83:_scalar} that there is an invariant metric of positive scalar curvature on each \(M_i\).

  By Lemma~\ref{sec:s1-manifolds-not-4}, we can replace \(Z'\) by an equivariant bordism \(Z''\) between several copies of \(M'\amalg M_1\amalg\dots\amalg M_k\) and a disjoint union  of copies of \(M\) such that all singular strata in \(Z''\) which do not meet the copies of \(M\) are orientable.
  By adding traces of \(0\)-dimensional equivariant surgeries we can assume that there is a bordism between \(M'\amalg M_1\amalg\dots\amalg M_k\) and an equivariant connected sum \(N\)  of several copies of \(M\) such that all singular strata in the bordism which do not meet \(N\) are orientable.
  Hence, the corollary follows from Theorem~\ref{sec:more-results}.
\end{proof}

Note that if the \(T\)-manifold \(M\) satisfies condition C, then for all closed subgroups \(H\subset T\) the fixed point set \(M^H\) is orientable.

\begin{lemma}
\label{sec:resolv-sing-3}
  Let \(M\) be a \(S^1_0\)-manifold, where \(S_0^1=S^1\), which satisfies condition C such that there is a fixed point component \(F\)  of codimension two.
  Then there is an invariant metric of positive scalar curvature on \(M\) which is normally symmetric in codimension two.
\end{lemma}
\begin{proof}
   This follows from an inspection of the proof of Theorem~2.4 of \cite{wiemeler15:_circl} and Lemma 24 of \cite{MR2376283}.
   We use the same notation as in the proof of Theorem~2.4 of \cite{wiemeler15:_circl}.

  Since \(M\) satisfies condition C, this also holds for the \(S^1_0\times S^1_1\)-manifold \(Z \times D^2\).
  Hence it follows from Lemma 24\footnote{The proof of this lemma also holds for \(T\)-manifolds where \(T\) is a torus, instead of \(S^1\)-manifolds.} of \cite{MR2376283} , that 
  \begin{equation*}
    \partial (Z\times D^2)= SF\times D^2\cup Z\times S^1
  \end{equation*}
 admits a \(S^1_0\times S^1_1\)-invariant metric of positive scalar curvature which is normally symmetric in codimension two.

For \(H\subsetneq S^1_0\), we have \(p^{-1}(M^H)= (\partial (Z\times D^2))^H\), where 
\begin{equation*}
  p:\partial (Z\times D^2)\rightarrow M=\partial (Z\times D^2)/\diag(S^1_0\times S^1_1)
\end{equation*}
 is the orbit map of the \(\diag(S^1_0\times S^1_1)\)-action.
 Hence, it follows from the construction in the proof of Theorem~2.2 of \cite{wiemeler15:_circl} that \(M\) admits a \(S^1_0\)-invariant metric of positive scalar curvature  which is normally symmetric in codimension two.
\end{proof}

Using the above lemma we can prove the following theorem:

\begin{theorem}
  Let \(M\) be a connected \(S^1\)-manifold satisfying condition C, such that 
$$\pi_1(M_{\text{max}})=0$$
 and \(M_{\text{max}}\) is not Spin.
  Moreover, let \(J\subset \Omega_*^{C,S^1}\) be the ideal generated by connected manifolds with non-trivial \(S^1\)-actions.
  If \(\dim M \geq 6\) and \([M]\in J^2\), then \(M\) admits an \(S^1\)-invariant metric of positive scalar curvature.
\end{theorem}
\begin{proof}
By Theorem~2.4 of \cite{wiemeler15:_circl} we may assume that \(\codim M^{S^1}\geq 4\).
  Let \(M_i,N_i\) be connected manifolds with non-trivial \(S^1\)-action satisfying Condition C, such that 
  \begin{equation*}
    [M]=\sum_i [M_i\times N_i].
  \end{equation*}
  Since \(\Omega_1^{C,S^1}=0\), we may assume that \(\dim M_i, \dim N_i \geq 2\) for all \(i\).
  Hence, by Lemma~3.1 of \cite{wiemeler15:_circl} we may assume that all \(M_i\) and \(N_i\) have \(S^1\)-fixed point components of codimension two.
  Therefore by Lemma~\ref{sec:resolv-sing-3}, we may assume that \(M_i\times N_i\) admits an \(S^1\)-invariant metric of positive scalar curvature which is normally symmetric in codimension two.
  Hence the theorem follows from Corollary~\ref{sec:resolv-sing-6}.
\end{proof}

\section{Normally symmetric metrics are generic}
\label{sec:loc_sym_are_gen}

In this section we prove that under mild conditions on the isotropy groups of the singular strata of codimension two in an \(S^1\)-manifold \(M\), any invariant metric \(g\) on \(M\) can be deformed to a metric which is normally symmetric in codimension two.

The main result of this section is as follows:

\begin{theorem}
  Let \(M\) be an orientable effective \(S^1\)-manifold. Moreover, let \(g\) be an invariant metric on \(M\).

If there are no codimension-two singular strata with isotropy group \(\mathbb{Z}_2\), then there is an invariant metric \(g'\) on \(M\) which is \(C^2\)-close to \(g\) and normally symmetric in codimension two.
\end{theorem}
\begin{proof}
  Let \(N\subset M\) be a codimension-two open singular stratum of \(M\).
  Let \(U=S^1\times_H W\times \R^{n-3}\) be a neighborhood of an orbit in \(N\).
  Here \(W\) can be assumed to be the standard one-dimensional complex representation of \(H\subset S^1\) because the \(S^1\)-action on \(M\) is effective.

  We pull back \(g\) to a metric \(\tilde{g}\) on \(\tilde{U}=S^1\times W\times \R^{n-3}\). This metric is \(S^1\times H\)-invariant.
  Let
  \begin{equation*}
    h: T\tilde{U}\otimes T\tilde{U}\rightarrow\R
  \end{equation*}
be the Taylor expansion of \(\tilde{g}\) in directions tangent to \(W\) up to terms of degree two.
Then \(h\) might be thought of as an invariant function on
\begin{equation*}
  (S^1\times W\times \R^{n-3}) \times (W\times \R^{n-2})\times (W\times \R^{n-2}),
\end{equation*}
which is linear in the copies of \(W\times \R^{n-2}\) and a polynomial of degree two in the first copy of \(W\).
Therefore \(h\) can be identified with a map \[\R^{n-3}\rightarrow ((S^0 W^*\oplus S^1W^* \oplus S^2 W^*)\otimes_\R (W^*\oplus_\R \R^{n-2})  \otimes_\R (W^*\oplus \R^{n-2}))^H,\] where \(W^*\) denotes the dual representation of \(W\) and  \(S^iW^*\) denotes the \(i\)-th symmetric product of \(W^*\).
 There are \(a_i\in \mathbb{N}\), such that
\begin{equation*}
  \begin{split}
  &\Big((S^0 W^*\oplus S^1W^* \oplus S^2 W^*)\otimes_\R ((W^*\oplus \R^{n-2}))\otimes_\R (W^*\oplus \R^{n-2})\Big)^H\otimes \C\\ &\subset
  \Big((\C\oplus (W^*\oplus W) \oplus (W^*\oplus W)\otimes_\C(W^*\oplus W))\otimes_\C((W^*\oplus W\oplus (n-2)\C)\\ & \quad\otimes_\C(W^*\oplus W \oplus (n-2)\C)\Big)^H\\
&=\Big(W^{*\otimes 4}\oplus W^{\otimes 4} \oplus a_3(W^{*\otimes 3}\oplus W^{\otimes 3})\oplus a_2 (W^{*\otimes 2}\oplus W^{\otimes 2})\\ & \quad \oplus a_1 (W^{*}\oplus W)\oplus a_0\C\Big)^H.
  \end{split}
\end{equation*}

Moreover, for \(H\) of order greater than \(4\) and \(b\leq 4\), we have
\begin{align*}
  (W^{\otimes b})^H&=(W^{\otimes b})^{S^1}&(W^{*\otimes b})^H&=(W^{*\otimes b})^{S^1}.
\end{align*}

Hence, it follows that \(h\) is  invariant under the rotational action of \(S^1\) on \(W\) if the order of \(H\) is greater than \(4\).

Now we can deform \(\tilde{g}\) so that it coincides with \(h\) in a neighborhood of \(S^1\times \{0\}\subset \tilde{U}\).
This metric induces a metric on \(U\) which is invariant under the rotational action of \(S^1\) on \(W\).

Since \(N\) is orientable the rotational action on \(W\) extends to an action on a neighborhood of \(N\) in \(M\) with fixed point set \(N\).
Therefore we can glue the metrics on different neighborhoods of orbits in \(N\).
This implies the claim if there are no singular strata of codimension two with isotropy group \(\mathbb{Z}_k\), \(k\leq 4\).

Now assume that that there is a singular stratum of codimension two with isotropy group \(\mathbb{Z}_3\).
Then we have to show that the projection \(\bar{h}\)  of \(h\) to \(a_3(W^{*\otimes 3}\oplus W^{\otimes 3})\) is trivial.
This projection has the form
\begin{equation*}
  \bar{h}= \alpha_1(u) z dz dz + \beta_1(u) \bar{z}d\bar{z}d\bar{z} + \sum_{j}(\alpha_{2j}(u) z^2 dz du_j + \beta_{2j}(u) \bar{z}^2 d\bar{z} du_j).
\end{equation*}
Here \(z\) denotes the complex coordinates in \(W\) and \(u\) denotes the coordinates in \(S^1\times\mathbb{R}^{n-3}\).

In polar coordinates \(z=re^{i\varphi}\) the above expression is equal to
\begin{align*}
\bar{h}&
=  r(\alpha_1(u)e^{i3\varphi}+ \beta_1(u)e^{-i3\varphi}) dr dr + r^2i(\alpha_1(u) e^{3i\varphi}- \beta_1(u) e^{-3i\varphi}) dr d\varphi \\ &\quad + \sum_j r^2(\alpha_{2j}(u)e^{3i\varphi}+ \beta_{2j}(u)e^{-3i\varphi})  dr du_j\\ &\quad - r^3(\alpha_1(u)e^{3i\varphi} + \beta_1(u)e^{-3i\varphi}) d\varphi d\varphi + \sum_j ir^3(\alpha_{2j}(u)e^{3i\varphi}-\beta_{2j}(u)e^{-3i\varphi}) d\varphi du_j. 
\end{align*}

But by the generalized Gauss Lemma \cite[Section 2.4]{MR2024928} we may assume that \(g\) is of the form
\begin{equation*}
  dr dr + h'(u,r,\varphi),
\end{equation*}
where \(h'(u,r,\varphi)\) is a metric on \(S^1\times \mathbb{R}^{n-3}\times S^1_r\).
Here \(S^1_r\) denotes the circle of radius \(r\) in \(W\).

Hence, we may assume that \(\alpha_1=\beta_1=\alpha_{2j}=\beta_{2j}=0\).
Therefore the metric can be deformed as in the first case.

The case of singular strata with isotropy group \(\mathbb{Z}_4\) is similar and left to the reader.
\end{proof}

If \(M\) is Spin and the \(S^1\)-action on \(M\) is of even type then there are no components of \(M^{\mathbb{Z}_2}\) of codimension two in \(M\).
Therefore we get the following corollary to the above theorem.

\begin{cor}
  Let \(M\) be a Spin \(S^1\)-manifold with an effective action of even type.
  Then \(M\) admits an invariant metric of positive scalar curvature if and only if it admits an invariant metric of positive scalar curvature which is normally symmetric in codimension two.
\end{cor}

\section{An obstruction to invariant metrics of positive scalar curvature}
\label{sec:obstruction}

Before we prove existence results for invariant metrics of positive scalar curvature on Spin-\(S^1\)-manifolds, we introduce an obstruction to the existence of such metrics.
Throughout this section we only deal with  Spin-\(S^1\)-manifolds with actions of even type.

To define our obstruction we first have to define \(S^1\)-equivariant versions of Stolz's relative bordism groups.
This goes as follows:

\begin{definition}
  Let \(\mathcal{F}\) be a family of \(S^1\)-slice types. Then the elements of the relative bordism group
\[R_n^{S^1}[\mathcal{F}]\]
are represented by pairs \((M,g)\) where
\begin{enumerate}
\item \(M\) is an \(n\)-dimensional Spin manifold with boundary \(\partial M\) with an even \(S^1\)-action of type \(\mathcal{F}\).
\item \(g\) is an \(S^1\)-invariant metric of positive scalar curvature on \(\partial M\).
\end{enumerate}

Two such pairs \((M,g)\) and \((N,h)\) are identified if there are
\begin{enumerate}
\item an \((n+1)\)-dimensional Spin manifold \(Z\) with boundary with an even \(S^1\)-action of type \(\mathcal{F}\),
\item an \(n\)-dimensional Spin manifold \(Y\) with boundary with an even \(S^1\)-action of type \(\mathcal{F}\) such that \(\partial Y\) is equivariantly diffeomorphic to \(\partial M \amalg \partial N\),
\item a invariant metric of positive scalar curvature on \(Y\) which extends the metrics \(g,h\) on the boundary and has product form near the boundary,
\end{enumerate}
such that \(\partial Z\) is equivariantly diffeomorphic to \(M\cup_{\partial M} Y\cup_{\partial N} N\).
\end{definition}

The following two theorems are crucial for the construction of our obstruction:

\begin{theorem}
\label{sec:an-obstr-invar}
  Let \(\mathcal{F}_1,\mathcal{F}_2\) be two families of effective \(S^1\)-slice types such that there is a slice type \(\sigma=[H,W]\) with \(\mathcal{F}_1=\mathcal{F}_2\amalg\{\sigma\}\).
Assume that the dimension of the \(H\)-representation \(W\) is at least \(3\).

Let \((M,g)\) an \(S^1\)-manifold with boundary of type \(\mathcal{F}_1\) with an invariant metric \(g\) of positive scalar curvature on \(\partial M\).

Then there is an \(S^1\)-manifold \(M'\)  with boundary of type \(\mathcal{F}_2\) with an invariant metric \(g'\) of positive scalar curvature on \(\partial M'\), such that \(g\) can be extended to an invariant metric of positive scalar curvature on all of \(M\) if and only if \(g'\) can be extended to an invariant metric of positive scalar curvature on all of \(M'\).

Moreover, the assignment
\begin{align*}\Phi_{\mathcal{F}_1,\mathcal{F}_2}:R_n^{S^1}[\mathcal{F}_1]&\rightarrow R_n^{S^1}[\mathcal{F}_2]& [M,g]&\mapsto [M',g']\end{align*}
is a well-defined group homomorphism.
\end{theorem}

For the proof of this theorem we need the following straightforward generalization of Theorem 2.13 of \cite{MR2789750}. We leave the proof to the reader.

\begin{lemma}
\label{sec:an-obstr-invar-1}
  Let \(X\) be a smooth compact \(S^1\)-manifold of dimension \(n\).
  Let \(Y\subset X\) be a compact invariant submanifold of codimension at least \(3\) and \(B\) a compact space.

  Moreover let \(\mathcal{B}=\{g_b\in \Riem_{S^1}^+(X);b\in B\}\) be a continuous family of invariant psc-metrics on \(X\).
Here \(\Riem_{S^1}^+(X)\) denotes the space of \(S^1\)-invariant psc-metrics on \(X\) equipped with the \(C^2\)-topology.
  Let \(h\) be an invariant metric on \(Y\).
  
  Then there is a continuous map
  \begin{align*}
    \mathcal{B}&\rightarrow \Riem_{S^1}^+(X)&g_b&\mapsto g_b^{\text{std}}
  \end{align*}
such that
\begin{enumerate}
\item Each metric \(g_b^{\text{std}}\) has standard form on a tubular neighborhood of \(Y\) and coincides with \(g_b\) outside a slightly bigger neighborhood.
  Here a metric on a tubular neighborhood of \(Y\) is called of standard form if it is a connection metric with base metric \(h\) and fibers isometric to torpedo metrics.
\item The map \(g_b\mapsto g_b^{\text{std}}\) is homotopic to the inclusion \(\mathcal{B}\hookrightarrow \Riem_{S^1}^+(X)\).
\end{enumerate}
\end{lemma}

\begin{proof}[Proof of Theorem~\ref{sec:an-obstr-invar}]
  Let \([M,g]\) be an element of \(R_n^{S^1}[\mathcal{F}_1]\).
  Then the \(\sigma\)-stratum \(M_{(\sigma)}\) of \(M\) is a compact invariant submanifold of codimension at least three in \(M\) such that \(\partial( M_{(\sigma)})=(\partial M)\cap M_{(\sigma)}\).

Therefore by applying Lemma~\ref{sec:an-obstr-invar-1} in the case \(\mathcal{B}=\{g\}\) (and \(B=[0,1]\)) we get a metric \(g_1\) of positive scalar curvature on \(\partial M\) (unique up to isotopy) such that \(g_1\) has standard form near \(\partial M_{(\sigma)}\).

By cutting out a small invariant tubular neighborhood of \(M_{(\sigma)}\) from \(M\), we get a new manifold \(M'\) with

\begin{equation*}
  \partial M'=S(M_{(\sigma)})\cup_{S(\partial M_{(\sigma)})}(\partial M -D(\partial M_{(\sigma)})).
\end{equation*}
Here \(S(M_{(\sigma)})\), \(S(\partial M_{(\sigma)})\) and \(D(\partial M_{(\sigma)})\) denote the normal sphere bundle of \(M_{(\sigma)}\) in \(M\) and the normal sphere and disc bundles of  \(\partial M_{(\sigma)}\) in \(\partial M\), respectively.

We can extend the metric \(g_1|_{\partial M -D(\partial M_{(\sigma)})}\) by a connection metric on \(S(M_{(\sigma)})\) with fibers isometric to round spheres to get an invariant  metric on all of \(\partial M'\).
By shrinking the fibers of this connection metric we can assume that the extended metric \(g'=g_2\) has positive scalar curvature.

If also follows from Lemma~\ref{sec:an-obstr-invar-1} that \(g'\) extends to a metric of positive scalar curvature on \(M'\) if \(g\) extends to such a metric on \(M\).

It  remains to show that the bordism class \([M',g']\) depends only on the bordism class of \((M,g)\).

To do so, let \((N,h)\) be another pair which represents \([M,g]\).
Then there is a \((n+1)\)-dimensional manifold \(Z\) with \(\partial Z=M\cup_{\partial M} Y\cup_{\partial N} N\) and \(Y\) has an invariant psc-metric extending \((g,h)\).

By Lemma~\ref{sec:an-obstr-invar-1}, we can assume that the metric on \(Y\) has standard form near \(Y_{(\sigma)}\) and restricts to \(g_1\), \(h_1\) on \(\partial M\) and \(\partial N\) respectively.
Then by cutting out a small tubular neighborhood of \(Y_{(\sigma)}\) from \(Y\) and arguing as in the construction of \(g'\) we get a bordism between \((M',g')\) and \((N',h')\).

Hence we have shown that the assignment \([M,g]\mapsto [M',g'']\) has all the desired properties.
\end{proof}

The second theorem which we need for our construction deals with the case of codimension-two singular strata.

\begin{theorem}
\label{sec:an-obstr-invar-2}
  Let \(\mathcal{F}_1,\mathcal{F}_2\) be two families of effective \(S^1\)-slice types such that there is a slice type \(\sigma=[H,W]\) with \(\mathcal{F}_1=\mathcal{F}_2\amalg\{\sigma\}\).
Assume that the dimension of the \(H\)-representation \(W\) is \(2\).

Let \((M,g)\) an \(S^1\)-manifold with boundary of type \(\mathcal{F}_1\) with an invariant metric \(g\) of positive scalar curvature on \(\partial M\).

Then there is an \(S^1\)-manifold \(M'\)  with boundary of type \(\mathcal{F}_2\) with an invariant metric \(g'\) of positive scalar curvature on \(\partial M'\), such that \(g'\) can be extended to an invariant metric of positive scalar curvature on all of \(M'\) if \(g\) can be extended to an invariant metric of positive scalar curvature on all of \(M\).
Moreover, the complement of the \(\sigma\)-stratum in \(M\) is equivariantly diffeomorphic to an invariant open dense subset of \(M'\).

The assignment
\begin{align*}\Phi_{\mathcal{F}_1,\mathcal{F}_2}:R_n^{S^1}[\mathcal{F}_1]&\rightarrow R_n^{S^1}[\mathcal{F}_2]& [M,g]&\mapsto [M',g']\end{align*}
is a well-defined group homomorphism.
\end{theorem}

\begin{proof}
Let \([M,g]\) be an element of \(R_n^{S^1}[\mathcal{F}_1]\).
We can assume that \(g\) is normally symmetric in codimension two by the discussion in section \ref{sec:loc_sym_are_gen}.

We can apply the following desingularization process to get a manifold \(M'\) with boundary and without codimension two singular strata.
It is similar to the desingularization process in \cite[Section 4]{MR2376283}.

Let \(N\) be the \(\sigma\)-stratum in \(M\).
Then a neighborhood of \(N\) in \(M\) is diffeomorphic to a fiber bundle \(E\) with fiber \(S^1\times_H D(W)\) and structure group \(S^1\times_H SO(W)=S^1\times_H S^1_1\) over \(N/S^1\).

The boundary of this neighborhood \(\partial E\) is given by the principal \(S^1\times_H S_1^1\)-bundle \(P\) associated to \(E\).

Now we fix a circle subgroup \(S^1_2\subset S^1\times_H S^1_1\) with \(S^1\cap S^1_2=\{1\}\).
Denote by \(E'\) the \(S^1\times D^2\)-bundle \(P\times_{S^1_2} D^2\) where \(S^1_2\) acts by rotation on \(D^2\).
Then \(E'\) has the same boundary as \(E\) and we define
\begin{equation*}
  M'=(M - E)\cup_P E'.
\end{equation*}
This \(M'\) does not have singular strata of type \(\sigma\).

Note here that \(M'\) itself might depend on the choice of \(S^1_2\).
But its orbit space \(M'/S^1\) does not depend on this choice.

Since the constructions in Section 4 of \cite{MR2376283} are mainly local arguments, they also hold in our desingularization process.
This means that, if \(M\) admits a metric of positive scalar curvature which is normally symmetric in codimension two, then \(M'\) also admits such a metric.
 
In any case we have a metric \(g'\) on \(M'\) whose restriction to the boundary has positive scalar curvature.
Checking that the assignment \([M,g]\mapsto [M',g']\) has all the required properties is similar to the proof of Theorem~\ref{sec:an-obstr-invar}.
We omit the details.
\end{proof}

Since there are only finitely many slice types in any compact \(S^1\)-manifold, by composing the maps from Theorems \ref{sec:an-obstr-invar} and \ref{sec:an-obstr-invar-2} we get a homomorphism
\begin{equation*}
  \Phi:R_n^{S^1}[\mathcal{AE}]\rightarrow R_n^{S^1}[\mathcal{F}_0]
\end{equation*}
where \(\mathcal{AE}\) and \(\mathcal{F}_0\) are the families of all effective \(S^1\)-slice types and the free slice type, respectively.

Using the construction in the proof of Berard Bergery's Theorem C \cite{berard83:_scalar}, we get a homomorphism
\begin{equation*}
  \Psi: R_n^{S^1}[\mathcal{F}_0]\rightarrow R_{n-1},
\end{equation*}
which sends a pair \((M,g)\) to \((M/S^1,h)\), where \(h\) is the metric of positive scalar curvature on the orbit space constructed by Berard Bergery and \(R_{n-1}\) denotes the usual non-equivariant Stolz relative bordism group.

Therefore we can define an obstruction \(\hat{A}_{S^1}\) to the existence of invariant positive scalar curvature metrics on compact \(S^1\)-manifolds as follows:
\begin{equation*}
  \hat{A}_{S^1}(M)=\hat{A}\circ \Psi\circ \Phi(M)\in \mathbb{Z}.
\end{equation*}
Here \(\hat{A}(M)\) is the index of the Dirac operator associated to a metric on \(M\) which extends the psc-metric on the boundary.
It vanishes if the metric on \(M\) can be chosen to have positive scalar curvature.

For semi-free \(S^1\)-manifolds \(M\), \(\hat{A}_{S^1}(M)\) coincides with the index obstruction to metrics of positive scalar curvature defined by Lott \cite{MR1758446}.
He defines his obstruction for semi-free \(S^1\)-manifolds as the integral of the \(\hat{A}\)-class over the regular part of \(M/S^1\).
Moreover, he shows that the value of this integral coincides with the index of a Dirac operator on \((M-D(M^{S^1}))/S^1\) with respect to a metric with suitable boundary conditions.
Here \(D(M^{S^1})\) denotes an open tubular neighborhood of the fixed point set in \(M\).

We summarize what we have proven in the following theorem.

\begin{theorem}
  Let \(M\) be an \(n\)-dimensional closed Spin-manifold with an effective action of \(S^1\) of even type.
  Then \(\hat{A}_{S^1}(M)\) is an invariant of the equivariant Spin-bordism type of \(M\). Moreover, it vanishes if there is an invariant psc-metric on \(M\).
\end{theorem}

\section{Invariant metrics of positive scalar curvature and a result of Atiyah and Hirzebruch}
\label{sec:atiyah_hirzebruch}

Now we can complete the proofs of our existence results for invariant metrics of positive scalar curvature. These are as follows: 

\begin{theorem}
\label{sec:non-semi-free-1}
  Let \(M\) be a connected effective \(S^1\)-manifold of dimension at least six such that \(\pi_1(M_\text{max})=0\) and \(M_{\text{max}}\) is not Spin.
  Then there is a \(k\geq 0\) such that the equivariant connected sum of \(2^k\) copies of \(M\) admits an invariant metric of positive scalar curvature.
\end{theorem}
\begin{proof}
  By Theorem~2.4 of \cite{wiemeler15:_circl}, we may assume that there is no codimension-two fixed point component in \(M\).
  Therefore, by Corollary~\ref{sec:resolv-sing-6}, it is sufficient to show that \(2M\) is equivariantly bordant to a manifold \(M'\) which admits an invariant metric of positive scalar curvature which is normally symmetric in codimension two such that \(\codim M'^{S^1} \geq 4\) and all of whose singular strata are orientable.
  From Theorems~\ref{sec:non-semi-free} and \ref{sec:s1-manifolds-not-1}, we see that there is an equivariant bordism \(Z\) between \(M\) and \(M'=M_1\amalg M_2\), where \(M_1\) is a semi-free \(S^1\)-manifold and \(M_2\) is a generalized Bott manifold. Note that \(M_2\) admits an invariant metric of positive scalar curvature which is normally symmetric in codimension two and that all singular strata in \(M_2\) are orientable. Moreover, \(M'\) satisfies \(\codim {M'}^{S^1}\geq 4\).
  After attaching \(S^1\)-handles to \(Z\) we may assume that all components of \(M_1\) are simply connected and not Spin.

Hence, the Theorem follows from Theorem~4.7 of \cite{wiemeler15:_circl}. 
\end{proof}

Now we turn to the proof of a similar result for Spin-manifolds.

\begin{theorem}
\label{sec:Spin-case-1}
  Let \(M\) be a Spin \(S^1\)-manifold with \(\dim M\geq 6\), an effective \(S^1\)-action of odd type and \(\pi_1(M_{\text{max}})=0\). Then there is a \(k\in \N\) such that the equivariant connected sum of \(2^k\) copies of \(M\) admits an invariant metric of positive scalar curvature which is normally symmetric in codimension two.
\end{theorem}

\begin{theorem}
\label{sec:Spin-case-2}
  For a Spin \(S^1\)-manifold with \(\dim M\geq 6\), an effective \(S^1\)-action of even type and \(\pi_{1}(M_{\text{max}})=0\), we have \(\hat{A}_{S^1}(M)=0\) if and only if there is a \(k\in \N\) such that the equivariant connected sum of \(2^k\) copies of \(M\) admits an invariant metric of positive scalar curvature which is normally symmetric in codimension two.
\end{theorem}

\begin{proof}
  By Theorem \ref{sec:Spin-case} and Proposition \ref{sec:s1-manifolds-not}, the connected sum of \(2^l\) copies of \(M\) is equivariantly bordant to a union \(M_1\amalg M_2\), where \(M_1\) is a semi-free simply connected \(S^1\)-manifold and \(M_2\) is a \(S^1\)-manifold which admits an invariant metric of positive scalar curvature which is normally symmetric in codimension two.

If the \(S^1\)-action on \(M\) is of even type, we have \(\hat{A}_{S^1}(M)=2^{-l}\hat{A}_{S^1}(M_1/S^1)\). Now the theorems follow from Theorems 4.7 and 4.11 of \cite{wiemeler15:_circl}.
\end{proof}

As an application of our results we give a new proof of the following result of Atiyah and Hirzebruch \cite{MR0278334}. 

\begin{theorem}[\cite{MR0278334}]
  Let \(M\) be a Spin manifold with a non-trivial action of \(S^1\).
  Then \(\hat{A}(M)\) vanishes.
\end{theorem}
\begin{proof}
  We may assume that \(\dim M = 4k\) and that the \(S^1\)-action is effective.
  Then by Theorem~\ref{sec:Spin-case} and Proposition~\ref{sec:s1-manifolds-not}, \(2^lM\) is equivariantly Spin bordant to a union \(M_1\amalg M_2\), where \(M_1\) is simply connected and semi-free and \(M_2\) admits an invariant metric of positive scalar curvature.
By Theorems 4.7 and 4.11 of \cite{wiemeler15:_circl}, the obstruction \(\hat{A}_{S^1}(M_1/S^1)\) that \(2^{l'}M_1\) admits an invariant metric of positive scalar curvature vanishes by dimension reasons.
Hence it follows that \(2^{l+l'}\hat{A}(M)=\hat{A}(2^{l'}M_1)+\hat{A}(2^{l'}M_2)=0\). This implies \(\hat{A}(M)=0\).
\end{proof}

\section{Rigidity of elliptic genera}
\label{sec:rigidity}

In this section we give a proof of the rigidity of elliptic genera.
At first we recall the definition of an equivariant genus.
We follow \cite{MR970284} for this definition.

A \(\Lambda\)-genus is a ring homomorphism \(\varphi:\Omega_*^{SO}\rightarrow \Lambda\) where \(\Lambda\) is a \(\C\)-algebra.
For such a homomorphism one denotes by
\begin{equation*}
  g(u)=\sum_{i\geq 0} \frac{\varphi[\C P^{2i}]}{2i+1} u^{2i+1} \in \Lambda[[u]]
\end{equation*}
the logarithm of \(\varphi\) and by \(\Phi\in H^{**}(BSO;\Lambda)\) the total Hirzebruch class associated to \(\varphi\). \(\Phi\) is uniquely determined by the property that for the canonical line bundle \(\gamma\) over \(BS^1\),
\(\Phi(L)\) is given by \(\frac{u}{g^{-1}(u)}\).

Then for every oriented manifold \(M\) one has
\begin{equation*}
  \varphi[M]=\langle \Phi(TM),[M]\rangle.
\end{equation*}

For a compact Lie group \(G\), the \(G\)-equivariant genus \(\varphi_{G}\)  associated to \(\varphi\) is defined as
\begin{equation*}
  \varphi_G[M]=p_*\Phi(TM_G)\in H^{**}(BG;\Lambda).
\end{equation*}
Here \(M\) is a \(G\)-manifold and \(TM_G\) is the Borel construction of the tangent bundle of \(M\).
It is a vector bundle over the Borel construction \(M_G\) of \(M\).
Moreover, \(p_*\) denotes the integration over the fiber in the fibration \(M\rightarrow M_G\rightarrow BG\).

It follows from this definition that if \(H\) is a closed subgroup of \(G\), then we have
\begin{equation*}
  \varphi_H[M]=f^*\varphi_G[M],
\end{equation*}
where \(f:BH\rightarrow BG\) is the map induced by the inclusion \(H\hookrightarrow G\).

\begin{lemma}
  Let \(G=S^1\) and \(M\) be an oriented \(G\)-manifold.
  Then the equivariant genus \(\varphi_G[M]\) depends only on the \(G\)-equivariant bordism type of \(M\). 
\end{lemma}
\begin{proof}
  It is sufficient to show that if \(M=\partial W\) is an equivariant boundary, then \(\varphi_G[M]=0\).
  Since the homology of \(BS^1=\C P^\infty\) is concentrated in even degrees and generated by the fundamental classes of the natural inclusions \(\iota_n:\C P^n\hookrightarrow \C P^\infty\), \(n\geq 0\),
 it is sufficient to show that
\begin{equation*}
  0=\langle p_*\Phi(TM_G),[\C P^n]\rangle,
\end{equation*}
for all \(n\).
Now we have
\begin{align*}
  \langle p_*\Phi(TM_G),[\C P^n]\rangle &=\langle p_*\Phi(TM_G|_{\C P^n}),[\C P^n]\rangle\\
&=\langle \Phi(TM_G|_{\C P^n}),[M_G|_{\C P^n}]\rangle=0. 
\end{align*}
Here the first two equations follow from the properties of \(p_*\).
Moreover, the last equality follows because \(M_G|_{\C P^n}\) bounds \(W_G|_{\C P^n}\).
This proves the lemma.
\end{proof}

A \(\Lambda\)-genus \(\varphi\) is called \emph{elliptic} if there are \(\delta,\epsilon\in\Lambda\) such that its logarithm is given by
\begin{equation*}
  g(u)=\int_0^u\frac{dz}{\sqrt{1-2\delta z^2+\epsilon z^4}}.
\end{equation*}

We call an equivariant genus \(\varphi_{S^1}\) of an \(S^1\)-manifold \(M\) rigid, if \(\varphi_{S^1}[M]\in H^{**}(BS^1;\Lambda)=\Lambda[[u]]\) is constant in \(u\).
The following has been proved by Ochanine \cite{MR970284}.

\begin{theorem}
  The elliptic genus of a semi-free Spin-\(S^1\)-manifold is rigid.
\end{theorem}

In view of the above lemma and Theorem~\ref{sec:introduction-3} it suffices to show the following lemma to prove the rigidity of elliptic genera (Theorem~\ref{sec:introduction-4}).
In an effective \(T\)-manifold the codimension of the fixed point set is at least \(2\dim T\).
The next lemma states that the \(T\)-equivariant elliptic genus of an effective \(T\)-manifold is constant if the codimension of all components of the fixed point set is minimal.

\begin{lemma}
  Let \(M\) be an effective Spin-\(T^n\)-manifold, such that all fixed point components have codimension \(2n\). Then the \(T^n\)-equivariant elliptic genera of \(M\) are rigid.
\end{lemma}
\begin{proof}
  It suffices to consider the equivariant elliptic genus \(\varphi_{T^n}(M)\) defined in Section 2.1 of Ochanine's paper.

This is defined as follows:
For a lattice \(W\subset \C\) and a non-trivial homomorphism \(r:W\rightarrow \mathbb{Z}_2\) there exists a unique meromorphic function \(x\) on \(\C\) such that
\begin{enumerate}
\item \(x\) is odd,
\item the poles of \(x\) are exactly the points in \(W\); they are all simple and the residues of \(x\) in \(w\in W\) is given by \((-1)^{r(w)}\),
\item\label{item:3} for all \(w\in W\) we have
  \begin{equation*}
    x(u+w)=(-1)^{r(w)}x(u).
  \end{equation*}
\end{enumerate}
From this one defines a genus \(\varphi\) such that \(g(u)^{-1}\) is the Taylor expansion of \(1/x\) in the point \(u=0\).

With this definition, \(\varphi_{T^n}[M]\) can be identified with a meromorphic function on \(\C^n\). 

  Let \(F\subset M\) be a fixed point component and \(\lambda_{1,F},\dots,\lambda_{n,F}:\mathbb{Z}^n\rightarrow \mathbb{Z}\) the weights of the \(T^n\)-action on the normal bundle to \(F\).
  Since the \(T^n\)-action is effective and \(\codim F=2n\), it follows that
  \begin{equation*}
    (\lambda_{1,F},\dots,\lambda_{n,F}):\mathbb{Z}^n\rightarrow \mathbb{Z}^n
  \end{equation*}
is an isomorphism.
In particular each \(\lambda_{i,F}\) is surjective.

As in the proof of Proposition 7 in Ochanine's paper \cite{MR970284} one sees that the genus \(\varphi_{T^n}(M)\) is a polynomial in \(x\circ\lambda_{i,F,\C}\) and \(y\circ\lambda_{i,F,\C}\), \(i=1,\dots,n\). Here \(F\subset M^{T^n}\) is a component of \(M^{T^n}\), and \(\lambda_{i,F,\C}\) is the linear extension of \(\lambda_{i,F}\) to \(\C^n\). Moreover, \(y\) is the derivative of \(x\).

In particular, the poles of \(\varphi_{T^n}(M)\) lie in the union of the following hyperplanes:
\begin{equation*}
  \ker \lambda_{i,F,\C} + z
\end{equation*}
with \(z\in W^n\). 
Since the singular set of a meromorphic function on \(\C^n\) is empty or an analytic set of codimension one,
we may assume that the singular set of \(\varphi_{T^n}\) has a non-empty open intersection with one of the hyperplanes above.

Since \(M\) is Spin, the mod two reduction of \(\sum_{i=1}^n\lambda_{i,F}\) does not depend on the component \(F\). It follows from the defining equation (\ref{item:3})  for \(x\), that a non-empty open set in the hyperplane
\(\ker \lambda_{i,F,\C}\) is singular for \(\varphi_{T^n}\).

But the restriction of \(\varphi_{T^n}\) to this hyperplane equals \(\varphi_{T^{n-1}}\) where \(T^{n-1}\) is the codimension one subtorus of \(T^n\) which is defined by \(\lambda_{i,F}\).
Since \(\varphi_{T^{n-1}}\) is a meromorphic function on \(\C^{n-1}\), it follows that the intersection of the singular set of \(\varphi_{T^n}\) with \(\ker \lambda_{i,F}\) can only be open if it is empty.
Therefore \(\varphi_{T^n}\) does not have singular points. This implies that it is constant. 
\end{proof}

Finally we want to compare our proof of the rigidity of elliptic genera with the proof of Bott--Taubes.
The difference between our proof and the proof of Bott--Taubes is that they prove that the equivariant universal elliptic genus of an Spin-\(S^1\)-manifold \(M\)  equals some twisted elliptic genus of some auxiliary manifolds \(M'\) by using the Lefschetz fixed point formula.
These auxiliary manifolds are fixed point sets \(M^{\mathbb{Z}_k}\) of finite cyclic subgroups of \(S^1\).
 Using this fact they can show that the equivariant universal elliptic genus of \(M\) does not have poles. Therefore it may be identified with a  bounded holomorphic function. Hence, it is constant.

We use the fact that we only have to prove the theorem for our generators of the \(S^1\)-equivariant Spin bordism ring.
For semi-free \(S^1\)-manifolds this has been done by Ochanine \cite{MR970284}, by using localization in equivariant cohomology and some elementary complex analysis.
The proof in the semi-free case is simpler than in the non-semi-free case because one sees directly from the fixed point formula that the elliptic genus does not have poles.
Therefore one does not need the auxiliary manifolds and the twisted elliptic genera in this case.

So, we only have to show that an \(S^1\)-equivariant elliptic genus of a generalized Bott manifold, which is Spin, is constant.
We do this by showing that the \(T\)-equivariant elliptic genera of a generalized Bott manifold \(M^{2n}\)  are constant.
The main new technical observation in the proof of this fact is that Ochanine's proof carries over to the situation where the codimension of the \(T\)-fixed point set is minimal.

\bibliography{circle_psc}{}

\providecommand{\bysame}{\leavevmode\hbox to3em{\hrulefill}\thinspace}
\providecommand{\MR}{\relax\ifhmode\unskip\space\fi MR }
\providecommand{\MRhref}[2]{%
  \href{http://www.ams.org/mathscinet-getitem?mr=#1}{#2}
}
\providecommand{\href}[2]{#2}
\begin{thebibliography}{10}

\bibitem{MR0236952}
M.~F. Atiyah and I.~M. Singer, \emph{The index of elliptic operators. {III}},
  Ann. of Math. (2) \textbf{87} (1968), 546--604. \MR{0236952 (38 \#5245)}

\bibitem{MR0278334}
Michael Atiyah and Friedrich Hirzebruch, \emph{Spin-manifolds and group
  actions}, Essays on {T}opology and {R}elated {T}opics ({M}\'emoires
  d\'edi\'es \`a {G}eorges de {R}ham), Springer, New York, 1970, pp.~18--28.
  \MR{0278334 (43 \#4064)}

\bibitem{berard83:_scalar}
L.~B{\'e}rard~Bergery, \emph{Scalar curvature and isometry group}, Spectra of
  Riemannian Manifolds, Kaigai Publications, Tokyo, 1983, pp.~9--28.

\bibitem{MR882700}
Luc{\'{\i}}lia~Daruiz Borsari, \emph{Bordism of semifree circle actions on
  {S}pin manifolds}, Trans. Amer. Math. Soc. \textbf{301} (1987), no.~2,
  479--487. \MR{882700 (88g:57037)}

\bibitem{MR954493}
Raoul Bott and Clifford Taubes, \emph{On the rigidity theorems of {W}itten}, J.
  Amer. Math. Soc. \textbf{2} (1989), no.~1, 137--186. \MR{954493 (89k:58270)}

\bibitem{MR0221518}
Glen~E. Bredon, \emph{A {$\Pi \sb\ast $}-module structure for {$\Theta \sb\ast
  $} and applications to transformation groups}, Ann. of Math. (2) \textbf{86}
  (1967), 434--448. \MR{0221518 (36 \#4570)}

\bibitem{MR2551516}
Suyoung Choi, Mikiya Masuda, and Dong~Youp Suh, \emph{Topological
  classification of generalized {B}ott towers}, Trans. Amer. Math. Soc.
  \textbf{362} (2010), no.~2, 1097--1112. \MR{2551516}

\bibitem{MR0176478}
P.~E. Conner and E.~E. Floyd, \emph{Differentiable periodic maps}, Ergebnisse
  der Mathematik und ihrer Grenzgebiete, N. F., Band 33, Academic Press Inc.,
  Publishers, New York, 1964. \MR{0176478 (31 \#750)}

\bibitem{MR603614}
Allan~L. Edmonds, \emph{Orientability of fixed point sets}, Proc. Amer. Math.
  Soc. \textbf{82} (1981), no.~1, 120--124. \MR{603614 (82h:57032)}

\bibitem{MR1152317}
Carla Farsi, \emph{Equivariant spin bordism in low dimensions}, Topology Appl.
  \textbf{43} (1992), no.~2, 167--180. \MR{1152317 (93f:57044)}

\bibitem{MR962295}
Pawe{\l} Gajer, \emph{Riemannian metrics of positive scalar curvature on
  compact manifolds with boundary}, Ann. Global Anal. Geom. \textbf{5} (1987),
  no.~3, 179--191. \MR{962295 (89m:53061)}

\bibitem{MR2024928}
Alfred Gray, \emph{Tubes}, second ed., Progress in Mathematics, vol. 221,
  Birkh\"auser Verlag, Basel, 2004, With a preface by Vicente Miquel.
  \MR{2024928 (2004j:53001)}

\bibitem{MR577131}
Mikhael Gromov and H.~Blaine Lawson, Jr., \emph{The classification of simply
  connected manifolds of positive scalar curvature}, Ann. of Math. (2)
  \textbf{111} (1980), no.~3, 423--434. \MR{577131 (81h:53036)}

\bibitem{MR1301185}
Michael Grossberg and Yael Karshon, \emph{Bott towers, complete integrability,
  and the extended character of representations}, Duke Math. J. \textbf{76}
  (1994), no.~1, 23--58. \MR{1301185}

\bibitem{MR2376283}
Bernhard Hanke, \emph{Positive scalar curvature with symmetry}, J. Reine Angew.
  Math. \textbf{614} (2008), 73--115. \MR{2376283 (2008m:53084)}

\bibitem{MR0309134}
Akio Hattori and Hajime Taniguchi, \emph{Smooth {$S^{1}$}-action and bordism},
  J. Math. Soc. Japan \textbf{24} (1972), 701--731. \MR{0309134 (46 \#8245)}

\bibitem{MR0461538}
Akio Hattori and Tomoyoshi Yoshida, \emph{Lifting compact group actions in
  fiber bundles}, Japan. J. Math. (N.S.) \textbf{2} (1976), no.~1, 13--25.
  \MR{0461538 (57 \#1523)}

\bibitem{MR632190}
Vappala~J. Joseph, \emph{Smooth actions of the circle group on exotic spheres},
  Pacific J. Math. \textbf{95} (1981), no.~2, 323--336. \MR{632190 (83b:57023)}

\bibitem{MR0286131}
Katsuo Kawakubo and Fuichi Uchida, \emph{On the index of a semi-free
  {$S^{1}$}-action}, J. Math. Soc. Japan \textbf{23} (1971), 351--355.
  \MR{0286131 (44 \#3345)}

\bibitem{MR1407034}
S.~O. Kochman, \emph{Bordism, stable homotopy and {A}dams spectral sequences},
  Fields Institute Monographs, vol.~7, American Mathematical Society,
  Providence, RI, 1996. \MR{1407034 (97i:55017)}

\bibitem{MR0261636}
C.~Kosniowski, \emph{Applications of the holomorphic {L}efschetz formula},
  Bull. London Math. Soc. \textbf{2} (1970), 43--48. \MR{0261636 (41 \#6249)}

\bibitem{MR518871}
\bysame, \emph{Actions of finite abelian groups}, Research Notes in
  Mathematics, vol.~18, Pitman (Advanced Publishing Program), Boston,
  Mass.-London, 1978. \MR{518871}

\bibitem{MR695647}
Czes Kosniowski and Mahgoub Yahia, \emph{Unitary bordism of circle actions},
  Proc. Edinburgh Math. Soc. (2) \textbf{26} (1983), no.~1, 97--105. \MR{695647
  (84e:57032)}

\bibitem{MR0358841}
H.~Blaine Lawson, Jr. and Shing~Tung Yau, \emph{Scalar curvature, non-abelian
  group actions, and the degree of symmetry of exotic spheres}, Comment. Math.
  Helv. \textbf{49} (1974), 232--244. \MR{0358841 (50 \#11300)}

\bibitem{MR1331972}
Kefeng Liu, \emph{On modular invariance and rigidity theorems}, J. Differential
  Geom. \textbf{41} (1995), no.~2, 343--396. \MR{1331972}

\bibitem{MR1758446}
John Lott, \emph{Signatures and higher signatures of {$S^1$}-quotients}, Math.
  Ann. \textbf{316} (2000), no.~4, 617--657. \MR{1758446 (2001f:58048)}

\bibitem{MR970284}
Serge Ochanine, \emph{Genres elliptiques \'equivariants}, Elliptic curves and
  modular forms in algebraic topology ({P}rinceton, {NJ}, 1986), Lecture Notes
  in Math., vol. 1326, Springer, Berlin, 1988, pp.~107--122. \MR{970284}

\bibitem{MR0263093}
E.~Ossa, \emph{Fixpunktfreie {$S^{1}$}-{A}ktionen}, Math. Ann. \textbf{186}
  (1970), 45--52. \MR{0263093 (41 \#7698)}

\bibitem{MR1268192}
Jonathan Rosenberg and Stephan Stolz, \emph{Manifolds of positive scalar
  curvature}, Algebraic topology and its applications, Math. Sci. Res. Inst.
  Publ., vol.~27, Springer, New York, 1994, pp.~241--267. \MR{1268192}

\bibitem{MR982331}
Jonathan Rosenberg and Shmuel Weinberger, \emph{Higher {$G$}-indices and
  applications}, Ann. Sci. \'Ecole Norm. Sup. (4) \textbf{21} (1988), no.~4,
  479--495. \MR{982331 (90f:58170)}

\bibitem{MR1846617}
In{\`e}s Saihi, \emph{Bordisme de {$S^1$}-fibr\'es vectoriels munis de
  structures suppl\'ementaires}, Manuscripta Math. \textbf{105} (2001), no.~2,
  175--199. \MR{1846617 (2002i:57045)}

\bibitem{MR535700}
R.~Schoen and S.~T. Yau, \emph{On the structure of manifolds with positive
  scalar curvature}, Manuscripta Math. \textbf{28} (1979), no.~1-3, 159--183.
  \MR{535700 (80k:53064)}

\bibitem{MR0380853}
Reinhard Schultz, \emph{Circle actions on homotopy spheres not bounding spin
  manifolds}, Trans. Amer. Math. Soc. \textbf{213} (1975), 89--98. \MR{0380853
  (52 \#1750)}

\bibitem{MR2115453}
Dev Sinha, \emph{Bordism of semi-free {$S^1$}-actions}, Math. Z. \textbf{249}
  (2005), no.~2, 439--454. \MR{2115453 (2006h:57027)}

\bibitem{MR1189863}
Stephan Stolz, \emph{Simply connected manifolds of positive scalar curvature},
  Ann. of Math. (2) \textbf{136} (1992), no.~3, 511--540. \MR{1189863
  (93i:57033)}

\bibitem{MR998662}
Clifford~Henry Taubes, \emph{{$S^1$} actions and elliptic genera}, Comm. Math.
  Phys. \textbf{122} (1989), no.~3, 455--526. \MR{998662 (90f:58167)}

\bibitem{MR0278338}
Fuichi Uchida, \emph{Cobordism groups of semi-free {$S^{1}$}- and
  {$S^{3}$}-actions}, Osaka J. Math. \textbf{7} (1970), 345--351. \MR{0278338
  (43 \#4068)}

\bibitem{MR2789750}
Mark Walsh, \emph{Metrics of positive scalar curvature and generalised {M}orse
  functions, {P}art {I}}, Mem. Amer. Math. Soc. \textbf{209} (2011), no.~983,
  xviii+80. \MR{2789750}

\bibitem{MR3355120}
Michael Wiemeler, \emph{Torus manifolds and non-negative curvature}, J. Lond.
  Math. Soc. (2) \textbf{91} (2015), no.~3, 667--692. \MR{3355120}

\bibitem{wiemeler15:_circl}
\bysame, \emph{Circle actions and scalar curvature}, Trans. Amer. Math. Soc.
  \textbf{368} (2016), no.~4, 2939--2966. \MR{3449263}

\bibitem{MR830167}
Edward Witten, \emph{Fermion quantum numbers in {K}aluza-{K}lein theory},
  Shelter {I}sland {II} ({S}helter {I}sland, {N}.{Y}., 1983), MIT Press,
  Cambridge, MA, 1985, pp.~227--277. \MR{830167 (87f:81144)}

\end{thebibliography}
\bibliographystyle{amsplain}
\end{document}